\documentclass[12pt]{amsart}

\usepackage[left=2cm, right=2cm, top=3cm, bottom=3cm]{geometry}
\usepackage{hyperref}
\usepackage{amsmath, amsthm, amsfonts, amssymb}
\usepackage{mathrsfs}
\usepackage{mathtools}
\usepackage{upgreek}
\usepackage{longtable}
\usepackage{booktabs}

\hypersetup{
  colorlinks=true,
  pdftitle={Dimension formulas for modular form spaces of rational weights, the classification of eta-quotient characters and an extension of Martin's theorem},
  pdfauthor={Xiao-Jie Zhu},
  pdfsubject={Mathematics},
  pdfkeywords={modular form, dimension formula, rational weight, Dedekind eta function, Hecke operator, multiplicative eta-quotient}
}

\numberwithin{equation}{section}

\makeatletter
\@namedef{subjclassname@2020}{%
  \textup{2020} Mathematics Subject Classification}
\makeatother

\theoremstyle{plain}
\newtheorem{thm}{Theorem}[section]
\newtheorem*{thm*}{A special case of the Main Theorem}{}
\newtheorem{coro}[thm]{Corollary}
\newtheorem{prop}[thm]{Proposition}

\newtheorem{lemm}[thm]{Lemma}

\theoremstyle{definition}
\newtheorem{deff}[thm]{Definition}
\newtheorem{examp}[thm]{Example}
\newtheorem*{examp*}{Example}

\theoremstyle{remark}
\newtheorem{rema}[thm]{Remark}

\newcommand\twoscript[2]{\substack{{#1} \\ {#2}}}

\newcommand\rmi{\mathrm{i}}
\newcommand\rme{\mathrm{e}}

\newcommand\legendre[2]{\genfrac{(}{)}{}{}{#1}{#2}}
\newcommand\tbtmat[4]{\left(\begin{smallmatrix}{#1} & {#2} \\ {#3} & {#4}\end{smallmatrix}\right)}
\newcommand\tbtMat[4]{\begin{pmatrix}{#1} & {#2} \\ {#3} & {#4}\end{pmatrix}}
\newcommand*\abs[1]{\lvert#1\rvert}
\newcommand\etp[1]{\mathfrak{e}\left(#1\right)}

\newcommand\diff{\,\mathrm{d}}

\newcommand\ord{\mathop{\mathrm{ord}}}
\renewcommand\div{\mathop{\mathrm{div}}}

\newcommand\rad{\mathop{\mathrm{rad}}}
\newcommand\radE{\mathop{\mathrm{rad}_E}}
\newcommand\radO{\mathop{\mathrm{rad}_O}}
\newcommand\radp{\mathop{\mathrm{rad}'}}
\newcommand\irad{\mathop{\mathrm{irad}}}
\newcommand\iradp{\mathop{\mathrm{irad}'}}

\newcommand\numZ{\mathbb{Z}}
\newcommand\numQ{\mathbb{Q}}
\newcommand\numR{\mathbb{R}}
\newcommand\numC{\mathbb{C}}
\newcommand\projQ{\mathbb{P}^1(\mathbb{Q})}
\newcommand\numgeq[2]{\mathbb{#1}_{\geq #2}}

\newcommand\glR{\mathrm{GL}_2(\mathbb{R})}

\newcommand\slZ{\mathrm{SL}_2(\mathbb{Z})}

\newcommand\slR{\mathrm{SL}_2(\mathbb{R})}

\newcommand\glpQ{\mathrm{GL}_2^{+}(\mathbb{Q})}
\newcommand\glpR{\mathrm{GL}_2^{+}(\mathbb{R})}

\newcommand\uhp{\mathfrak{H}}

\newcommand\eleglptRaDp[6]{\left(\left(\begin{smallmatrix}{#1} & {#2} \\ {#3} & {#4}\end{smallmatrix}\right),{#5}\left({#3}^\prime \tau+{#4}^\prime \right)^{\frac{1}{#6}}\right)}
\newcommand\elesltRaDs[5]{\left(\left(\begin{smallmatrix}{#1} & {#2} \\ {#3} & {#4}\end{smallmatrix}\right),{#5}\right)}

\newcommand\Dcover[2]{\widetilde{#1^{#2}}}

\begin{document}


\baselineskip=17pt


\title[Dimension formulas for modular form spaces]{Dimension formulas for modular form spaces of rational weights, the classification of eta-quotient characters and an extension of Martin's theorem}

\author{Xiao-Jie Zhu}
\address{School of Mathematical Sciences\\
Key Laboratory of MEA (Ministry of Education) \& Shanghai Key Laboratory of PMMP\\
East China Normal University\\
500 Dongchuan Road, 200241\\
Shanghai, P. R. China}
\email{zhuxiaojiemath@outlook.com}
\urladdr{https://orcid.org/0000-0002-6733-0755}

\begin{abstract}
We give an explicit formula for dimensions of spaces of rational-weight modular forms whose multiplier systems are induced by eta-quotients of fractional exponents. As the first application, we give series expressions of Fourier coefficients of the $n$th root of certain infinite $q$-products. As the second application, we extend Yves Martin's list of multiplicative holomorphic eta-quotients of integral weights by first extending the meaning of multiplicativity, then identifying one-dimensional spaces, and finally applying Wohlfahrt's extension of Hecke operators. A table containing $2277$ of such eta-quotients is presented. As a related result, we completely classify the multiplier systems induced by eta-quotients of integral exponents. For instance, there are totally $384$ such multiplier systems on $\Gamma_0(4)$ for any fixed weight. There are also some new results on $n$-fold covers of modular groups for $n\geq3$. Finally, we provide SageMath programs for verifying the theorems and generating the tables.
\end{abstract}



\thanks{This work is supported in part by Science and Technology Commission of Shanghai Municipality (No. 22DZ2229014).}

\maketitle

\tableofcontents

\section{Introduction}
\label{sec:Introduction}
\subsection{Dimension formulas}
A holomorphic function $f$ defined on the upper half plane $\uhp$ is called a modular form if it satisfies two conditions: first $f\left(\frac{a\tau+b}{c\tau+d}\right)=(c\tau+d)^k\chi\tbtmat{a}{b}{c}{d}f(\tau),\,\tau\in\uhp$ for all $\tbtmat{a}{b}{c}{d}$ in some discrete matrix group $G$, $k$ being an integer or half integer, called the weight, and $\chi$ being a multiplier system of $G$; second $f$ is holomorphic at all cusps of $G$, the meaning of which will be clarified in Definition \ref{deff:modularForm}. Let the space of all modular forms with given $G$, $k$ and $\chi$ be denoted by $M_k(G,\chi)$.

Modular forms, spaces of modular forms and their variants play important roles in mathematics. In the study of these topics, having an explicit formula for the dimension of the complex vector space $M_k(G,\chi)$ is crucial to some applications. For instance, $\dim_\numC M_k(G,\chi)=1$ would imply that any function $f$ in $M_k(G,\chi)$ is an eigenfunction for a family of operators called the Hecke operators and hence we could obtain nontrivial relations among the Fourier coefficients of $f$.

A natural and elegant way for computing $\dim_\numC M_k(G,\chi)$ is to use the Riemann-Roch theorem and the Riemann-Hurwitz formula. Petersson \cite[p. 194]{Pet38} applied this method and gave a dimension formula which he called the generalized Riemann-Roch theorem. Petersson's formula concerns arbitrary Fuchsian group $G$ of the first kind, arbitrary complex weight $k$ and arbitrary multiplier system $\chi$, but this formula is not so explicit in the sense that one cannot directly compute the dimension using elementary operations and numerical information about $G$ and $\chi$. To our best knowledge, Shimura \cite[Section 2.4]{Shi71} gave explicit formulas for $G$ arbitrary, $k$ being an integer and $\chi$ trivial first. Another way for computing $\dim_\numC M_k(G,\chi)$ is to use the Eichler-Selberg trace formula for Hecke operators; cf. \cite{Zag77_1,Zag77_2} for the case $G$ being the full modular group, $\chi$ trivial and \cite{Coh77} for the case $G=\Gamma_0(N)$ and $\chi$ induced by a Dirichlet character modulo $N$, both dealing with integral weights. See also \cite{Hij74}. For the half-integral weights, Cohen and Oesterl\'e \cite{CO77} gave an explicit dimension formula for the case $G=\Gamma_0(4N)$, $\chi$ being the product of the multiplier system of the theta series $\sum_{n\in\numZ}\rme^{2\uppi\rmi n^2\tau}$ and that of a Dirichlet character modulo $4N$. Since then, there appear many excellent works on computing $\dim_\numC M_k(G,\chi)$ for certain special $G$ or special $\chi$, or computing the dimensions of certain subspaces; cf. \cite{Mar05,Que10, KP17,Mar18,ZZ22,Ros26}.

Modular forms of rational weight have received comparatively less systematic attention. We know little about them. For instance, no comparably developed Hecke theory seems to be available in this generality at present. However, there actually exist many such forms and more importantly, there are interesting applications, e.g., to noncongruence modular forms. Explicit dimension formulas for spaces of rational weights have been obtained by Ibukiyama \cite[Lemma 1.7]{Ibu00}, \cite[p. 5]{Ibu20} in some special cases, using essentially the method of Petersson \cite{Pet38}, but no proof is given.

As our main result, we give explicit formulas for $\dim_\numC M_k(\Gamma_0(N),\chi)$ where $k$ is any rational number, $\chi$ is the multiplier system of any level $N$ eta-quotient $\prod_{n\mid N}\eta(n\tau)^{r_n}$ of fractional exponents $r_n$. See Definition \ref{deff:modularForm} for the exact meaning of $M_k(\Gamma_0(N),\chi)$ and \eqref{eq:charEtaQuotient} for that of $\chi$. The proof is based on Petersson's method \cite{Pet38} and we provide the full details. Petersson’s theorem applies in a much broader complex-weight setting, but the present work restricts to rational weights and to multiplier systems induced by eta-quotients, where the $D$-cover formalism (see \eqref{eq:Dcover}) allows one to obtain fully explicit formulas.

The following theorem is a special case of Theorem \ref{thm:main}, which is our main theorem.
\begin{thm*}
If $k>2-\frac{6}{m}\varepsilon_2-\frac{8}{m}\varepsilon_3-\frac{12}{m}\sum_{c\mid N}\phi(c,N/c)\cdot\left(1-\left\{\frac{x_c}{24}\right\}\right)$, then
\begin{equation}
\label{eq:IntroMain}
\dim_\numC M_k(\Gamma_0(N), \chi)=\frac{k-1}{12}m+\frac{1}{4}\varepsilon_2+\frac{1}{3}\varepsilon_3+\sum_{c\mid N}\phi(c,N/c)\cdot\left(\frac{1}{2}-\left\{\frac{x_c}{24}\right\}\right),
\end{equation}
where $m$, $\varepsilon_2$ and $\varepsilon_3$, depending on $N$, are given in \eqref{eq:mGamma0N}, \eqref{eq:e2Gamma0N} and \eqref{eq:e3Gamma0N} respectively and
\begin{equation*}
k=\frac{1}{2}\sum_{n\mid N}r_n,\quad x_c=\sum_{n\mid N}\frac{N}{(N,c^2)}\cdot\frac{(n,c)^2}{n}r_n\quad\text{  for }c\mid N.
\end{equation*}
Note that the notation $\phi(c,N/c)$ refers to the Euler totient $\phi$ of the greatest common divisor of $c$ and $N/c$.
\end{thm*}

In particular, if $k\geq2$, then \eqref{eq:IntroMain} always holds. For any rational $k$ with $0<k<2$, there are also infinitely many pairs $(N,\chi)$ such that $\dim_\numC M_k(\Gamma_0(N), \chi)$ can be computed using \eqref{eq:IntroMain}. See Section \ref{sec:Dimension formulas for weight 1/2, 1, 3/2 and small levels} for complete lists of such pairs $(N,\chi)$ for $k=1/2,\,1,\,3/2$ and $r_n$ being integers. The dimensions for $k=1/2$ and $k=3/2$ in the case $G=\Gamma_0(4N)$ and $\chi$ is the product of the multiplier system of the theta series $\sum_{n\in\numZ}\rme^{2\uppi\rmi n^2\tau}$ and that of a Dirichlet character modulo $4N$ are previously known: Serre and Starks \cite{SS77} gave explicit bases of the spaces of weight $1/2$; Cohen and Oesterl\'e \cite{CO77} established relations between spaces of weight $1/2$ and those of weight $3/2$. The important and interesting problem of computing dimensions in weight $1$ is difficult and the most complete result about this was obtained by Deligne and Serre \cite{DS74}. Using their theory, one can compute $\dim_\numC M_1(\Gamma_0(N), \chi)$ for any $N\in\numgeq{Z}{1}$ and $\chi$ induced by any Dirichlet character modulo $N$ although there is no explicit formula.

The difference between our formulas for $\dim_\numC M_k(\Gamma_0(N), \chi)$ with $k\in\frac{1}{2}\numZ$ and others' is that we deal with those $\chi$ induced by eta-quotients while the existing formulas are about those $\chi$ induced by Dirichlet characters or the multiplier system of $\sum_{n\in\numZ}\rme^{2\uppi\rmi n^2\tau}$. Moreover, when $k\in\numQ\setminus\frac{1}{2}\numZ$, our formulas include the one in \cite[Lemma 1.7]{Ibu00} and \cite[p. 5]{Ibu20} as special cases.

We use the standard construction of compact modular curves, the (classical) Riemann-Roch theorem, the Riemann-Hurwitz formula, and eta-transformation formulas as background. The new point here is the explicit evaluation of the resulting divisor-theoretic formula for eta-quotient multiplier systems of rational weight.

We note that the dimensions of spaces of modular forms of \emph{real} weight $k>2$ with unitary multiplier systems can be computed using the Selberg trace formula \cite[p. 84]{Sel56} as well.\footnote{A more general (and more explicit) form of Selberg's formula, commonly called the Eichler-Selberg trace formula, was obtained by Eichler \cite{Eic57}.} (In contrast, Petersson's divisor-theoretic method applies to complex weights, but is valid under a restriction on the degree of the corresponding divisor; see \eqref{eq:RiemannRochSpecial}.) See for instance \cite[p. 481]{Hej83}, \cite[Theorem 2.5.4]{Fis87} or \cite[Theorem 12.4.11]{CS17} for various versions of dimension formulas derived from the Selberg trace formula. Fischer's version actually applies to arbitrary unitary multiplier systems and to vector-valued modular forms. More importantly, \cite[Theorem 2.5.4]{Fis87} indeed provides dimension formulas in real weight $k>2$ since by \cite[Proposition 1.3.6]{Fis87} there exists a unitary multiplier system for every real weight. However, like Petersson's formula, this kind of general Selberg trace formula is not explicit. Finally, note that the $k\geq2$ part of Theorem \ref{thm:main} can also be proved using the Selberg trace formula, but it is beyond our ability to determine whether the Selberg trace formula can be used to derive the $0<k<2$ part of Theorem \ref{thm:main}.

After preparing, stating, and proving the main theorem and establishing some related tools (Proposition \ref{prop:rnxc}, Theorem \ref{thm:SNrepr}, Tables \ref{table:wt1/2}, \ref{table:wt1} and Section \ref{subsec:wt3/2}), we will give mainly two applications, both of which are based on identifying one-dimensional spaces. The first application concerns rational-weight modular forms and the second concerns ordinary eta-quotients whose exponents $r_n$ are all integers.

\subsection{Application I}
\label{subsec:AppI}
We give a series expression of the Fourier coefficients of the $n$th root of certain infinite $q$-products. See Corollary \ref{coro:etaQuotientCoefficient} for the details. To avoid overlapping with the main text, here we give a randomly chosen example which is different from Examples \ref{examp:App1Examp1} and \ref{examp:N4rn274}.
\begin{prop}
\label{prop:inIntroExamp}
We have
\begin{multline*}
\frac{\sqrt[3]{\prod_{n\in\numgeq{Z}{1}}(1-q^n)}\cdot\sqrt{\prod_{n\in\numgeq{Z}{1}}(1-q^{2n})^{29}}}{\sqrt[3]{\prod_{n\in\numgeq{Z}{1}}(1-q^{4n})^{22}}}=1+\rme^{-\frac{15\uppi\rmi}{8}}\frac{(2\uppi)^{\frac{15}{4}}}{\Gamma(\frac{15}{4})}\\
\cdot\sum_{n\in\numgeq{Z}{1}}n^{\frac{11}{4}}\left(\sum_{c\in\numgeq{Z}{1}}\frac{1}{(4c)^{\frac{15}{4}}}\sum_{\twoscript{0\leq d<4c}{(d,4c)=1}}\etp{\frac{dn}{4c}-\frac{1}{24}P(4c,d;\frac{29}{2},-\frac{22}{3})}\right)\cdot q^n,
\end{multline*}
where $q=\rme^{2\uppi\rmi\tau}$ throughout the paper, $\tau\in\uhp$ and
\begin{gather*}
P(4c,d;\frac{29}{2},-\frac{22}{3})=-88s(-d,c)+174s(-d,2c)+4s(-d,4c)-\frac{45}{2},\\
s(-d,c)=\sum_{r=1}^{c-1}\frac{r}{c}\left(\left\{\frac{-dr}{c}\right\}-\frac{1}{2}\right).
\end{gather*}
\end{prop}

Corollary \ref{coro:etaQuotientCoefficient} contains infinitely many such identities of level $4$, extending a previous result of the author \cite[Theorem 7.1]{Zhu23} which concerns levels $2$ and $3$. For the outline of the proof, say of Proposition \ref{prop:inIntroExamp}, let the infinite product be $f$, which is an eta-quotient of fractional exponents. Then $f\in M_{15/4}(\Gamma_0(4),\chi)$ where $\chi$ is the multiplier system of $f$. There is as well an Eisenstein series $g$ in $M_{15/4}(\Gamma_0(4),\chi)$; see Definition \ref{deff:Eis}. Now we apply Theorem \ref{thm:main} and find that $\dim_\numC M_{15/4}(\Gamma_0(4),\chi)=1$. Hence $f$ and $g$ are proportional. The identity thus follows by figuring out the Fourier coefficients of $g$.

\subsection{Application II}
\label{subsec:AppII}
This is the major application. In 1996, Martin \cite{Mar96} obtained the complete list\footnote{There is a misprint in \cite[p. 4853]{Mar96}. The entry $2^{-1}\cdot4^4\cdot6^{-1}\cdot8^{-1}\cdot12^{4}\cdot24$ should be corrected to $2^{-1}\cdot4^4\cdot6^{-1}\cdot8^{-1}\cdot12^{4}\cdot24^{-1}$.} of integral-weight holomorphic eta-quotients $f$ (with integral exponents) satisfying that
\begin{enumerate}
  \item the multiplier system of $f$ is induced by some Dirichlet character,
  \item if we write $f=\sum_{n\in\numgeq{Z}{0}}c_f(n)q^n$, then $c_f(n_1)c_f(n_2)=c_f(n_1n_2)c_f(1)$ for any coprime positive integers.
\end{enumerate}
We will discard condition (a) and seek for arbitrary holomorphic eta-quotients (with integral exponents) satisfying a multiplicativity property similar to the one in condition (b). We will also give many interesting new identities involving Fourier coefficients of these eta-quotients. Let us first present some randomly chosen examples. For the basic knowledge about Dedekind eta function $\eta(\tau)$ and eta-quotients, see the second half of Section \ref{sec:Modular forms of rational weight}.
\begin{examp*}
Let $f(\tau)=\eta(\tau)^{-7}\eta(2\tau)^{17}\eta(4\tau)^{-3}=\sum_{n\in\frac{5}{8}+\numgeq{Z}{0}}c_f(n)q^n$. Let $l_1,l_2$ be coprime odd square-free integers. Then
\begin{equation*}
c_f\left(\frac{5l_1^2}{8}\right)c_f\left(\frac{5l_2^2}{8}\right)=c_f\left(\frac{5l_1^2l_2^2}{8}\right).
\end{equation*}
See Proposition \ref{prop:semiMultiExampLevel4} for details and the proof.
\end{examp*}

\begin{examp*}
Let $f(\tau)=\eta(3\tau)^2\eta(9\tau)^{-1}\eta(27\tau)=\sum_{n\in\numgeq{Z}{1}}c_f(n)q^n$. Then $c_f(l_1)c_f(l_2)=c_f(l_1l_2)$ whenever $l_1,l_2\equiv1\bmod{3}$ are square-free positive integers with $(l_1,l_2)=1$. See Example \ref{examp:level27Multi} for details.
\end{examp*}

\begin{examp*}
Let $f(\tau)=\eta(\tau)^{1}\eta(2\tau)^{-1}\eta(3\tau)^{-1}\eta(4\tau)^{1}\eta(6\tau)^4\eta(12\tau)^{-2}=\sum_{n\in\numgeq{Z}{0}}c_f(n)q^n$. Then
\begin{equation*}
c_f(l)=-\sum_{a\mid l}\legendre{a}{6}=-\prod_{p\mid l}\left(1+\legendre{p}{6}\right)
\end{equation*}
where $1\leq l\equiv1\bmod{24}$ is square-free. Note that $\legendre{a}{6}$ refers to the Kronecker-Jacobi symbol. See Example \ref{examp:level12} for details.
\end{examp*}

Martin \cite{Mar96} proved his results using Hecke operators in the sense of \cite{AL70} while we prove our results using Wohlfahrt's extension of Hecke operators \cite{Woh57}. See also \cite[Section 3]{ZZ23} for a theory of the double coset version of these operators. These operators are denoted by $T_l$, where $l$ runs through a multiplicative submonoid $L$ of the positive integers; see Corollary \ref{coro:Lfmf}. As a prerequisite, we give an explicit formula for the action on Fourier coefficients; see Theorem \ref{thm:cTlfbycf}. Then we aim to find holomorphic eta-quotient $f$ such that $T_lf=c_l\cdot f$ with $c_l\in\numC$. At this point Theorem \ref{thm:main} enters into play. If $f$, whose level is $N$ and character is $\chi$, lies in a one-dimensional space $M_k(\Gamma_0(N),\chi)$ or $S_k(\Gamma_0(N),\chi)$, the subspace of cusp forms, then we will have $T_lf=c_l\cdot f$ for all $l\in L$. ($L$ depends on $f$.) This fact, which is stated in Theorem \ref{thm:mainApp2}, is the main theorem of this application.

It seems that there are more than 10000 eta-quotients that are Hecke eigenforms (in the sense of \eqref{eq:mainApp2}) for infinitely many $l$. We list 2277 of them in Table \ref{table:admissibleTypeI}. They are what we call admissible eta-quotients of type I with levels
$$
N=1,2,3,4,5,6,7,8,9,10,11,13,14,15,17,19,21,27.
$$
For the meaning of admissible eta-quotients of type I, see the beginning part of Section \ref{sec:an extension of Martin's list of multiplicative eta-quotients}. For other levels, either there is no admissible eta-quotient of type I, or there are too many so it is not appropriate to list them in the paper. One can find the SageMath code that generates admissible eta-quotients of type I or II with any given level in Appendix \ref{apx:Usage of SageMath code}.

Besides Theorem \ref{thm:mainApp2}, another important result is Theorem \ref{thm:multiplicativeEtaQuotients}, which says that
\begin{equation*}
c_{l_1}\cdot c_{l_2}=c_{l_1l_2},\quad l_1,l_2\in L,\quad(l_1,l_2)=1.
\end{equation*}
Here $c_l=T_lf/f$. This is the genuine multiplicativity property that is satisfied by all admissible eta-quotients and reduces to the relation $c_f(n_1)c_f(n_2)=c_f(n_1n_2)c_f(1)$ for ordinary multiplicative eta-quotients in \cite{Mar96}.

All identities in the above three examples are immediate consequences of Theorems \ref{thm:mainApp2} and \ref{thm:multiplicativeEtaQuotients}.

\subsection{Other results, structure of the paper, and notations}
As a related result, we give the complete classification of linear characters that are induced by eta-quotients $\prod_{n\mid N}\eta(n\tau)^{r_n}$ with $r_n\in\numZ$ on the double cover of $\Gamma_0(N)$. See Corollary \ref{coro:AllEtaQuoChars} for the conclusion and Examples \ref{examp:classifyChar1}, \ref{examp:classifyChar2}, \ref{examp:classifyChar3}, \ref{examp:Nis4palpha}, \ref{examp:classifyChar5} for examples. For instance, for each $k\in\frac{1}{2}\numZ$, there are totally $384$ linear characters of the double cover of $\Gamma_0(4)$ that are induced by $\eta(\tau)^{r_1}\eta(2\tau)^{r_2}\eta(4\tau)^{r_n}$ of weight $k$. As a comparison, there are only two characters of $\Gamma_0(4)$ that are induced by Dirichlet characters modulo $4$ and these two characters are also contained in the above $384$ characters with any fixed $k\in\numZ$. For a description of \emph{all} linear unitary characters of $\Gamma_0(4)$, cf. \cite[Main Theorem]{Zhu25}.

The structure of the paper is as follows. We review some elements of rational-weight modular forms in Section \ref{sec:Modular forms of rational weight}. The second half of this section contains elements of Dedekind eta function and eta-quotients of fractional exponents. In Section \ref{sec:Divisors of modular forms} we associate a divisor on certain compact Riemann surface with any rational-weight meromorphic modular form and in addition, we give a detailed proof of the valence formula in the case of rational weights. In Section \ref{sec:Dimension formulas} we state and prove the main theorem via Petersson's method \cite{Pet38}. The following three sections contain some tools which are needed by the following two applications: Section \ref{sec:Order-character relations} contains a formula that relates the orders at cusps and the exponents (or the character) of an eta-quotient. This formula is due to Bhattacharya (cf. \cite[Eq. (5.13)]{Bha17}). Here we give more details on the proof. In Section \ref{sec:The classification of characters induced by eta-quotients} we classify all linear characters that are induced by eta-quotients $\prod_{n\mid N}\eta(n\tau)^{r_n}$ with $r_n\in\numZ$ on the double cover of $\Gamma_0(N)$, as is described in the last paragraph. Section \ref{sec:Dimension formulas for weight 1/2, 1, 3/2 and small levels}, which is subdivided into three subsections, contains tables of dimensions of modular form spaces of weight $1/2$, $1$ and $3/2$ that can be computed by our main theorem. In Sections \ref{sec:one-dimensional spaces and Eisenstein series of rational weights} and \ref{sec:an extension of Martin's list of multiplicative eta-quotients} we carry out the first and second applications, which has been described in Sections \ref{subsec:AppI} and \ref{subsec:AppII}, respectively. Section \ref{sec:Miscellaneous observations and open questions} contains some comments, open problems and conjectures. In Appendix \ref{apx:covers}, we establish some fundamental facts about the $D$-cover \eqref{eq:Dcover}. Finally, many formulas in this paper have been verified and many tables are generated by SageMath \cite{Sage} programs. Appendix \ref{apx:Usage of SageMath code} contains the usage of the code.

We collect some notations. For a set $A$, the symbols $\abs{A}$ and $\#A$ both denote its cardinality. The notation $\numQ^A$ denotes the $\numQ$-vector space of functions from $A$ into $\numQ$. A multiset $\{x_i\colon i\in I\}$ is defined to be the ordinary set $\{(y,n)\colon n=\#\{i\in I\colon x_i=y\},\,n\neq0\}$. Its underlying (ordinary) set is still denoted by $\{x_i\colon i\in I\}$. If $f\colon X\rightarrow Y$ is a function, then $f\vert_A$ is its restriction to $A$ where $A$ is a subset of $X$. If $f,g$ are functions then $f\circ g$ is their composition: $f\circ g(x)=f(g(x))$. For a family of sets $A_i$, $i\in I$, the notation $\bigsqcup_{i\in I}A_i$ denotes the union and only when the family $A_i$ are disjoint can one use this notation.

For a group $G$, $\langle X\rangle$ denotes the subgroup generated by $X$ where $X$ is a subset of $G$. If $H$ is a subgroup of $G$ then $[G\colon H]=\#G/H=\#H\backslash G$ is the index. If $G$ acts on the left (respectively right) on a set $X$, then $G\backslash X$ (respectively $X/G$) denotes the set of orbits $G\cdot x$ (respectively $x\cdot G$) where $x\in X$. If X is itself a group, $G$ is a subgroup of $X$ and the action is the group operation, then $G\backslash X$ (respectively $X/G$) is called the left (respectively right) coset space. For $x\in X$, $G_x$ denotes the stabilizer, that is, the subgroup of $g\in G$ such that $gx=x$ (or $xg=x$). A character $\chi$ of $G$ is a complex linear character, that is, a group homomorphism from $G$ to $\numC^\times$, the multiplicative group of nonzero complex numbers. If $\abs{\chi(g)}=1$ for all $g\in G$, then $\chi$ is called unitary. If $G'$ is another group and $\pi\colon G\rightarrow G'$ is a group surjection, then we say the character $\chi$ descends to a character on $G'$ if there is a character $\chi'$ of $G'$ such that $\chi=\chi'\circ\pi$.

The symbol $\phi$ refers to the Euler totient function. For integers $a, b,\dots$, the notation $(a,b,\dots)=\gcd(a,b,\dots)$ denotes the greatest common divisor and $\phi(a,b,\dots)$ denotes the value of $\phi$ at $(a,b,\dots)$. Let $N\in\numgeq{Z}{1}$; the summation range of $\sum_{n\mid N}$ is implicitly understood to be the \emph{positive} divisors of $N$. Let $p$ be a prime and $\alpha\in\numgeq{Z}{0}$; then $p^\alpha\parallel N$ means $p^\alpha\mid N$ but $p^{\alpha+1}\nmid N$. The $p$-adic exponential valuation, $v_p(N)$, is the largest $\alpha$ such that $p^\alpha\parallel N$. An empty sum is understood to be $0$ and an empty product be $1$. For a real number $x$, the notation $[x]$ means the largest integer not exceeding $x$ and $\{x\}=x-[x]$. The functions $\Gamma(s)$ and $\zeta(s)$ are the usual Euler Gamma function and Riemann zeta function, respectively.

Let $m,n\in\numZ$; the Kronecker-Jacobi symbol $\legendre{m}{n}$ is defined as follows:
\begin{itemize}
\item $\legendre{m}{p}$ is the usual Legendre symbol if $p$ is an odd prime.
\item $\legendre{m}{2}$ equals $0$ if $2 \mid m$, and equals $(-1)^{(m^2-1)/8}$ if $2 \nmid m$.
\item $\legendre{m}{-1}$ equals $1$ if $m \geq 0$, and equals $-1$ otherwise.
\item $\legendre{m}{1}=1$ by convention.
\item $\legendre{m}{n}$ is defined to make it a complete multiplicative function of $n \in \numZ-\{0\}$.
\item $\legendre{m}{0}=0$ if $m \neq \pm1$, and $\legendre{\pm1}{0}=1$.
\end{itemize}
We shall freely use the following properties, especially in the proof of Theorem \ref{thm:cTlfbycf}.
\begin{itemize}
\item $\legendre{m}{n_1n_2}=\legendre{m}{n_1}\legendre{m}{n_2}$. When $m=-1$, it is required that $n_1,n_2\neq0$.
\item $\legendre{m_1m_2}{n}=\legendre{m_1}{n}\legendre{m_2}{n}$. When $n=-1$, it is required that either $m_1,m_2\neq0$, or one of $m_1$, $m_2$ is $0$ and the other is nonnegative.
\item $\legendre{m}{n}\legendre{n}{m}=\varepsilon(m,n)\cdot(-1)^{\frac{n-1}{2}\cdot\frac{m-1}{2}}$ where $m,n$ are coprime odd integers, $\varepsilon(m,n)=-1$ if $n,m<0$ and $\varepsilon(m,n)=1$ otherwise.
\item $\legendre{-1}{n}=(-1)^{\frac{n-1}{2}}$ and $\legendre{2}{n}=(-1)^{\frac{n^2-1}{8}}$ where $n$ is odd.
\item The function $n\mapsto\legendre{m}{n}$ is $\abs{m}$-periodic if $m\equiv0,1\bmod{4}$; it is $\abs{4m}$-periodic if $m\equiv2\bmod{4}$.
\item The function $m\mapsto\legendre{m}{n}$ is $n$-periodic if $n$ is odd and positive.
\end{itemize}
For the proofs, see \cite[Section 2.2.2]{Coh07}.

\section{Modular forms of rational weight}
\label{sec:Modular forms of rational weight}
In this section, we review some elements of the theory of modular forms of rational weight (cf. \cite{Zhu23, Ibu20, Aok17, KM12, KM03, Ibu00, Pet38}). The concept of modular forms of rational weight is a special case of that of generalized modular forms which was initiated by Knopp and Mason \cite{KM03}. The multiplier system of a generalized modular form is not required to be unitary or to have finite order, while in this paper, the multiplier systems of all modular forms that occur are unitary and have finite orders. In addition, it seems that the modular forms of rational weights that occur in this paper can only be regarded as modular forms on non-congruence subgroups of $\slZ$ with trivial character. As a comparison, Freitag and Hill \cite{FH23} have recently constructed modular forms of weight $1/3$ on $\mathrm{SU}(2,1)$ whose levels are certain congruence subgroups of $\mathrm{SU}(2,1)$. Furthermore, see \cite[Section 2]{LL23} for harmonic weak Maass forms of real weight and their relations with weakly holomorphic modular forms.

First some notations. Let $\glpR$ be the group of all $2\times 2$ real matrices with positive determinants and $\slR$ be the subgroup of $\glpR$ whose elements have determinant $1$. Let $D$ be a positive integer. We shall investigate modular forms of weight $k\in\frac{1}{D}\numZ$, where the automorphic factor $(c\tau+d)^{k}$ appears. This is not single-valued unless we keep track of a chosen $D$th root. Motivated by this, we define the $D$-cover of $\glpR$ by
\begin{equation}
\label{eq:Dcover}
\Dcover{\glpR}{D}=\left\{\eleglptRaDp{a}{b}{c}{d}{\varepsilon}{D} \colon \tbtmat{a}{b}{c}{d} \in \glpR,\, \varepsilon^D=1\right\},
\end{equation}
where $\tbtmat{a'}{b'}{c'}{d'}$ is the matrix in $\slR$ proportional to $\tbtmat{a}{b}{c}{d}$ and $\left(c^\prime \tau+d^\prime \right)^{\frac{1}{D}}$ means a function of $\tau\in\uhp=\{z\in\numC\colon \Im(z)>0\}$. The notation $\eleglptRaDp{a}{b}{c}{d}{\varepsilon}{D}$ is sometimes abbreviated to $\elesltRaDs{a}{b}{c}{d}{\varepsilon}$ when $D$ is understood. The composition is given by
\begin{multline}
\label{eq:multplicationCover}
\eleglptRaDp{a_1}{b_1}{c_1}{d_1}{\varepsilon_1}{D}\cdot\eleglptRaDp{a_2}{b_2}{c_2}{d_2}{\varepsilon_2}{D} \\
= \left(\left(\begin{smallmatrix} a_1 & b_1 \\ c_1 & d_1\end{smallmatrix}\right)\left(\begin{smallmatrix} a_2 & b_2 \\ c_2 & d_2\end{smallmatrix}\right),\varepsilon_1\varepsilon_2\left(c_1^\prime\frac{a_2^\prime\tau+b_2^\prime}{c_2^\prime\tau+d_2^\prime}+d_1^\prime\right)^\frac{1}{D}\left(c_2^\prime\tau+d_2^\prime\right)^\frac{1}{D}\right).
\end{multline}
It can be verified straightforwardly that $\Dcover{\glpR}{D}$ with this composition is a group. Another less direct, but more natural way to define $\Dcover{\glpR}{D}$ is to realize it as a semi-direct product of $\glR$ and the group of $D$th roots of unity; cf. \cite[Lemma 1.1.3]{CS17}. One should note that different choices of holomorphic branches of $\left(c^\prime \tau+d^\prime \right)^{\frac{1}{D}}$ lead to different $D$-covers. We choose the following branch throughout:
\begin{equation}
\label{eq:branchRoot}
z^{r}=\exp\left(r\log z\right),\quad -\uppi < \Im(\log z) \leq \uppi.
\end{equation}
It is possible to make $\Dcover{\glpR}{D}$ a topological covering of $\glpR$ and hence $\Dcover{\glpR}{D}$ becomes a Lie group whose smooth structure is transposed from $\glpR$ using the above covering map; see Appendix \ref{apx:covers}. Using $D$-covers when dealing with modular forms of weight $k\in\frac{1}{D}\numZ$ has the advantage that the multiplier systems become group characters on $D$-covers.

Let $G$ be a subgroup of $\glpR$; then by $\Dcover{G}{D}$ we understand the preimage of $G$ under the natural projection $\Dcover{\glpR}{D}\rightarrow\glpR$, $\elesltRaDs{a}{b}{c}{d}{\varepsilon}\mapsto\tbtmat{a}{b}{c}{d}$. For $\gamma\in\glpR$, let $\widetilde{\gamma}$ denote $(\gamma,1)\in\Dcover{\glpR}{D}$. (The cover index $D$ should be inferred from the context.) In addition, we set $\overline{G}=\left\{g\cdot\{\pm \tbtmat{1}{0}{0}{1}\}\colon g\in G\right\}$, which is a subgroup of $\overline{\glpR}=\glpR/\{\pm \tbtmat{1}{0}{0}{1}\}$.

The group $\slZ$ is the set of $2\times 2$ integral matrices of determinant $1$ and is known as the \emph{full modular group}. We also need the congruence subgroup
\begin{equation*}
\Gamma_0(N)=\left\{\tbtmat{a}{b}{c}{d}\in\slZ\colon c \equiv 0 \bmod N\right\},
\end{equation*}
where $N$ is a positive integer called the level. It is well known that the index $[\slZ\colon\Gamma_0(N)]=N\prod_{p \mid N}(1+1/p)$ where $p$ denotes a prime (cf. \cite[Coro. 6.2.13]{CS17}). The matrices $\tbtmat{1}{0}{0}{1}$, $\tbtmat{1}{1}{0}{1}$, $\tbtmat{0}{-1}{1}{0}$ and $\tbtmat{N}{0}{0}{1}$ are denoted by $I$, $T$, $S$, $B_N$ respectively.

Let $\mathscr{M}(\uhp)$ denote the field of meromorphic functions on the upper half plane $\uhp$ and $k\in\frac{1}{D}\numZ$. Define the \emph{slash operator} of weight $k$ by
\begin{equation*}
f\vert_k\elesltRaDs{a}{b}{c}{d}{\varepsilon}(\tau)=\left(\varepsilon\sqrt[D]{c^\prime\tau+d^\prime}\right)^{-Dk}f\left(\frac{a\tau+b}{c\tau+d}\right),
\end{equation*}
where $f\in\mathscr{M}(\uhp)$, $\elesltRaDs{a}{b}{c}{d}{\varepsilon}\in\Dcover{\glpR}{D}$ and $\tau\in\uhp$. It is immediate that $\Dcover{\glpR}{D}$ acts on $\mathscr{M}(\uhp)$ on the right via the slash operator of weight $k$. Let $G$ be a subgroup of $\slZ$ and $\chi$ be a character\footnote{A character always means a unitary linear character, that is, a group homomorphism to the group of complex numbers of absolute value $1$.} on $\Dcover{G}{D}$. We say $f$ transforms like a modular form of weight $k$ and with multiplier system (or with character) $\chi$ if $f\vert_k\gamma=\chi(\gamma)f$ for any $\gamma\in\Dcover{G}{D}$. If this is the case, and $[\slZ\colon G]<+\infty$ and $\chi(\gamma)$ has finite order for any $\gamma\in\Dcover{G}{D}$, then for $\tbtmat{a}{b}{c}{d}\in\slZ$ we have
\begin{equation*}
f\vert_k\widetilde{\tbtmat{a}{b}{c}{d}}(\tau)=\sum_{n\in m^{-1}\numZ}c_nq^n,
\end{equation*}
where $q=\etp{\tau}=\exp(2\uppi\rmi\tau)$ and $m\in\numgeq{Z}{1}$ provided that there is $Y_0>0$ (depending on $\tbtmat{a}{b}{c}{d}$) such that $f\vert_k\widetilde{\tbtmat{a}{b}{c}{d}}$ has no poles on $\uhp_{Y_0}=\{\tau\in\uhp\colon\Im(\tau)>Y_0\}$. Moreover, the series converges normally on $\uhp_{Y_0}$. To see this, note that there exists $m\in\numgeq{Z}{1}$ such that $\chi\left(\widetilde{\tbtmat{a}{b}{c}{d}}\widetilde{T}^m\widetilde{\tbtmat{a}{b}{c}{d}}^{-1}\right)=1$ and hence $f\vert_k\widetilde{\tbtmat{a}{b}{c}{d}}(\tau+m)=f\vert_k\widetilde{\tbtmat{a}{b}{c}{d}}(\tau)$. Therefore, the desired expansion follows from Fourier's theorem or Laurent's theorem. Define the order $\ord_{\rmi\infty}(f\vert_k(\tbtmat{a}{b}{c}{d},\varepsilon))$ to be the least $n$ such that $c_n \neq 0$. If the expansion does not exist (that is, $f\vert_k(\tbtmat{a}{b}{c}{d},\varepsilon)$ has a nonisolated singularity at infinity) then the order is undefined; if the expansion holds but for any $n_0\in\numZ$ there exists $n<n_0$ such that $c_n \neq 0$ then define $\ord_{\rmi\infty}(f\vert_k(\tbtmat{a}{b}{c}{d},\varepsilon))=-\infty$; if $f\vert_k(\tbtmat{a}{b}{c}{d},\varepsilon)$ is identically zero then define $\ord_{\rmi\infty}(f\vert_k(\tbtmat{a}{b}{c}{d},\varepsilon))=+\infty$.

\begin{deff}
\label{deff:modularForm}
Let $D\in\numgeq{Z}{1}$ and $k\in\frac{1}{D}\numZ$. Let $G$ be a subgroup of $\slZ$ of finite index and $\chi\colon\Dcover{G}{D}\rightarrow\numC^\times$ be a character of finite order\footnote{If $\chi$ is not of finite order as considered by Knopp and Mason \cite{KM03}, this definition also makes sense.} (that is, $\chi(\gamma)$ has finite order for any $\gamma\in\Dcover{G}{D}$). Let $f\in\mathscr{M}(\uhp)$. Then we say $f$ is a \emph{meromorphic modular form} of weight $k$ for $G$ with character $\chi$ if $f\vert_k\gamma=\chi(\gamma)f$ for any $\gamma\in\Dcover{G}{D}$ and $\ord_{\rmi\infty}(f\vert_k\gamma)>-\infty$ for any $\gamma\in\Dcover{\slZ}{D}$. Suppose this is the case. Then we say $f$ is a \emph{weakly holomorphic modular form} if it has no poles on $\uhp$; it is a \emph{modular function} if the weight $k=0$; it is a \emph{modular form} (\emph{cusp form} respectively) if it is weakly holomorphic and $\ord_{\rmi\infty}(f\vert_k\gamma)\geq 0$ ($>0$ respectively) for any $\gamma\in\Dcover{\slZ}{D}$. The vector spaces (over $\numC$) of modular forms and of cusp forms are denoted by $M_k(G,\chi)$ and $S_k(G,\chi)$ respectively. When $\chi$ is trivial (which means $k\in2\numZ$ and $\chi$ always takes the value $1$), we let $M_k(G)=M_k(G,\chi)$ and $S_k(G)=S_k(G,\chi)$.
\end{deff}

Note that the cover index $D$ can be recovered from the domain of $\chi$ and $M_k(G,\chi)$ is also denoted by $M_k(\Dcover{G}{D},\chi)$ when needed.

\begin{rema}
\label{rema:mulSysProp}
Suppose there exists a \emph{nonzero} meromorphic modular form $f$ of weight $k$ for $G$ with character $\chi$. Then we have the following facts:
\begin{itemize}
  \item $\chi\left(I,\etp{1/D}\right)=\etp{-k}$.
  \item If $-I\in G$, then $\chi\left(\widetilde{-I}\right)=\etp{-k/2}$.
  \item If $D'$ is another positive integer with $k\in\frac{1}{D'}\numZ$ and $D'\mid D$, then $\chi\colon\Dcover{G}{D}\rightarrow\numC^\times$ descends to a character on $\Dcover{G}{D'}$, i.e., $\chi$ can be factored as $\chi=\chi_1p_{D,D'}$ where $p_{D,D'}\colon\Dcover{G}{D}\rightarrow\Dcover{G}{D'}$ is the natural projection $(\gamma,\varepsilon)\mapsto(\gamma,\varepsilon^{D/D'})$ and $\chi_1$ is a character on $\Dcover{G}{D'}$ (which is unique).
\end{itemize}
The proof is straightforward by considering slash operators acting on $f$ of which we omit the details. Motivated by the first two facts, we call a character $\chi$ on $\Dcover{G}{D}$ with the properties $\chi\left(I,\etp{1/D}\right)=\etp{-k}$ and $\chi\left(\widetilde{-I}\right)=\etp{-k/2}$ if $-I\in G$ a \emph{multiplier system} for $G$ of weight $k$ of cover index $D$. Note that if $-I\in G$, then the second property implies the first one. Finally, note that the third fact still holds when we require that $\chi$ is a multiplier system for $G$ of weight $k$ of cover index $D$ even if we can \emph{not} find a nonzero form $f$.
\end{rema}

\begin{rema}
\label{rema:multiplicationModularForm}
If $f$ is a meromorphic modular form of weight $k$ for $G$ with character $\chi$, and $g$ is a meromorphic modular form of weight $k'$ for $G$ with character $\chi'$ with $k,k'\in\frac{1}{D}\numZ$, then it is immediate that $f\cdot g$ is a meromorphic modular form of weight $k+k'$ for $G$ with character $\chi\cdot\chi'$. Moreover, if $f$ and $g$ are both modular forms (cusp forms respectively), then so is $f\cdot g$. 
\end{rema}

One of the aims of this paper is to give explicit formulas concerning $\dim_\numC M_k(\Gamma_0(N),\chi)$ and $\dim_\numC S_k(\Gamma_0(N),\chi)$ where $k\in\numQ$ and $\chi$ is a character on $\Dcover{\Gamma_0(N)}{D}$ of some kind which we explain now.

We need the Dedekind eta function and its logarithm, which are defined by
\begin{align}
\eta(\tau)&=q^{1/24}\prod_{n\in\numgeq{Z}{1}}\left(1-q^n\right),\quad\tau\in\uhp, \notag\\
\log\eta(\tau)&=\log\eta(\rmi)+\int_{\rmi}^{\tau}\frac{\eta'(z)}{\eta(z)}\diff z, \label{eq:logeta}
\end{align}
where $\log\eta(\rmi)$ is the real logarithm. The transformation equations of $\log\eta$ under $\slZ$ are obtained by Dedekind (cf. \cite[Equation (12), Section 3.4]{Apo90}):
\begin{equation}
\label{eq:transLogEta}
\log\eta\left(\frac{a\tau+b}{c\tau+d}\right)-\log\eta(\tau)=2\uppi\rmi\left(\frac{a+d}{24c}+\frac{1}{2}s(-d,c)-\frac{1}{8}\right)+\frac{1}{2}\log(c\tau+d),
\end{equation}
where $\tbtmat{a}{b}{c}{d} \in \slZ$ with $c>0$ and $s(-d,c)$ is the Dedekind sum
\begin{equation}
\label{eq:DedekindSum}
s(h,k)=\sum_{r \bmod k}\Big(\Big(\frac{r}{k}\Big)\Big)\cdot\Big(\Big(\frac{hr}{k}\Big)\Big),\quad h\in\numZ,\,k\in\numgeq{Z}{1}.
\end{equation}
In the above definition, $((x))=x-[x]-1/2$ if $x\in\numR\setminus\numZ$ and $((x))=0$ if $x\in\numZ$.
Set
\begin{align}
\label{eq:Psi}
\Psi\colon \slZ &\rightarrow \numZ \\
\tbtmat{a}{b}{c}{d} &\mapsto \begin{dcases}
\frac{a+d}{c}+12s(-d,c)-3,   & c>0; \\
\frac{a+d}{c}+12s(d,-c)+3,   & c<0; \\
b,                           & c=0, a>0; \\
-b-6,                        & c=0, a<0. 
\end{dcases}\notag
\end{align}
Fix a positive integer $D$. Then the transformation equations of $\eta$ follow from that of $\log\eta$ and can be expressed as $\eta\vert_{1/2}(\gamma,\varepsilon)=\chi_{\eta}(\gamma,\varepsilon)\eta$, where $\chi_{\eta}$ is the multiplier system for $\slZ$ of weight $1/2$ of cover index $2D$ defined by
\begin{equation*}
\chi_{\eta}(\gamma,\varepsilon)=\varepsilon^{-D}\etp{\frac{\Psi(\gamma)}{24}},\quad(\gamma,\varepsilon)\in\Dcover{\slZ}{2D}.
\end{equation*}
The fact that $\Psi(\slZ)\subseteq\numZ$ (when we give \eqref{eq:Psi} this is tacitly assumed) follows from the above formula with $D=1$, the fact $\slZ$ is generated by $T$ and $S$ and the fact $\chi_{\eta}(\widetilde{T})^{24}=\chi_{\eta}(\widetilde{S})^{24}=1$.

Let $N,D\in\numgeq{Z}{1}$. By an eta-quotient of level $N$ and cover index $D$ we understand a product $f(\tau)=\prod_{0<n\mid N}\eta(n\tau)^{r_n}$ with $r_n\in\numQ$ such that $D\cdot r_n\in2\numZ$ for any $n\mid N$. Thus an eta-quotient of cover index $2$ is just an ordinary eta-quotient whose exponents are integers. The fractional powers are defined by $\eta(n\tau)^{r_n}=\exp\left(r_n\log\eta(n\tau)\right)$ (see \eqref{eq:logeta}). Since $\eta$ has no poles on $\uhp$, the eta-quotient $f$ is a weakly holomorphic modular form of weight $k=\frac{1}{2}\sum_{n\mid N}r_n$ for $\Dcover{\Gamma_0(N)}{D}$ with character
\begin{align}
\chi\colon\Dcover{\Gamma_0(N)}{D} &\rightarrow\numC^\times \label{eq:charEtaQuotient}\\
\elesltRaDs{a}{b}{c}{d}{\varepsilon} &\mapsto \varepsilon^{-Dk}\etp{\frac{1}{24}\sum_{n \mid N}r_n\Psi\tbtmat{a}{bn}{c/n}{d}}.\notag
\end{align}
The above formula can be proved using \eqref{eq:transLogEta}, Remark \ref{rema:multiplicationModularForm} and the fact $\eta(n\tau)=n^{-1/4}\eta\vert_{1/2}\widetilde{B_n}(\tau)$. Note that $\eta$ has no zeros on $\uhp$ according to its infinite product expansion. Therefore any eta-quotient has no zeros or poles on $\uhp$. For the order at infinity, we have (cf. \cite[Lemma 4.2]{Zhu23})
\begin{equation}
\label{eq:orderEtaQuotient}
\ord\nolimits_{\rmi\infty}(f\vert_k\widetilde{\tbtmat{a}{b}{c}{d}})=\frac{1}{24}\sum_{n \mid N}\frac{(n,c)^2}{n}r_n,\quad\tbtmat{a}{b}{c}{d}\in\slZ.
\end{equation}
We recommend the reader to see \cite[Section 2]{Bha21} or \cite[Section 5]{Bha17} for more details on ordinary eta-quotients. If \eqref{eq:orderEtaQuotient} are always nonnegative for all $\tbtmat{a}{b}{c}{d}\in\slZ$, we say $f$ is a \emph{holomorphic eta-quotient}.

We will investigate the dimensions $\dim_\numC M_k(\Gamma_0(N),\chi)$ and $\dim_\numC S_k(\Gamma_0(N),\chi)$ where $k\in\frac{1}{D}\numZ$ and $\chi$ is the character of an eta-quotient of weight $k$, level $N$ and cover index $D$ in the next three sections.

\begin{rema}
The reader should be warned that when dealing with half-integral weights, our notation may differ from other authors', e.g. from the notation encountered in \cite{Shi73} and \cite{SS77}. Let $\psi$ be a Dirichlet character modulo $4N$ and let $2k\in\numgeq{Z}{1}$. The space $M_{0}(4N, k, \psi)$ in Serre and Stark's notation is the same thing as $M_{k}(\Gamma_0(4N), \chi_1\psi_1)$ in our notation where $\chi_1$ is the character of $\eta(\tau)^{-4k}\eta(2\tau)^{10k}\eta(4\tau)^{-4k}$ and $\psi_1$ is the character that maps $\elesltRaDs{a}{b}{c}{d}{\varepsilon}$ to $\psi(d)$. The character $\chi_1$ is hidden in many authors' notation of spaces and it sometimes appears in their definitions of slash operators or modular transformations or automorphic factors.
\end{rema}

\section{Divisors of modular forms}
\label{sec:Divisors of modular forms}
Our method relies on the classical Riemann-Roch theorem, so we recall some elements here. By a \emph{compact Riemann surface} $X$, we understand a compact connected Hausdorff topological space equipped with a maximal atlas $\{(U_i,\phi_i)\}$ (i.e., each $U_i$ is open in $X$ and $\phi_i$ is a homeomorphism from $U_i$ onto an open set in $\numC$) such that when $(U_i,\phi_i)$ and $(U_j,\phi_j)$ overlap, the transition map $\phi_j\circ\phi_i^{-1}\colon\phi_i(U_i\cap U_j)\rightarrow\phi_j(U_i\cap U_j)$ is holomorphic. It can be shown that $X$ is second-countable (cf. \cite[p. 88]{Fri11}). Therefore, according to the well known classification of compact connected topological surfaces (cf. e.g. \cite[Theorem 10.22]{Lee11}) and the fact that Riemann surfaces are oriented, $X$ is homeomorphic to the $2$-sphere $\mathbb{S}^2$ or to the connected sum of $g$ copies of the torus $\mathbb{T}^2$ with $g\in\numgeq{Z}{1}$. The number $g$ is called the \emph{genus} of $X$. (If $X$ is a sphere, then its genus is defined to be $0$.)

A \emph{meromorphic differential} $\omega$ on $X$ is by definition a holomorphic differential on $X\setminus S$ where $S$ is a finite subset of $X$ such that if $(U,\phi\colon U\rightarrow\phi(U)\subseteq\numC)$ is any chart of $X$, then $\omega$ is of the form $f(z)\diff z$ on this chart where $f(z)$ is a meromorphic function on $\phi(U)$ whose poles are exactly $\phi(U\cap S)$. The basic feature of a meromorphic differential is the transformation equation between different charts: if $(V,\psi)$ is another chart that overlaps $(U,\phi)$ and $\omega=g(w)\diff w$ on $(V,\psi)$, then we have
\begin{equation*}
f(z)=g(\psi\circ\phi^{-1}(z))\frac{\diff \psi\circ\phi^{-1}(z)}{\diff z},\quad z\in\phi(U\cap V).
\end{equation*}
The set of all meromorphic differentials on $X$ is denoted by $\mathscr{K}(X)$ which is obviously a complex vector space. Let $\mathscr{M}(X)$ denote the complex vector space of all meromorphic functions on $X$ as in Section \ref{sec:Modular forms of rational weight}. If there exists a nonzero $\omega_0\in\mathscr{K}(X)$, then the map $\mathscr{M}(X)\rightarrow\mathscr{K}(X)$ that sends $f$ to $f\cdot\omega_0$ is a $\numC$-linear isomorphism. Such $\omega_0$ actually exists according to a fundamental result in the theory of Riemann surfaces (cf. \cite[Theorem 1.10, Chapter IV]{Fri11}). As a consequence, nonzero meromorphic functions always exist on any compact Riemann surface.

By a ($\numQ$-valued) divisor on $X$, we understand a function $X\rightarrow\numQ$ with finite support. We often write a divisor $D$ as a formal sum $\sum_{x\in X}d_x\cdot(x)$ where $d_x=D(x)$. The term ``divisor'' traditionally refers to a $\numZ$-valued divisor but in this paper to a $\numQ$-valued divisor. The sets of all $\numZ$-valued divisors and of all $\numQ$-valued divisors are denoted by $\mathscr{D}(X)$ and $\mathscr{D}_\numQ(X)$ respectively. Equipped with the pointwise addition, they are abelian groups. We define the \emph{degree} of $D\in\mathscr{D}_\numQ(X)$ by $\deg(D)=\sum_{x\in X}D(x)$, which is well-defined since the sum is actually a finite sum. Moreover, define the floor function as follows: for $D\in\mathscr{D}_\numQ(X)$, set $[D]=\sum_{x\in X}[d_x]\cdot(x)$, where $[d_x]$ is the greatest integer that does not exceed $d_x$. Finally, for two divisors $D_1$ and $D_2$, we say $D_1\geq D_2$ if $D_1(x) \geq D_2(x)$ far any $x\in X$, and we say $D_1 > D_2$ if $D_1 \geq D_2$ but $D_1\neq D_2$.

Generally speaking, with any meromorphic section of a holomorphic line bundle over $X$ one can associate a $\numZ$-valued divisor. Specifically, if $f\in\mathscr{M}(X)$, then define $\div(f)=\sum_{x\in X}d_x\cdot(x)$ with $d_x$ being the least exponent in the Laurent expansion of $f$ at $x$ in any chart. For $\omega\in\mathscr{K}(X)$ we define $\div(\omega)$ in a similar manner.

The compact Riemann surfaces we need are the \emph{modular curves} $X_G$ where $G$ is a finite index subgroup of $\slZ$ which we describe now. Let $\uhp^{*}=\uhp\cup\projQ$ where $\projQ$ is the projective line over $\numQ$ which can be identified with $\numQ$ together with a point $\rmi\infty=1/0$. The group $\slZ$ acts on $\projQ$ on the left via $\tbtmat{a}{b}{c}{d}\cdot(p/q)=\frac{ap+bq}{cp+dq}$. We endow $\uhp^{*}$ with the topology generated by the usual topology of $\uhp$ and the sets
\begin{equation*}
\{\gamma\tau\colon \tau\in\uhp_{Y_0}\cup\{\rmi\infty\}\},\quad\gamma\in\slZ,\,Y_0>0,
\end{equation*}
where $\uhp_{Y_0}=\{\tau\in\uhp\colon\Im(\tau)>Y_0\}$. It is immediate that $\uhp^{*}$ with this topology is a Hausdorff space\footnote{The space $\uhp^{*}$ is not locally compact and hence can not be a topological manifold.}, $\uhp$ is an open subset of $\uhp^{*}$ and $G$ acts on $\uhp^{*}$ (on the left) by homeomorphisms. Let $G\backslash\uhp^{*}$ be the orbit space (endowed with the quotient topology) and $p\colon\uhp^{*}\rightarrow G\backslash\uhp^{*}$ be the quotient map. It can be shown that $G\backslash\uhp^{*}$ is Hausdorff (cf. e.g. \cite[Lemma 1.7.7]{Miy06}) and when $[\slZ\colon G]<+\infty$ $G\backslash\uhp^{*}$ is compact (cf. e.g. \cite[Proposition 14.6, Chapter IV]{Fri11}). Hereafter we always assume that $[\slZ\colon G]<+\infty$. Write $G\backslash\uhp^{*}=G\backslash\uhp\cup G\backslash\projQ$, a disjoint union. Then $G\backslash\uhp$ is open in $G\backslash\uhp^{*}$ and $G\backslash\projQ$ is a finite set since $\slZ$ acts transitively on $\projQ$ and hence $\abs{G\backslash\projQ}\leq[\slZ\colon G]$. 
The orbits in $G\backslash\projQ$ are called \emph{cusps} of $G\backslash\uhp^{*}$ and by abuse of language, the points in $\projQ$ are also called cusps. The \emph{width} of a cusp $G\cdot x\in G\backslash\projQ$ is defined by $w_{Gx}=w_x=[\overline{\slZ}_{x}\colon\overline{G}_x]$ which is independent of the choice of the representative $x$. We have (cf. \cite[Proposition 6.3.8(b)]{CS17})
\begin{equation*}
[\overline{\slZ}\colon\overline{G}]=\sum_{Gx\in G\backslash\projQ}w_{Gx}.
\end{equation*}
We also need the notion of \emph{elliptic points}. It is well known that the stabilizer of each point in $\{\gamma\rmi\colon\gamma\in\slZ\}$ under the action of $\slZ/\{\pm I\}$ is a cyclic group of order $2$, the stabilizer of each point in $\{\gamma\rho\colon\gamma\in\slZ\}$ is a cyclic group of order $3$ where $\rho=\frac{-1+\sqrt{3}\rmi}{2}$ and the stabilizer of any other point in $\uhp$ is trivial. Thus for $\tau\in\uhp$, $\abs{\overline{G}_\tau}=1$, $2$ or $3$ since $\abs{\overline{G}_\tau}$ divides $\abs{\overline{\slZ}_\tau}$. If $\abs{\overline{G}_\tau}=2$ ($3$ respectively) then we call $p(\tau)$ an elliptic point for $G\backslash\uhp^{*}$ of period $2$ ($3$ respectively). By abuse of language, we also call the number $\tau$ an elliptic point if $p(\tau)$ is.

The description and verification of the atlas on $G\backslash\uhp^{*}$ that turns it into a compact Riemann surface are rather technical and tedious but it is basic and well known (cf. e.g. \cite[Section 1.8]{Miy06}). Let $X_G$ denote the resulting compact Riemann surface. We will mainly use $X_{\Gamma_0(N)}$ hereafter so set $X_0(N)=X_{\Gamma_0(N)}$ for simplicity. Let $g$ be the genus of $X_G$. Then by applying the Riemann-Hurwitz ramification formula to the holomorphic map $X_G\rightarrow X_{\slZ}$, $G\cdot z \mapsto \slZ\cdot z$, noticing that the degree of this map equals $m=[\overline{\slZ}\colon\overline{G}]$ and the genus of $X_{\slZ}$ is $0$, we obtain (cf. e.g. \cite[Theorem 3.1.1]{DS05} for the case $G$ is a congruence subgroup)
\begin{equation}
\label{eq:genusGamma0N}
g=1+\frac{m}{12}-\frac{\varepsilon_2}{4}-\frac{\varepsilon_3}{3}-\frac{\varepsilon_\infty}{2},
\end{equation}
where $\varepsilon_2$ and $\varepsilon_3$ are the numbers of elliptic points of $X_G$ of period $2$ and $3$ respectively and $\varepsilon_\infty$ is the number of cusps of $X_G$. More formally,
\begin{align*}
\varepsilon_2&=\#\{G\cdot z\in G\backslash\slZ\rmi\colon\abs{\overline{G}_{z}}=2\},\\
\varepsilon_3&=\#\{G\cdot z\in G\backslash\slZ\rho\colon\abs{\overline{G}_{z}}=3\},\quad\rho=\frac{-1+\sqrt{3}\rmi}{2},\\
\varepsilon_\infty&=\#G\backslash\projQ.
\end{align*}
For $G=\Gamma_0(N)$, it is known that
\begin{align}
m&=[\overline{\slZ}\colon\overline{\Gamma_0(N)}]=[\slZ\colon\Gamma_0(N)]=N\prod_{p\mid N}\left(1+\frac{1}{p}\right),\label{eq:mGamma0N}\\
\varepsilon_2&=\begin{dcases}
\prod_{p \mid N}\left(1+\legendre{-4}{p}\right) & \text{ if } 4\nmid N,\\
0 & \text{ if } 4\mid N,
\end{dcases}\label{eq:e2Gamma0N}\\
\varepsilon_3&=\begin{dcases}
\prod_{p \mid N}\left(1+\legendre{-3}{p}\right) & \text{ if } 9\nmid N,\\
0 & \text{ if } 9\mid N,
\end{dcases}\label{eq:e3Gamma0N}\\
\varepsilon_\infty&=\sum_{d\mid N}\phi(d,N/d),\label{eq:einftyGamma0N}
\end{align}
where $p$ denotes a prime, $\phi(d,N/d)=\phi(\gcd(d,N/d))$, and $\phi$ is the Euler function. For proofs, see \cite[Corollary 6.2.13, Corollary 6.3.24(b)]{CS17} and \cite[Corollary 3.7.2]{DS05}.

Now let us associate a $\numQ$-valued divisor to any meromorphic modular form of rational weight. For integral or half-integral weights, this association is the same as the ordinary one\footnote{But it is different from the one described in \cite[p. 299]{Fri11}} (cf. \cite[Eq. (3.2) and (3.3)]{DS05}). Our treatment has the feature that one need not distinguish between regular and irregular cusps.
\begin{deff}
\label{deff:divisorModularForm}
Let $G$ be a subgroup of $\slZ$ of finite index, $D\in\numgeq{Z}{1}$ and $k\in\frac{1}{D}\numZ$. Let $\chi$ be a multiplier system for $G$ of weight $k$ of cover index $D$ (not assumed to be of finite order). Let $f$ be a \emph{nonzero} meromorphic modular form of weight $k$ for $G$ with character $\chi$. We define its divisor $\div(f)=\sum_{x\in X_G}\div_x(f)\cdot(x)\in\mathscr{D}_\numQ(X_G)$ as follows:
\begin{enumerate}
  \item If $\tau\in\uhp$, then set $\div_{G\tau}(f)=n_0/\abs{\overline{G}_\tau}$ where $n_0$ is the integer such that $f(z)=\sum_{n\geq n_0}c_n(z-\tau)^n$, $c_{n_0}\neq 0$.
  \item If $x\in\projQ$ and $\gamma\in\slZ$ satisfying $\gamma(\rmi\infty)=x$, then set $\div_{Gx}(f)=\ord_{\rmi\infty}(f\vert_k\widetilde{\gamma})\cdot w_{Gx}$.
\end{enumerate}
\end{deff}
One can verify that $\div(f)$ is well-defined, which is not totally trivial but straightforward. The usefulness of this notion relies on the following two simple facts:
\begin{itemize}
  \item If $g$ is another nonzero meromorphic modular form of weight $k'\in\frac{1}{D}\numZ$ for $G$ with character $\chi'$, then we have $\div(fg)=\div(f)+\div(g)$ in $\mathscr{D}_\numQ(X_G)$ (see Remark \ref{rema:multiplicationModularForm}).
  \item If $k=0$ and $\chi$ is the trivial character, then $f$ descends to a meromorphic function $\widetilde{f}$ on $X_G$. We have $\div(f)=\div(\widetilde{f})$ where $\div(\widetilde{f})$ is the usual divisor of a meromorphic function.
\end{itemize}
Note that the factors $1/\abs{\overline{G}_\tau}$ and $w_{Gx}$ in Definition \ref{deff:divisorModularForm} are chosen to let the latter fact hold (to see this one must dive into the atlas of $X_G$ which we have omitted). Also note that the former fact holds because the factors $1/\abs{\overline{G}_\tau}$ and $w_{Gx}$ remain unchanged for all $k$.

Similarly, we have $\div(f/g)=\div(f)-\div(g)$ and in particular $\div(1/g)=-\div(g)$.

A fundamental theorem on modular forms of rational weight is the following one, which in cases of integral or half-integral weights is sometimes called the \emph{valence formula}.
\begin{thm}
\label{thm:valence}
Let us use the notation of Definition \ref{deff:divisorModularForm} and set $m=[\overline{\slZ}\colon\overline{G}]$. Then we have
\begin{equation*}
\deg(\div f)=\frac{1}{12}mk.
\end{equation*}
\end{thm}
\begin{proof}
Suppose that $\overline{\slZ}=\bigcup_{1\leq j\leq m}\overline{G}\overline{\gamma_j}$ is a disjoint union where $\gamma_j\in\slZ$ and $\overline{\gamma_j}=\gamma_j\cdot\{\pm I\}$. Set $g=\prod_{1 \leq j \leq m}f\vert_k\widetilde{\gamma_j}$. Then $g\vert_{mk}\gamma=c_\gamma g$ for any $\gamma\in\Dcover{\slZ}{D}$ where $c_\gamma\in\numC^\times$. Since $g$ is not identically zero, the map $\gamma\mapsto c_\gamma$ is a linear character on $\Dcover{\slZ}{D}$. According to a presentation of $\Dcover{\slZ}{D}$ (cf. \cite[Lemma 5.2]{Zhu23}) the order of this character divides $12D$ and hence $g^{12D}$ transforms like a modular form of weight $12Dmk$ for $\slZ$ with trivial character. Since $f$ is meromorphic at cusps, so is $g^{12D}$. Therefore $g^{12D}$ is a meromorphic modular form of weight $12Dmk$ for $\slZ$ with trivial character. Applying the usual valence formula for even weights (cf. \cite[Theorem 5.6.1]{CS17}) we obtain $\deg(\div g^{12D})=Dmk$. It remains to prove that $12D\deg(\div f)=\deg(\div g^{12D})$, which is equivalent to $\deg(\div f)=\deg(\div g)$. We now prove a stronger assertion, i.e.
\begin{equation}
\label{eq:valenceToProveDivisor}
\div\nolimits_{x}(g)=\sum_{y\in p^{-1}(x)}\div\nolimits_{y}(f)
\end{equation}
for any $x\in X_{\slZ}$, where $p\colon X_G\rightarrow X_{\slZ}$ is the natural projection $G\cdot \tau \mapsto \slZ\cdot \tau$. In the case that $x\in \slZ\backslash\uhp$, we write $x=\slZ\cdot\tau$ and set $w_\tau=[\overline{\slZ}_\tau\colon\overline{G}_\tau]$ (similar to the notion of the width of a cusp in the group-theoretical aspect). Notice that $\ord_\tau f\vert_k\widetilde{\gamma_j}=\ord_{\gamma_j\tau}f$ where $\ord_\tau f$ is the integer $n_0$ such that $\lim_{z\to\tau}f(z)(z-\tau)^{-n_0}$ is nonzero. Thus we have
\begin{equation*}
\div\nolimits_x(g)=\abs{\overline{\slZ}_\tau}^{-1}\cdot\sum_{1 \leq j \leq m}\ord\nolimits_{\gamma_j\tau} f.
\end{equation*}
The underlying set of the multiset $\{G\cdot\gamma_j\tau\colon 1 \leq j \leq m\}$ equals $p^{-1}(x)$ and the multiplicity of each element $G\cdot\gamma_j\tau$ equals $w_{\gamma_j\tau}$. It follows that
\begin{equation*}
\div\nolimits_x(g)=\abs{\overline{\slZ}_\tau}^{-1}\cdot\sum_{y=G\cdot\tau \in p^{-1}(x)}w_\tau\cdot\ord\nolimits_{\tau} f=\sum_{y\in p^{-1}(x)}\div\nolimits_{y}(f),
\end{equation*}
i.e., the desired assertion \eqref{eq:valenceToProveDivisor} holds. In the other case $x\in \slZ\backslash\projQ$, we have $x=\slZ\cdot\rmi\infty$ and its width is $1$. Therefore
\begin{equation*}
\div\nolimits_x(g)=\sum_{1 \leq j \leq m}\ord\nolimits_{\rmi\infty}(f\vert_k\widetilde{\gamma_j}).
\end{equation*}
The underlying set of the multiset $\{G\cdot\gamma_j\rmi\infty\colon 1 \leq j \leq m\}$ equals $p^{-1}(x)$ and the multiplicity of each element $G\cdot\gamma_j\rmi\infty$ equals $w_{\gamma_j\rmi\infty}$. It follows that
\begin{equation*}
\div\nolimits_x(g)=\sum_{y=G\cdot\gamma\rmi\infty\in p^{-1}(x)}w_{\gamma\rmi\infty}\ord\nolimits_{\rmi\infty}(f\vert_k\widetilde{\gamma})=\sum_{y\in p^{-1}(x)}\div\nolimits_{y}(f),
\end{equation*}
i.e., the desired assertion \eqref{eq:valenceToProveDivisor} holds. This concludes the proof.
\end{proof}

\begin{rema}
The above theorem holds even when $\chi$ is non-unitary and the proof remains unchanged. For instance, one can apply it to the generalized modular forms constructed in \cite{KM03}.
\end{rema}

In the rest we need to know information about the divisor of an eta-quotient times a product of Eisenstein series on $\slZ$. We let
\begin{align*}
E_4(\tau)&=1+240\sum_{n\in\numgeq{Z}{1}}\sigma_3(n)q^n,\\
E_6(\tau)&=1-504\sum_{n\in\numgeq{Z}{1}}\sigma_5(n)q^n,
\end{align*}
where $\sigma_k(n)=\sum_{d\mid n}d^k$. It is well known that $E_4\in M_4({\slZ})$ and $E_6\in M_6({\slZ})$.
\begin{prop}
\label{prop:etaEisenstein}
Let $N\in\numgeq{Z}{1}$, $r_n\in\numQ$ for any $n\mid N$ and $t_1,t_2\in\numZ$. Let $\chi$ be the character \eqref{eq:charEtaQuotient} where $D\in\numgeq{Z}{1}$ is chosen so that $Dr_n\in2\numZ$ for any $n \mid N$. Set
\begin{equation*}
f(\tau)=\prod_{n \mid N}\eta(n\tau)^{r_n}\cdot E_4(\tau)^{t_1}\cdot E_6(\tau)^{t_2}.
\end{equation*}
Then $f$ is a nonzero meromorphic modular form of weight $k=\frac{1}{2}\sum_{n}r_n+4t_1+6t_2$ for $\Gamma_0(N)$ with character $\chi$ and
\begin{equation*}
\deg\left([\div f]\right)=\sum_{c\mid N}\phi(c,N/c)\cdot\left[\frac{N}{24(N,c^2)}\sum_{n \mid N}\frac{(n,c)^2}{n}r_n\right]+\left(\frac{t_1}{3}+\frac{t_2}{2}\right)m-\left\{\frac{t_2}{2}\right\}\varepsilon_2-\left\{\frac{t_1}{3}\right\}\varepsilon_3,
\end{equation*}
where $m,\varepsilon_2$ and $\varepsilon_3$ are given in \eqref{eq:mGamma0N}, \eqref{eq:e2Gamma0N} and \eqref{eq:e3Gamma0N}.
\end{prop}
\begin{proof}
The fact that $f$ is a nonzero meromorphic modular form follows from Remark \ref{rema:multiplicationModularForm}. It remains to compute $\deg\left([\div f]\right)$. The divisors of $E_4$ and $E_6$ on $X_{\Gamma_0(N)}$ are known (cf. e.g. \cite[Proposition 5.6.5]{CS17}) as follows:
\begin{equation*}
\div\nolimits_x(E_u)=\begin{dcases}
1/e_\tau & \text{ if } x=\Gamma_0(N)\cdot\tau\in\Gamma_0(N)\backslash\slZ\cdot\tau_0,\\
0 & \text{ otherwise},
\end{dcases}
\end{equation*}
where $u=4$ or $6$, $e_\tau=\abs{\overline{\Gamma_0(N)}_\tau}$ and $\tau_0=\rho=\frac{-1+\sqrt{3}\rmi}{2}$ if $u=4$; $\tau_0=\rmi$ if $u=6$. On the other hand, if we set $f_0(\tau)=\prod_{n \mid N}\eta(n\tau)^{r_n}$, then $\div_x(f_0)=0$ for $x\in\Gamma_0(N)\backslash\uhp$ since $f_0$ has no poles or zeros on $\uhp$. For $x\in\Gamma_0(N)\backslash\projQ$, we can find $\tbtmat{a}{b}{c}{d}\in\slZ$ such that $x=\Gamma_0(N)\cdot\tbtmat{a}{b}{c}{d}\rmi\infty$. It is known that $w_x=N/(N,c^2)$ (cf. \cite[Corollary 6.3.24(a)]{CS17}). It follows from this and \eqref{eq:orderEtaQuotient} that
\begin{equation}
\label{eq:divEtaQuotient}
\div\nolimits_x(f_0)=\frac{N}{24(N,c^2)}\sum_{n \mid N}\frac{(n,c)^2}{n}r_n.
\end{equation}
Therefore
\begin{equation}
\label{eq:threeSumsDegDiv}
\deg\left([\div f]\right)=\sum_{x\in\Gamma_0(N)\backslash\slZ\rho}[t_1\cdot\div\nolimits_x(E_4)]+\sum_{x\in\Gamma_0(N)\backslash\slZ\rmi}[t_2\cdot\div\nolimits_x(E_6)]+\sum_{x\in\Gamma_0(N)\backslash\projQ}[\div\nolimits_x(f_0)],
\end{equation}
where the three sums in the right-hand side are denoted by $S_1$, $S_2$ and $S_3$ respectively hereafter.
To compute $S_1$, let $\slZ=\cup_{1 \leq j \leq m}\Gamma_0(N)\gamma_j$ be a disjoint union\footnote{Since $-I\in\Gamma_0(N)$, this is equivalent to $\overline{\slZ}=\cup_{1 \leq j \leq m}\overline{\Gamma_0(N)}\overline{\gamma_j}$, and as well to $\widetilde{\slZ}=\cup_{1 \leq j \leq m}\widetilde{\Gamma_0(N)}\widetilde{\gamma_j}$ both of which are disjoint unions.}. Then the underlying set of the multiset $\{\Gamma_0(N)\cdot\gamma_j\rho\colon 1 \leq j \leq m\}$ equals $\Gamma_0(N)\backslash\slZ\rho$ and the multiplicity of each element $\Gamma_0(N)\cdot\gamma_j\rho$ equals $\abs{\overline{\slZ}_{\gamma_j\rho}}/\abs{\overline{\Gamma_0(N)}_{\gamma_j\rho}}=3/e_{\gamma_j\rho}$. Hence we have
\begin{equation*}
S_1=\frac{1}{3}\sum_{1 \leq j \leq m\colon e_{\gamma_j\rho}=1}\left[\frac{t_1}{e_{\gamma_j\rho}}\right] + \sum_{1 \leq j \leq m\colon e_{\gamma_j\rho}=3}\left[\frac{t_1}{e_{\gamma_j\rho}}\right]=\frac{mt_1}{3}-\left\{\frac{t_1}{3}\right\}\varepsilon_3.
\end{equation*}
Similarly, $S_2=\frac{mt_2}{2}-\left\{\frac{t_2}{2}\right\}\varepsilon_2$. For $S_3$, we need a complete set of representatives of $\Gamma_0(N)\backslash\projQ$, i.e., the set of $a/c\in\projQ$ where $c\mid N$ and for each $a_0$ such that $0 \leq a_0 < (c,N/c)$ with $(a_0,c,N/c)=1$, $a$ is chosen to satisfy $a\equiv a_0 \bmod (c,N/c)$ and $(a,c)=1$ (cf. \cite[Corollary 6.3.23]{CS17}). It follows from this and \eqref{eq:divEtaQuotient} that
\begin{equation*}
S_3=\sum_{c\mid N}\phi(c,N/c)\cdot\left[\frac{N}{24(N,c^2)}\sum_{n \mid N}\frac{(n,c)^2}{n}r_n\right].
\end{equation*}
Inserting the expressions for $S_1$, $S_2$ and $S_3$ into \eqref{eq:threeSumsDegDiv} gives the desired formula.
\end{proof}

\section{The main theorem: Dimension formulas for rational weights}
\label{sec:Dimension formulas}
We recall the Riemann-Roch theorem here for proving the main theorem. Let $X$ be a compact Riemann surface of genus $g$ and $D\in\mathscr{D}_\numQ(X)$. Then the \emph{Riemann-Roch space} is defined by $\mathscr{L}(D)=\{f\in\mathscr{M}(X)\colon\div(f)\geq-D\}$ with the convention that $0\in\mathscr{L}(D)$. It is immediate that $\mathscr{L}(D)$ is a complex vector space and that $\mathscr{L}(D)=\mathscr{L}([D])$. Similarly we define $\mathscr{K}(D)=\{\omega\in\mathscr{K}(X)\colon\div(\omega)\geq D\}\cup\{0\}$. Then the Riemann-Roch theorem states that
\begin{equation}
\label{eq:RiemannRoch}
\dim_\numC\mathscr{L}(D)=\dim_\numC\mathscr{K}(D)+\deg(D)-g+1,
\end{equation}
provided that $D$ is integral, i.e., $D\in\mathscr{D}(X)$. Note that it is tacitly understood that $\dim_\numC\mathscr{L}(D)<+\infty$. For a proof based on the existence theorems for harmonic functions on Riemann surfaces, see \cite[p. 249]{Fri11}. Set $D=0$; we find that $\dim_\numC\mathscr{K}(0)=g$, that is, the dimension of the space of all holomorphic differentials is $g$. Let $\omega_0$ be any nonzero meromorphic differential (necessarily exists) and $K=\div(\omega_0)$. Then $\mathscr{L}(K-D)\to\mathscr{K}(D)$, $f\mapsto f\cdot\omega_0$ is a $\numC$-linear isomorphism. Now set $D=K$ in \eqref{eq:RiemannRoch}; we find that $\deg(K)=2g-2$. Finally, if $D\in\mathscr{D}(X)$ and $\deg(D)>2g-2$, then $\deg(K-D)<0$ and hence $\mathscr{L}(K-D)=\mathscr{K}(D)=\{0\}$ since the degree of a nonzero meromorphic function is $0$. It follows that
\begin{equation}
\label{eq:RiemannRochSpecial}
\dim_\numC\mathscr{L}(D)=\deg(D)-g+1,\quad\text{if } D\in\mathscr{D}(X),\,\deg(D)>2g-2.
\end{equation}

Our proof of the main theorem rests on the following general lemma, which is the key point for using the Riemann-Roch theorem.
\begin{lemm}
\label{lemm:LMcorrespond}
Let $G$ be a subgroup of $\slZ$ of finite index, $D\in\numgeq{Z}{1}$, $k\in\frac{1}{D}\numZ$ and $v\colon\widetilde{G}^D\to\numC^\times$ be a character. Let $p\colon\uhp\to G\backslash\uhp$ be the natural projection. Suppose there exists a nonzero meromorphic modular form $f_0$ of weight $k$ for group $G$ with character $v$. Then the map
\begin{align*}
\mathscr{L}([\div f_0]) &\to M_k(G,v)\\
f &\mapsto (f\vert_{G\backslash\uhp})\circ p \cdot f_0
\end{align*}
is a $\numC$-linear isomorphism. Moreover, let
\begin{equation}
\label{eq:R}
R=\{x\in G\backslash\projQ\colon \div\nolimits_x(f_0)\in\numZ\}.
\end{equation}
Then the map
\begin{align*}
\mathscr{L}\left([\div f_0]-\sum\nolimits_{x\in R}(x)\right) &\to S_k(G,v)\\
f &\mapsto (f\vert_{G\backslash\uhp})\circ p \cdot f_0
\end{align*}
is also a $\numC$-linear isomorphism.
\end{lemm}
\begin{proof}
Let $f\in\mathscr{L}([\div f_0])$ be arbitrary. Since $p$ is holomorphic, $(f\vert_{G\backslash\uhp})\circ p$ is a modular function for group $G$ with trivial character. Thus $(f\vert_{G\backslash\uhp})\circ p \cdot f_0$ is a meromorphic modular form of weight $k$ for group $G$ with character $v$ by Remark \ref{rema:multiplicationModularForm}. Now assume $f\neq 0$ and let $x\in X_G$ be arbitrary. Since $f\in\mathscr{L}([\div f_0])$ we have $\div_x((f\vert_{G\backslash\uhp})\circ p)=\div_x(f)\geq -\div_x(f_0)$. It follows that
\[
\div\nolimits_x((f\vert_{G\backslash\uhp})\circ p \cdot f_0)=\div\nolimits_x((f\vert_{G\backslash\uhp})\circ p)+\div\nolimits_x(f_0) \geq 0.
\]
Hence $(f\vert_{G\backslash\uhp})\circ p \cdot f_0\in M_k(G,v)$. It is obvious that the considered map is a $\numC$-linear injection. To prove the surjectivity, let $0\neq g\in M_k(G,v)$ be arbitrary, then by Remark \ref{rema:multiplicationModularForm} $g/f_0$ is a modular function for group $G$ with trivial character and hence it descends to some nonzero $f\in\mathscr{M}(X_G)$. For any $x\in X_G$ we have
\[
\div\nolimits_x(f)=\div\nolimits_x(g/f_0)=\div\nolimits_x(g)-\div\nolimits_x(f_0)\geq-\div\nolimits_x(f_0).
\]
Therefore $f\in \mathscr{L}(\div f_0)=\mathscr{L}([\div f_0])$ and the image of $f$ is $g$ from which the surjectivity follows.

The assertion on the second map can be proved in a similar manner and the only thing that need to explain is that $(f\vert_{G\backslash\uhp})\circ p \cdot f_0\in S_k(G,v)$ if and only if $f\in\mathscr{L}\left([\div f_0]-\sum\nolimits_{x\in R}(x)\right)$. If $0\neq f\in\mathscr{L}\left([\div f_0]-\sum\nolimits_{x\in R}(x)\right)$ and $x\in R$, then
\[
\div\nolimits_x((f\vert_{G\backslash\uhp})\circ p \cdot f_0)=\div\nolimits_x(f)+\div\nolimits_x(f_0) \geq 1-[\div\nolimits_x(f_0)]+\div\nolimits_x(f_0)>0.
\]
On the other hand, if $x\not\in R$ then $-[\div\nolimits_x(f_0)]+\div\nolimits_x(f_0)>0$ and hence
\[
\div\nolimits_x((f\vert_{G\backslash\uhp})\circ p \cdot f_0)=\div\nolimits_x(f)+\div\nolimits_x(f_0) \geq -[\div\nolimits_x(f_0)]+\div\nolimits_x(f_0)>0.
\]
It follows that $(f\vert_{G\backslash\uhp})\circ p \cdot f_0\in S_k(G,v)$. The converse can be proved similarly.
\end{proof}

Now we state and prove the main theorem.
\begin{thm}
\label{thm:main}
Let $N\in\numgeq{Z}{1}$, $r_n\in\numQ$ for any $n \mid N$ and $t\in\numgeq{Z}{0}$. Let $D$ be a positive integer such that $Dr_n\in2\numZ$ for any $n \mid N$. Let $\chi\colon\Dcover{\Gamma_0(N)}{D}\to\numC^\times$ be the character \eqref{eq:charEtaQuotient} and $m$, $\varepsilon_2$, $\varepsilon_3$ are given by \eqref{eq:mGamma0N}, \eqref{eq:e2Gamma0N}, \eqref{eq:e3Gamma0N} respectively. Set $k=\frac{1}{2}\sum_n r_n+2t$ and
\begin{equation}
\label{eq:rntoxc}
x_c=\sum_{n\mid N}\frac{N}{(N,c^2)}\cdot\frac{(n,c)^2}{n}r_n
\end{equation}
with $c\mid N$. Suppose either
\begin{equation}
\label{eq:mainThmCondtion1}
\sum_{c\mid N}\phi(c,N/c)\cdot\left(\left[\frac{x_c}{24}\right]+1\right)>0,\, t \geq 1
\end{equation}
or
\begin{equation}
\label{eq:mainThmCondtion2}
k>2-\frac{6}{m}\varepsilon_2-\frac{8}{m}\varepsilon_3-\frac{12}{m}\sum_{c\mid N}\phi(c,N/c)\cdot\left(1-\left\{\frac{x_c}{24}\right\}\right),\,t=0.
\end{equation}
Then we have
\begin{multline}
\label{eq:dimFormula}
\dim_\numC M_k(\Gamma_0(N), \chi)=\frac{k-1}{12}m+\left(\frac{1}{4}-\left\{\frac{t}{2}\right\}\right)\varepsilon_2+\left(\frac{1}{3}-\left\{-\frac{t}{3}\right\}\right)\varepsilon_3\\
+\sum_{c\mid N}\phi(c,N/c)\cdot\left(\frac{1}{2}-\left\{\frac{x_c}{24}\right\}\right).
\end{multline}
\end{thm}
\begin{proof}
Set
\begin{equation*}
f_0(\tau)=\prod_{n \mid N}\eta(n\tau)^{r_n}\cdot E_4(\tau)^{-t}\cdot E_6(\tau)^{t}.
\end{equation*}
According to Proposition \ref{prop:etaEisenstein} $f_0$ is a nonzero meromorphic modular form of weight $k$ for group $\Gamma_0(N)$ with character $\chi$. Hence $\dim_\numC M_k(\Gamma_0(N), \chi)=\dim_\numC\mathscr{L}([\div f_0])$ by Lemma \ref{lemm:LMcorrespond}. Using Proposition \ref{prop:etaEisenstein} again we find that
\begin{equation}
\label{eq:degdevf0}
\deg\left([\div f_0]\right)=\sum_{c\mid N}\phi(c,N/c)\cdot\left[\frac{x_c}{24}\right]+\frac{t}{6}m-\left\{\frac{t}{2}\right\}\varepsilon_2-\left\{-\frac{t}{3}\right\}\varepsilon_3.
\end{equation}
We now prove either \eqref{eq:mainThmCondtion1} or \eqref{eq:mainThmCondtion2} implies that $\deg\left([\div f_0]\right)>2g-2$ where $g$ is the genus of $X_{\Gamma_0(N)}$ (see \eqref{eq:genusGamma0N}). First suppose \eqref{eq:mainThmCondtion1} holds. If $N=1$, then $g=0$, $r_1\geq 0$ and hence
\begin{equation*}
\deg\left([\div f_0]\right)=\left[\frac{r_1}{24}\right]+\frac{t}{6}-\left\{\frac{t}{2}\right\}-\left\{-\frac{t}{3}\right\}>-2=2g-2.
\end{equation*}
If $N \geq 2$, set
\begin{equation*}
\gamma(t)=\frac{t}{6}m-\left\{\frac{t}{2}\right\}\varepsilon_2-\left\{-\frac{t}{3}\right\}\varepsilon_3.
\end{equation*}
We have
\begin{equation*}
\gamma(t+1)-\gamma(t)\geq\frac{1}{6}m-\frac{1}{2}\varepsilon_2-\frac{2}{3}\varepsilon_3=2g-2+\varepsilon_\infty
\end{equation*}
for $t \geq 1$. Since $N \geq 2$ we have $\varepsilon_\infty \geq 2$ and hence $\gamma(t+1)-\gamma(t)\geq 0$. Therefore $\deg\left([\div f_0]\right)$ as a function of $t\in\numgeq{Z}{1}$ is increasing which means we need only to prove
\begin{equation*}
\sum_{c\mid N}\phi(c,N/c)\cdot\left[\frac{x_c}{24}\right]+\frac{1}{6}m-\frac{1}{2}\varepsilon_2-\frac{2}{3}\varepsilon_3>2g-2.
\end{equation*}
This is actually equivalent to \eqref{eq:mainThmCondtion1} according to \eqref{eq:genusGamma0N} and \eqref{eq:einftyGamma0N}. Thereby we have shown that \eqref{eq:mainThmCondtion1} implies $\deg\left([\div f_0]\right)>2g-2$. Next suppose \eqref{eq:mainThmCondtion2} holds, which is equivalent to
\begin{equation}
\label{eq:mainThmCondtion2Equiv}
\frac{1}{12}mk>\frac{1}{6}m-\frac{1}{2}\varepsilon_2-\frac{2}{3}\varepsilon_3-\varepsilon_\infty+\sum_{c\mid N}\phi(c,N/c)\left\{\frac{x_c}{24}\right\}=2g-2+\sum_{c\mid N}\phi(c,N/c)\left\{\frac{x_c}{24}\right\}.
\end{equation}
Applying Theorem \ref{thm:valence} and \eqref{eq:divEtaQuotient} to $f_0$ (noting $t=0$ here) and using the complete set of representatives of $\Gamma_0(N)\backslash\projQ$ described in the proof of Proposition \ref{prop:etaEisenstein} we find that
\begin{equation}
\label{eq:mkxc}
\frac{1}{12}m\cdot\frac{1}{2}\sum_nr_n=\deg(\div f_0)=\sum_{c \mid N}\phi(c,N/c)\cdot\frac{x_c}{24}.
\end{equation}
Inserting this into \eqref{eq:mainThmCondtion2Equiv} we obtain $\deg\left([\div f_0]\right)>2g-2$ as required.

Thus, in both cases we can apply \eqref{eq:RiemannRochSpecial} to $D=[\div f_0]$, $X=X_{\Gamma_0(N)}$. It follows that
\begin{equation*}
\dim_\numC M_k(\Gamma_0(N), \chi)=\dim_\numC\mathscr{L}([\div f_0])=\deg([\div f_0])-g+1.
\end{equation*}
Inserting \eqref{eq:degdevf0}, \eqref{eq:genusGamma0N} and \eqref{eq:mkxc} into the above identity gives the desired formula.
\end{proof}
\begin{rema}
\label{rema:uboundlbound}
If neither \eqref{eq:mainThmCondtion1} nor \eqref{eq:mainThmCondtion2} holds, then the difference between the left-hand side and the right-hand side of $\eqref{eq:dimFormula}$ equals $\dim_\numC\mathscr{K}([\div f_0])$ according to \eqref{eq:RiemannRoch}. It follows that in all cases we have
\begin{multline*}
\dim_\numC M_k(\Gamma_0(N), \chi)\geq\frac{k-1}{12}m+\left(\frac{1}{4}-\left\{\frac{t}{2}\right\}\right)\varepsilon_2+\left(\frac{1}{3}-\left\{-\frac{t}{3}\right\}\right)\varepsilon_3\\
+\sum_{c\mid N}\phi(c,N/c)\cdot\left(\frac{1}{2}-\left\{\frac{x_c}{24}\right\}\right),
\end{multline*}
where the right-hand side is always an integer. We can derive an upper bound\footnote{This upper bound is still valid if $\Gamma_0(N)$ is replaced by any subgroup $G$ of finite index in $\slZ$.} for $\dim_\numC M_k(\Gamma_0(N), \chi)$ as well. Suppose that $k\geq 0$. (When $k<0$ $M_k(\Gamma_0(N), \chi)=\{0\}$ by Theorem \ref{thm:valence}.) Then
\begin{equation*}
\dim_\numC M_k(\Gamma_0(N), \chi) \leq \left[\frac{1}{12}mk\right]+1,
\end{equation*}
for if $f_1,\dots,f_j\in M_k(\Gamma_0(N), \chi)$ with $j>\left[\frac{1}{12}mk\right]+1$, then by linear algebra we can find a nontrivial linear combination $g=\sum_n c_nf_n$ such that $\div_x(g)\geq \left[\frac{1}{12}mk\right]+1$ where $x\in \Gamma_0(N)\backslash\uhp$ is not an elliptic point. If $g \neq 0$ then applying Theorem \ref{thm:valence} to $g$ we reach a contradiction. Hence $g=0$ which means $f_1,\dots,f_j$ are linearly dependent.
\end{rema}
\begin{rema}
We can derive a formula for $\dim_\numC S_k(\Gamma_0(N), \chi)$ as well, using the second map in Lemma \ref{lemm:LMcorrespond}. However, it may happen that for certain $r_n$ and $t$ \eqref{eq:dimFormula} holds while the corresponding formula for $\dim_\numC S_k(\Gamma_0(N), \chi)$ is not applicable. This happens precisely when $\deg\left([\div f_0]\right)>2g-2$ but $\deg\left([\div f_0]\right)-\abs{R}\leq2g-2$. Besides, when $k>2$, we always have $\dim_\numC S_k(\Gamma_0(N), \chi)=\dim_\numC M_k(\Gamma_0(N), \chi)-\abs{R}$ which is a well known fact when $k$ is integral and $\chi$ is induced by a Dirichlet character. In the rational weight and arbitrary multiplier system case, this holds as well since the Petersson inner product can also be defined. The details are omitted here.
\end{rema}

To conclude this section, we write down the $\slZ$ (i.e. $N=1$) case of Theorem \ref{thm:main}, and an example illustrating the philosophy that sometimes one need to work with rational $q$-powers, instead of traditional integral $q$-powers, to make \eqref{eq:dimFormula} applicable.

\begin{coro}
Let $D$ be a positive integer, and $r\in\frac{2}{D}\numZ$ with $r\geq-24$. Let $\chi$ be the character on $\Dcover{\slZ}{D}$ that maps $\elesltRaDs{a}{b}{c}{d}{\varepsilon}$ to $\varepsilon^{-Dr/2}\etp{\frac{r}{24}\Psi\tbtmat{a}{b}{c}{d}}$. Then
\begin{equation}
\label{eq:dimSlZ}
\dim_\numC M_{r/2}(\slZ, \chi)=\left[\frac{r}{24}\right]+1.
\end{equation}
\end{coro}
\begin{proof}
Setting $N=1$, $r_1=r$, and $t=0$ in Theorem \ref{thm:main}, we have $x_1=r$, $k=r/2$, and $m=\varepsilon_2=\varepsilon_3=1$. Since $r\geq-24$, \eqref{eq:mainThmCondtion2} holds. Therefore, \eqref{eq:dimFormula} is true, which becomes \eqref{eq:dimSlZ} in the current setting.
\end{proof}

\begin{examp}
We consider
\begin{align*}
\eta(\tau)^{1/3}&=q^{1/72}\cdot\left(\prod_{n\geq1}(1-q^n)\right)^{1/3}\\
&=q^{1/72}\cdot\left(1-\frac{1}{3}q-\frac{4}{9}q^2-\frac{23}{81}q^3-\frac{82}{243}q^4-\frac{34}{729}q^5-\frac{1711}{6561}q^6+\cdots\right).
\end{align*}
Note that there are three holomorphic branches of the third roots of $\prod_{n\geq1}(1-q^n)$. By our definitions of fractional powers and of $\log\eta$ (see \eqref{eq:logeta}), we should choose the branch whose first coefficient is $1$. See \cite[Remark 4.3]{Zhu23} for a proof. Let $\chi$ be the character on $\Dcover{\slZ}{6}$ that maps $\elesltRaDs{a}{b}{c}{d}{\varepsilon}$ to $\varepsilon^{-1}\etp{\frac{1}{72}\Psi\tbtmat{a}{b}{c}{d}}$. We have $\eta^{1/3}\in M_{1/6}(\slZ, \chi)$ (see Section \ref{sec:Modular forms of rational weight}) and $\dim_\numC M_{1/6}(\slZ, \chi)=1$ (see \eqref{eq:dimSlZ}). It follows that
\begin{equation}
\label{eq:basisM16slZ}
M_{1/6}(\slZ, \chi)=\numC\eta^{1/3}.
\end{equation}
This has some potential applications. For example, if we could establish a framework of Harmonic Maass forms of rational weight (cf. \cite[Section 2]{LL23}, where some aspects of the theory of Harmonic Maass forms have been developed for real weights), and define the corresponding Poincar\'e-Maass series, then there might exist a Poincar\'e-Maass series whose shadow lies in $M_{1/6}(\slZ, \chi)=S_{1/6}(\slZ, \chi)$, making it proportional to $\eta^{1/3}$. This would yield a formula for the Fourier coefficients of $\eta^{1/3}$. Once obtained, we could further estimate both the magnitude and the sign of the coefficients.

Traditionally, one tends to consider $\eta(72\tau)^{1/3}$ instead of $\eta(\tau)^{1/3}$, as the former has the feature that its Fourier expansion involves only integral powers of $q$. We have $\eta(72\tau)^{1/3}\in M_{1/6}(\Gamma_0(72), \chi_1)$, where $\chi_1\colon\Dcover{\Gamma_0(72)}{6}\rightarrow\numC^\times$ is the character given by $\chi_1\elesltRaDs{a}{b}{c}{d}{\varepsilon}=\varepsilon^{-1}\etp{\frac{1}{72}\Psi\tbtmat{a}{72b}{c/72}{d}}$. See \eqref{eq:charEtaQuotient}. Now we check whether Theorem \ref{thm:main} is applicable to $M_{1/6}(\Gamma_0(72), \chi_1)$. For $N=72$, we have $m=144$ and $\varepsilon_2=\varepsilon_3=0$ by \eqref{eq:mGamma0N}--\eqref{eq:e3Gamma0N}. By \eqref{eq:rntoxc}, we have
\begin{equation*}
(x_1, x_2, x_3, x_4, x_6, x_8, x_9, x_{12}, x_{18}, x_{24}, x_{36}, x_{72})=3^{-1}\cdot(1,1,1,2,1,8,9,2,9,8,18,72).
\end{equation*}
Inserting these quantities into \eqref{eq:mainThmCondtion2} we find that the right-hand side of it equals $3/4$. Therefore, \eqref{eq:mainThmCondtion2} does not hold, making Theorem \ref{thm:main} not applicable to $M_{1/6}(\Gamma_0(72), \chi_1)$. This means the potential applications mentioned above no longer exist if one works with $\eta(72\tau)^{1/3}$.
\end{examp}

\section{Order-character relations}
\label{sec:Order-character relations}
Let us assume \eqref{eq:mainThmCondtion2} holds and consider Theorem \ref{thm:main}. The sequence $(x_c)_{c\mid N}$ occurs in the formula for $\dim_\numC M_k(\Gamma_0(N),\chi)$ and is determined by the sequence $(r_n)_{n \mid N}$ via \eqref{eq:rntoxc}. Note that $(r_n)_{n \mid N}$ represents an eta-quotient (or its character) and $(x_c)_{c\mid N}$ represents the orders of this eta-quotient at cusps (or its divisor). Therefore \eqref{eq:rntoxc} can be regarded as a map that sends the character represented by $(r_n)_{n \mid N}$ to the orders at cusps $(x_c)_{c\mid N}$. In some potential applications, the sequence of orders $(x_c)_{c\mid N}$ is first given, e.g., when one want to known $\dim_\numC M_k(\Gamma_0(N),\chi)$ for fixed $N$ and $k$, and then one works out $(r_n)_{n \mid N}$ via \eqref{eq:rntoxc}. Motivated by this, we need an inverse formula of \eqref{eq:rntoxc}. This formula is an enhancement of a statement in \cite[p. 129]{KM08} and is equivalent to \cite[Eq. (5.13)]{Bha17}.

Let us fix some notations in linear algebra. For $N\in\numgeq{Z}{1}$, $\mathscr{D}(N)$ denotes the set of positive divisors of $N$ and $\sigma_0(N)=\abs{\mathscr{D}(N)}$. Let $V$ be a $\numQ$-vector space of dimension $\sigma_0(N)$, then a $\mathscr{D}(N)$-indexed basis of $V$ is a sequence $(\mathbf{b}_n)_{n\in \mathscr{D}(N)}$ so that $\{\mathbf{b}_n\}$ is a basis of $V$. Let $W$ be a $\numQ$-vector space of dimension $\sigma_0(M)$ with a $\mathscr{D}(M)$-indexed basis $(\mathbf{b}'_m)_{m \mid M}$ where $M$ is another positive integer and let $f\colon V\to W$ be a linear map. The matrix of $f$ with respect to the pair $((\mathbf{b}_n),(\mathbf{b}'_m))$ is the $\mathscr{D}(M)\times\mathscr{D}(N)$-indexed sequence $(a_{m,n})_{m\mid M,n\mid N}$  with the property $f(\mathbf{b}_n)=\sum_{m \mid M}a_{m,n}\mathbf{b}'_m$. All linear algebra machinery works for such kind of matrices. The difference is just the underlying index sets: we use $\mathscr{D}(N)$ and the ordinary one is $\{1,2,\dots,\sigma_0(N)\}$. We call the matrix $(a_{m,n})_{m\mid M,n\mid N}$ a $\mathscr{D}(M)\times\mathscr{D}(N)$-indexed matrix and when $M=N$ a $\mathscr{D}(N)$-indexed matrix.

Let $N_1,N_2\in\numgeq{Z}{1}$ such that $\gcd(N_1,N_2)=1$ and let $(a_{c_1,n_1})_{c_1,n_1\mid N_1}$, $(b_{c_2,n_2})_{c_2,n_2\mid N_2}$ be $\mathscr{D}(N_1)$-indexed and $\mathscr{D}(N_2)$-indexed matrices respectively. Define
\begin{equation*}
(a_{c_1,n_1})_{c_1,n_1\mid N_1} \otimes(b_{c_2,n_2})_{c_2,n_2\mid N_2}=(a_{c_1,n_1}\cdot b_{c_2,n_2})_{c_1c_2,n_1n_2\mid N_1N_2}.
\end{equation*}
It is immediate that this is a well defined $\numQ$-bilinear map from $\numQ^{\mathscr{D}(N_1)\times\mathscr{D}(N_1)}\times\numQ^{\mathscr{D}(N_2)\times\mathscr{D}(N_2)}$ to $\numQ^{\mathscr{D}(N_1N_2)\times\mathscr{D}(N_1N_2)}$. Moreover it is universal among all bilinear maps on $\numQ^{\mathscr{D}(N_1)\times\mathscr{D}(N_1)}\times\numQ^{\mathscr{D}(N_2)\times\mathscr{D}(N_2)}$ and hence is a tensor product. We call it the \emph{Kronecker product}.

Following Bhattacharya \cite{Bha17} we define $A_N=\left(\frac{N}{(N,c^2)}\cdot\frac{(n,c)^2}{n}\right)_{c,n\mid N}$ so that \eqref{eq:rntoxc} becomes $x_c=\sum_{n\mid N}A_N(c,n)r_n$. It is immediate that $A_{N_1N_2}=A_{N_1}\otimes A_{N_2}$ whenever $\gcd(N_1,N_2)=1$.
\begin{prop}
\label{prop:rnxc}
Let $N\in\numgeq{Z}{1}$, $(x_c)_{c\mid N}$ and $(r_n)_{n\mid N}$ be in $\numQ^{\mathscr{D}(N)}$. Then the following two relations are equivalent:
\begin{align}
x_c&=\sum_{n\mid N}\frac{N}{(N,c^2)}\cdot\frac{(n,c)^2}{n}r_n,\quad \forall c\mid N, \label{eq:rntoxc2}\\
r_n&=\frac{1}{N}\prod_{p \mid N}\frac{p}{p^2-1}\cdot\sum_{c\mid N}\left(\prod_{p^\alpha \parallel N}b_{p,\alpha}(v_p(n),v_p(c))\right)x_c,\quad \forall n\mid N, \label{eq:xctorn}
\end{align}
where $p$ denotes a prime, $v_p(n)$ is the integer $\beta$ such that $p^\beta\parallel n$ and
\begin{equation*}
b_{p,\alpha}(i,j)=\begin{dcases}
p &\text{if } i=j=0\text{ or }\alpha,\\
(p^2+1)p^{\min(i,\alpha-i)-1} &\text{if } 0<i=j<\alpha,\\
-p^{\min(j,\alpha-j)} &\text{if } \abs{i-j}=1,\\
0 &\text{otherwise}.
\end{dcases}
\end{equation*}
\end{prop}
In another words, we have
\begin{equation*}
A_N^{-1}(n,c)=\frac{1}{N}\prod_{p \mid N}\frac{p}{p^2-1}\cdot\prod_{p^\alpha \parallel N}b_{p,\alpha}(v_p(n),v_p(c)).
\end{equation*}
This proposition is due to Bhattacharya (cf. \cite[Eq. (5.13)]{Bha17}).
\begin{proof}
Suppose $N_1,N_2\in\numgeq{Z}{1}$ such that $\gcd(N_1,N_2)=1$. Then $A_{N_1}$ and $A_{N_2}$ are both invertible if and only if $A_{N_1}\otimes A_{N_2}$ is and in this case $(A_{N_1}\otimes A_{N_2})^{-1}=A_{N_1}^{-1}\otimes A_{N_2}^{-1}$ by basic properties of the Kronecker product. Therefore it is sufficient to prove that
\begin{equation*}
A_{p^\alpha}^{-1}(p^i,p^j)=\frac{1}{p^{\alpha-1}(p^2-1)}\cdot b_{p,\alpha}(i,j),
\end{equation*}
where $p$ is a prime, $\alpha\in\numgeq{Z}{1}$ and $0\leq i,j\leq\alpha$. This is equivalent to
\begin{align}
&\sum_{j=0}^{\alpha}A_{p^\alpha}(p^i,p^j)\cdot b_{p,\alpha}(j,i)=(p^2-1)p^{\alpha-1}, \label{eq:toProveInverse1}\\
&\sum_{j=0}^{\alpha}A_{p^\alpha}(p^{i_1},p^j)\cdot b_{p,\alpha}(j,i_2)=0 \text{  if } i_1\neq i_2. \label{eq:toProveInverse2}
\end{align}
Note that
\begin{equation*}
A_{p^\alpha}(p^i,p^j)=\frac{p^\alpha}{(p^\alpha,p^{2i})}\cdot\frac{(p^j,p^i)^2}{p^j}=\begin{dcases}
p^{\alpha-j} &\text{  if }i\leq j \text{ and } i\leq\alpha/2,\\
p^{2i-j} &\text{  if }i\leq j \text{ and } i>\alpha/2,\\
p^{\alpha-2i+j} &\text{  if }i\geq j \text{ and } i\leq\alpha/2,\\
p^{j} &\text{  if }i\geq j \text{ and } i>\alpha/2.
\end{dcases}
\end{equation*}
Now we begin to prove \eqref{eq:toProveInverse1}. In the case $i=0$, we have
\begin{align*}
\sum_{j=0}^{\alpha}A_{p^\alpha}(p^i,p^j)\cdot b_{p,\alpha}(j,i)&=A_{p^\alpha}(p^0,p^0)\cdot b_{p,\alpha}(0,0)+A_{p^\alpha}(p^0,p^1)\cdot b_{p,\alpha}(1,0)\\
&=p^\alpha\cdot p + p^{\alpha-1}\cdot(-1)=(p^2-1)p^{\alpha-1}
\end{align*}
as required. In the case $i=\alpha$, we have
\begin{align*}
\sum_{j=0}^{\alpha}A_{p^\alpha}(p^i,p^j)\cdot b_{p,\alpha}(j,i)&=A_{p^\alpha}(p^\alpha,p^{\alpha-1})\cdot b_{p,\alpha}(\alpha-1,\alpha)+A_{p^\alpha}(p^\alpha,p^\alpha)\cdot b_{p,\alpha}(\alpha,\alpha)\\
&=p^{\alpha-1}\cdot(-1) + p^{\alpha}\cdot p=(p^2-1)p^{\alpha-1}
\end{align*}
as required. In the case $0<i \leq \alpha/2$ we have
\begin{align*}
&\sum_{j=0}^{\alpha}A_{p^\alpha}(p^i,p^j)\cdot b_{p,\alpha}(j,i)\\
=&A_{p^\alpha}(p^i,p^{i-1})\cdot b_{p,\alpha}(i-1,i)+A_{p^\alpha}(p^i,p^{i})\cdot b_{p,\alpha}(i,i)+A_{p^\alpha}(p^i,p^{i+1})\cdot b_{p,\alpha}(i+1,i)\\
=&p^{\alpha-i-1}\cdot(-p^{\min(i,\alpha-i)}) + p^{\alpha-i}\cdot (p^2+1)p^{\min(i,\alpha-i)-1}+p^{\alpha-i-1}\cdot(-p^{\min(i,\alpha-i)})\\
=&(p^2-1)p^{\alpha-1}
\end{align*}
as required. In the remaining case $\alpha/2<i<\alpha$ we have
\begin{align*}
&\sum_{j=0}^{\alpha}A_{p^\alpha}(p^i,p^j)\cdot b_{p,\alpha}(j,i)\\
=&A_{p^\alpha}(p^i,p^{i-1})\cdot b_{p,\alpha}(i-1,i)+A_{p^\alpha}(p^i,p^{i})\cdot b_{p,\alpha}(i,i)+A_{p^\alpha}(p^i,p^{i+1})\cdot b_{p,\alpha}(i+1,i)\\
=&p^{i-1}\cdot(-p^{\min(i,\alpha-i)}) + p^{i}\cdot (p^2+1)p^{\min(i,\alpha-i)-1}+p^{i-1}\cdot(-p^{\min(i,\alpha-i)})\\
=&(p^2-1)p^{\alpha-1}
\end{align*}
which concludes the proof of \eqref{eq:toProveInverse1}. We can prove \eqref{eq:toProveInverse2} in a similar manner which is tedious (there are so many cases) so we omit the details.
\end{proof}
\begin{rema}
\label{rema:ANcolumn}
There is another useful formula concerning the matrix $A_N$:
\begin{equation*}
\sum_{c\mid N}\phi(c,N/c)\frac{N}{(N,c^2)}\frac{(n,c)^2}{n}=N\prod_{p \mid N}\left(1+\frac{1}{p}\right).
\end{equation*}
In another words, the Euclidean inner product of each column of $A_N$ and $(\phi(c,N/c))_{c\mid N}$ is equal to $N\prod_{p \mid N}\left(1+\frac{1}{p}\right)$. This can be proved by applying Theorem \ref{thm:valence} to $\eta(n\tau)$. As a consequence, $\sum_{c\mid N}\phi(c,N/c)x_c=N\prod_{p \mid N}\left(1+\frac{1}{p}\right)\cdot\sum_{n\mid N}r_n$.
\end{rema}

\begin{examp}
We present some examples concerning dimensions of spaces of weight $2$ for infinite many unitary characters. Suppose that $N$ is square-free for simplicity. Let $\mathbf{x}=(x_c)_{c\mid N}\in\numQ^{\mathscr{D}(N)}$ such that $\sum_{c\mid N}x_c=0$. Then \eqref{eq:mainThmCondtion1} with $t=1$ holds since $\phi(c,N/c)=1$ in the current case. Thus \eqref{eq:dimFormula} holds. According to the above remark we have $\sum_{n\mid N}r_n=m^{-1}\sum_{c\mid N}x_c=0$ and hence $k=2$. Therefore \eqref{eq:dimFormula} is equivalent to
\begin{equation}
\label{eq:dimFormulak2}
\dim_\numC M_2(\Gamma_0(N), \chi)=g-1+\sigma_0(N)+\sum_{c\mid N}\left[\frac{x_c}{24}\right],
\end{equation}
where $g$ is the genus of $X_{\Gamma_0(N)}$. The multiplier system $\chi$ is the character \eqref{eq:charEtaQuotient} where $(r_n)_{n\mid N}$ is determined by $\mathbf{x}$ via \eqref{eq:xctorn} which in the current case is
\begin{equation}
\label{eq:xctornSpecial}
r_n=\prod_{p\mid N}(p^2-1)^{-1}\cdot \sum_{c\mid N}\left((-1)^{\#\{p\mid N\colon v_p(n)\neq v_p(c)\}}\cdot\prod_{p\mid N,\,v_p(n)=v_p(c)}p\right)x_c.
\end{equation}
More generally, if $N\in\numgeq{Z}{1}$ is not divisible by $16$ and has no odd square factor except $1$, then we have $\phi(c,N/c)=1$ for any $c\mid N$ as well and hence \eqref{eq:dimFormulak2} is still valid. However in this case, \eqref{eq:xctornSpecial} should be modified slightly at the prime factor $2$.

Now we try to find out one-dimensional spaces among the above spaces. Note that
\begin{equation*}
g-1<\dim_\numC M_2(\Gamma_0(N), \chi)\leq g-1+\sigma_0(N),
\end{equation*}
so it is necessary $g=0$ or $1$ and $N>1$. When $g=0$, since $N\in\numgeq{Z}{2}$ is not divisible by $16$ and has no odd square factor except $1$, then $N\in\{2,3,4,5,6,7,8,10,12,13\}$. For such $N$ $\dim_\numC M_2(\Gamma_0(N), \chi)=1$ if $\mathbf{x}$ satisfies
\begin{equation*}
\sum_{c\mid N}x_c=0,\quad \sum_{c\mid N}\left[\frac{x_c}{24}\right]=2-\sigma_0(N).
\end{equation*}
When $g=1$, we have $N\in\{11,14,15,17,19,20,21,24\}$. For such $N$, $\dim_\numC M_2(\Gamma_0(N), \chi)=1$ if $\mathbf{x}$ satisfies
\begin{equation*}
\sum_{c\mid N}x_c=0,\quad \sum_{c\mid N}\left[\frac{x_c}{24}\right]=1-\sigma_0(N).
\end{equation*}

Let us look at a very simple case $N=11$. Since $\sigma_0(11)=2$, $\dim_\numC M_2(\Gamma_0(11), \chi)=1$ if and only if $x_1\not\in24\numZ$ (of course $x_{11}=-x_1$). We have $r_1=-r_{11}=x_1/10$ so $\chi$ is the multiplier system of $\left(\eta(\tau)/\eta(11\tau)\right)^{x_1/10}$ which descends to a unitary character on $\Gamma_0(11)$. If $x_1\in24\numZ$, then $\dim_\numC M_2(\Gamma_0(11), \chi)=2$. There are two completely different subcases in this case: $x_1\in120\numZ$ and $x_1\not\in120\numZ$. If $x_1\in120\numZ$, then $\chi$ is trivial and the formula reads $\dim_\numC M_2(\Gamma_0(11))=2$ which is well known (cf. \cite[Corollary 7.4.3]{CS17}). Otherwise, if $x_1\not\in120\numZ$ then $x_1\not\in10\numZ$, which implies that $\chi$ is the character of an eta-quotient of rational exponents. To our best knowledge, the dimension formulas in such situation are new. Moreover, it seems that $\ker\chi$ is a noncongruence subgroup of $\slZ$ so that our formulas may be applied to the theory of noncongruence modular forms.
\end{examp}

\section{The classification of characters induced by eta-quotients}
\label{sec:The classification of characters induced by eta-quotients}
In this section we give a complete classification of the characters of eta-quotients of level $N$ and cover index $2$ for any $N\in\numgeq{Z}{1}$. The motivation is that, with this classification in hand, we can know exactly the set of $\chi$ for which $\dim_\numC M_k(\Gamma_0(N), \chi)$ can be calculated using Theorem \ref{thm:main}. Since we only deal with the cover index $D=2$, set $\widetilde{\Gamma_0(N)}=\Dcover{\Gamma_0(N)}{2}$ in this section.

Now we establish some notations. Each $(r_n)_{n\mid N}\in\numZ^{\mathscr{D}(N)}$ represents an eta-quotient $\prod_{n\mid N}\eta(n\tau)^{r_n}$ of level $N$ and cover index $2$ bijectively and $(r_n)_{n\mid N}$ with $\sum_{n}r_n=2k$ represents an eta-quotient of weight $k$. Thus, set
\begin{equation*}
S(N,k)=\left\{(r_n)_{n\mid N}\in\numZ^{\mathscr{D}(N)}\colon\sum\nolimits_{n\mid N}r_n=2k\right\},\quad N\in\numgeq{Z}{1},\, k\in\frac{1}{2}\numZ,
\end{equation*}
and $S(N)=S(N,0)$.
Then $\numZ^{\mathscr{D}(N)}=\cup_{2k\in\numZ}S(N,k)$ which is a disjoint union. In fact, $S(N)$ is a submodule of $\numZ^{\mathscr{D}(N)}$ and each $S(N,k)$ is a coset in $\numZ^{\mathscr{D}(N)}/S(N)$. The element $(r_n)_{n\mid N}$ in $\numZ^{\mathscr{D}(N)}$ is customary to be written as $\mathbf{r}=\sum_{n\mid N}r_ne_n$ where $\{e_n\colon n\mid N\}$ is the standard basis of $\numZ^{\mathscr{D}(N)}$. Moreover, the character \eqref{eq:charEtaQuotient} with $D=2$ is denoted by $\chi_{\mathbf{r}}$ to emphasize its dependence upon $\mathbf{r}\in S(N,k)$.

Two different vectors $\mathbf{r}$, $\mathbf{r}'$ in $S(N,k)$ may correspond to the same character $\chi_{\mathbf{r}}=\chi_{\mathbf{r}'}$. The following lemma of Newman \cite{New59} tells us precisely when this happens.
\begin{lemm}
\label{lemm:whenCharSame}
Let $N\in\numgeq{Z}{1}$, $k\in\frac{1}{2}\numZ$. Let $\mathbf{r}=(r_n)_{n\mid N}$ and $\mathbf{r}'=(r'_n)_{n\mid N}$ be vectors in $S(N,k)$. Then $\chi_{\mathbf{r}}=\chi_{\mathbf{r}'}$ if and only if
\begin{gather*}
\sum_{n\mid N}n\cdot(r_n-r'_n)\equiv0\pmod{24},\\
\sum_{n\mid N}\frac{N}{n}\cdot(r_n-r'_n)\equiv0\pmod{24},\\
\prod_{2\nmid r_n-r'_n}n \text{ is a perfect square}.
\end{gather*}
\end{lemm}
\begin{proof}
This is a special case of the last assertion of \cite[Theorem 3.9]{ZZ23}. Setting $l=1$ in that assertion and noting that $\chi_{\mathbf{r}}=\chi_{\mathbf{r}'}$ if and only if they are $\widetilde{I}$-compatible give the desired equivalence.
\end{proof}
\begin{rema}
It seems that this method does not work in the case $D\geq3$ since we can not find any generalized double coset operator (cf. \cite[Definition 1]{ZZ23}) between characters of eta-quotients of cover index $D\geq3$.
\end{rema}

As a corollary, for $\mathbf{r}=(r_n)_{n\mid N}\in S(N)$, $\chi_{\mathbf{r}}$ is trivial if and only if $\sum_{n}n\cdot r_n\equiv0\bmod{24}$, $\sum_{n}\frac{N}{n}\cdot r_n\equiv0\bmod{24}$ and $\prod_{2\nmid r_n}n$ is a perfect square. The set of all such $\mathbf{r}$ is denoted by $Z(N)$ which is a subgroup of $S(N)$. We define
\begin{equation*}
\mathscr{S}(N,k)=S(N,k)/Z(N)=\left\{\mathbf{r}+Z(N)\colon\mathbf{r}\in S(N,k)\right\},
\end{equation*}
and $\mathscr{S}(N)=\mathscr{S}(N,0)$.
Thus $\numZ^{\mathscr{D}(N)}/Z(N)=\cup_{2k\in\numZ}\mathscr{S}(N,k)$, which is a disjoint union. Note that elements of $\mathscr{S}(N,k)$ are in one-to-one correspondence with characters of eta-quotients of level $N$, cover index $2$ and weight $k$ according to Lemma \ref{lemm:whenCharSame}, so to classify such characters we need only to obtain a complete system of representatives of $\mathscr{S}(N,k)$. Moreover, the map $\mathscr{S}(N)\to\mathscr{S}(N,k)$ that sends $\sum_{n\mid N}r_ne_n+Z(N)$ to $2ke_1+\sum_{n\mid N}r_ne_n+Z(N)$ is a bijection. Therefore, we have reduced our task to the study of $\mathscr{S}(N)$. (For the result, see Theorem \ref{thm:SNrepr}.)
\begin{rema}
\label{rema:SNsubgroup}
Let $\chi_{\mathbf{r}}\colon\widetilde{\Gamma_0(N)}\to\numC^\times$ be a character given by \eqref{eq:charEtaQuotient} where $\mathbf{r}\in S(N)$. Then $\chi_{\mathbf{r}}$ is the multiplier system of an eta-quotient of weight $0$ and cover index $2$. According to the third property in Remark \ref{rema:mulSysProp}, $\chi_{\mathbf{r}}$ descends to a character on $\Gamma_0(N)$. Thus $\mathscr{S}(N)$ can be regarded as a subgroup of the group of unitary linear characters of $\Gamma_0(N)$.
\end{rema}

The sets given in the following definition are subsets of $\mathscr{D}(N)$ that occur in the complete system of representatives of $\mathscr{S}(N)$.
\begin{deff}
Let $N$ be a positive integer and $p$ be a prime divisor of $N$. If $p\neq2,3$ then define $\mathscr{B}_N^p=\{p\}$. For $p=2$, we define
\begin{equation*}
\mathscr{B}_N^2=\begin{dcases}
\{2^1\}  &\text{ if }v_2(N)=1,\\
\{2^1,2^2\}  &\text{ if }v_2(N)=2,\\
\{2^{\alpha-2},2^{\alpha-1},2^{\alpha}\}  &\text{ if }v_2(N)=\alpha\geq3.\\
\end{dcases}
\end{equation*}
For $p=3$, we define
\begin{equation*}
\mathscr{B}_N^3=\begin{dcases}
\{3^1\}  &\text{ if }v_3(N)=1,\\
\{3^{\alpha-1},3^{\alpha}\}  &\text{ if }v_3(N)=\alpha\geq2.\\
\end{dcases}
\end{equation*}
Finally, set $\mathscr{B}_N=\cup_{p\mid N}\mathscr{B}_N^p$. Note that $\mathscr{B}_1=\emptyset$.
\end{deff}
We introduce a useful notation which will be used in the proofs below: if $\mathbf{r}_1,\mathbf{r}_2\in S(N)$, then $\mathbf{r}_1\sim_N\mathbf{r}_2$, or simply $\mathbf{r}_1\sim\mathbf{r}_2$, means $\mathbf{r}_1-\mathbf{r}_2\in Z(N)$. This is an equivalence relation compatible with the addition of $S(N)$.
\begin{lemm}
\label{lemm:EquivZN}
Let $N=2^{\alpha_0}\cdot 3^{\alpha_1}\cdot N_1$ with $\gcd(N_1,6)=1$ and $\alpha_0,\alpha_1\in\numgeq{Z}{0}$.
\begin{itemize}
  \item For any $0\leq\beta_0\leq\alpha_0$, $0\leq\beta_1\leq\alpha_1$ and $n_1,n_2\mid N_1$ with $\gcd(n_1,n_2)=1$, there exist $x,y,z\equiv1\bmod 2$ such that
  \begin{equation}
  \label{eq:lemmEquivZN1}
  e_{2^{\beta_0}3^{\beta_1}n_1n_2}\sim xe_{2^{\beta_0}3^{\beta_1}n_1}+ye_{2^{\beta_0}3^{\beta_1}n_2}+ze_{2^{\beta_0}3^{\beta_1}}.
  \end{equation}
  \item For any $0\leq\beta_0\leq\alpha_0$, $0\leq\beta_1\leq\alpha_1$ and $n\mid N_1$, there exist $x,y\equiv1\bmod 2$ and $z\in\numZ$ such that
  \begin{equation}
  \label{eq:lemmEquivZN2}
  e_{2^{\beta_0}3^{\beta_1}n}\sim xe_{2^{\beta_0}3^{\beta_1}}+ye_{n}+ze_{1}.
  \end{equation}
  \item For any $0\leq\beta_0\leq\alpha_0$, $0\leq\beta_1\leq\alpha_1$ , there exist $x,y,z\in\numZ$ such that
  \begin{equation}
  \label{eq:lemmEquivZN3}
  e_{2^{\beta_0}3^{\beta_1}}\sim xe_{2^{\beta_0}}+ye_{3^{\beta_1}}+ze_{1}.
  \end{equation}
\end{itemize}
\end{lemm}
\begin{proof}
First we prove \eqref{eq:lemmEquivZN1}. If $n_1=1$ or $n_2=1$ then \eqref{eq:lemmEquivZN1} is immediate. Thus assume that $n_1,n_2>1$. By the definition of $Z(N)$, a solution $(x,y,z)\in\numZ^3$ of the system
\begin{equation}
\label{eq:prooflemmEquivZN1}
\begin{dcases}
2^{\beta_0}3^{\beta_1}n_1x+2^{\beta_0}3^{\beta_1}n_2y+2^{\beta_0}3^{\beta_1}z \equiv 2^{\beta_0}3^{\beta_1}n_1n_2 \pmod{24}\\
\frac{N}{2^{\beta_0}3^{\beta_1}n_1}x+\frac{N}{2^{\beta_0}3^{\beta_1}n_2}y+\frac{N}{2^{\beta_0}3^{\beta_1}}z \equiv \frac{N}{2^{\beta_0}3^{\beta_1}n_1n_2} \pmod{24}\\
x,y,z\equiv 1 \pmod 2
\end{dcases}
\end{equation}
is also a solution of \eqref{eq:lemmEquivZN1}. Moreover, a solution of the system
\begin{equation}
\label{eq:prooflemmEquivZN2}
\begin{dcases}
n_1x+n_2y+z \equiv n_1n_2 \pmod{24}\\
n_2x+n_1y+n_1n_2z \equiv 1 \pmod{24}\\
x,y,z\equiv 1 \pmod 2
\end{dcases}
\end{equation}
is a solution of \eqref{eq:prooflemmEquivZN1}. Since $2,3\nmid n_1,n_2$ we have $n_1^2\equiv 1\bmod 24$ and $n_2^2\equiv 1\bmod 24$. Thus the first equation of \eqref{eq:prooflemmEquivZN2} is equivalent to the second one. Therefore let $x,y\equiv 1 \bmod 2$ be arbitrary and set $z\equiv n_1n_2-n_1x-n_2y$; we obtain a solution of \eqref{eq:prooflemmEquivZN2} and hence of \eqref{eq:lemmEquivZN1}.

The proofs of \eqref{eq:lemmEquivZN2} and \eqref{eq:lemmEquivZN3} are similar of which we omit the details.
\end{proof}

\begin{lemm}
\label{lemm:SNsimpleForm}
Any element in $\mathscr{S}(N)$ can be represented as
\begin{equation}
\label{eq:elementSNrepr}
\sum_{n\in\mathscr{B}_N}c_ne_n-\left(\sum_{n\in\mathscr{B}_N}c_n\right)e_1+Z(N),
\end{equation}
where $c_n\in\numZ$ and $0\leq c_n<24$ for any $n\in\mathscr{B}_N$.
\end{lemm}
As a consequence, $\mathscr{S}(N)$ is a finite group and $\abs{\mathscr{S}(N)}\leq 24^{\abs{\mathscr{B}_N}}$.
\begin{proof}
Let $\mathbf{r}+Z(N)=\sum r_ne_n+Z(N)\in\mathscr{S}(N)$ be arbitrary. According to Lemma \ref{lemm:EquivZN} we can assume that $r_n=0$ unless $n$ is a prime power or $n=1$. Therefore there exist $\mathbf{r}^{(p)}+Z(N)=\sum_{0\leq\beta\leq v_p(N)}r^{(p)}_{p^\beta}e_{p^\beta}+Z(N)\in\mathscr{S}(N)$ such that $\mathbf{r}\sim\sum_{p\mid N}\mathbf{r}^{(p)}$. (Explicitly, let $r^{(p)}_{p^\beta}=r_{p^\beta}$ for $\beta>0$ and $r^{(p)}_{p^0}=-\sum_{\beta>0}r_{p^\beta}$.) It remains to find $c_n\in\numZ$ such that
\begin{equation}
\label{eq:toProveLocalatp}
\mathbf{r}^{(p)}\sim\sum_{n\in\mathscr{B}_N^p}c_ne_n-\left(\sum_{n\in\mathscr{B}_N^p}c_n\right)e_1
\end{equation}
and $0\leq c_n < 24$ for each $p\mid N$. If $p\neq2,3$ then \eqref{eq:toProveLocalatp} is equivalent to
\begin{equation*}
\begin{dcases}
\sum_{1\leq\beta\leq v_p(N)}(p^\beta-1)r^{(p)}_{p^\beta}\equiv (p-1)c_p \pmod{24}\\
\sum_{1\leq\beta\leq v_p(N)}(Np^{-\beta}-N)r^{(p)}_{p^\beta}\equiv (Np^{-1}-N)c_p \pmod{24}\\
c_p\equiv \sum_{\beta \geq1,\,r^{(p)}_{p^\beta}\equiv1\bmod2}\beta \pmod 2.
\end{dcases}
\end{equation*}
Since $p^2\equiv 1\bmod{24}$, the solution of this system is $c_p\equiv\sum_{2\nmid \beta}r^{(p)}_{p^\beta}\bmod{24/(12,p-1)}$. If $p=2$ and $v_2(N)=\alpha\geq3$, then \eqref{eq:toProveLocalatp} is equivalent to
\begin{equation*}
\begin{dcases}
\sum_{1\leq\beta\leq \alpha}(2^\beta-1)r^{(2)}_{2^\beta}\equiv (2^\alpha-1)c_{2^\alpha}+(2^{\alpha-1}-1)c_{2^{\alpha-1}}+(2^{\alpha-2}-1)c_{2^{\alpha-2}} \pmod{24}\\
\sum_{1\leq\beta\leq \alpha}(N2^{-\beta}-N)r^{(2)}_{2^\beta}\equiv (N2^{-\alpha}-N)c_{2^\alpha}+(N2^{-\alpha+1}-N)c_{2^{\alpha-1}}+(N2^{-\alpha+2}-N)c_{2^{\alpha-2}} \pmod{24}\\
\sum_{\twoscript{\beta\in\{\alpha-2,\alpha-1,\alpha\}}{c_{2^\beta}\equiv r^{(2)}_{2^\beta}+1\bmod2}}\beta\equiv \sum_{1\leq\beta<\alpha-2,\,r^{(2)}_{2^\beta}\equiv1\bmod2}\beta \pmod 2.
\end{dcases}
\end{equation*}
By considering the subcases $2\mid \alpha\geq6$, $2\nmid \alpha\geq5$ and $\alpha=3,4$ separately and splitting the congruences modulo $24$ to congruences modulo $3$ and $8$, one can verify that this system is solvable. We omit the proof of the case $p=2$, $v_2(N)=1,2$. If $p=3$ and $v_3(N)=\alpha\geq2$, then 
\eqref{eq:toProveLocalatp} is equivalent to
\begin{equation*}
\begin{dcases}
\sum_{1\leq\beta\leq \alpha}(3^\beta-1)r^{(3)}_{3^\beta}\equiv (3^\alpha-1)c_{3^\alpha}+(3^{\alpha-1}-1)c_{3^{\alpha-1}} \pmod{24}\\
\sum_{1\leq\beta\leq \alpha}(N3^{-\beta}-N)r^{(3)}_{3^\beta}\equiv (N3^{-\alpha}-N)c_{3^\alpha}+(N3^{-\alpha+1}-N)c_{3^{\alpha-1}} \pmod{24}\\
\sum_{\twoscript{\beta\in\{\alpha-1,\alpha\}}{c_{3^\beta}\equiv r^{(3)}_{3^\beta}+1\bmod2}}\beta\equiv \sum_{1\leq\beta<\alpha-1,\,r^{(3)}_{3^\beta}\equiv1\bmod2}\beta \pmod 2.
\end{dcases}
\end{equation*}
As in the above case, by considering the subcases $2\mid \alpha$ and $2\nmid\alpha$ separately and splitting the congruences modulo $24$ to congruences modulo $3$ and $8$, one can verify that this system is solvable. The case $p=3$, $v_3(N)=1$ is obvious. Thereby we have shown \eqref{eq:toProveLocalatp} which concludes the proof.
\end{proof}

To state the main theorem of this section, let us introduce the system of congruences
\begin{equation}
\label{eq:congruencesZN}
\begin{dcases}
\sum_{n\in\mathscr{B}_N}(n-1)c_n\equiv0 \pmod{24}\\
\sum_{n\in\mathscr{B}_N}(Nn^{-1}-N)c_n\equiv0 \pmod{24}\\
c_p\equiv0\pmod{2},\,\forall p\mid N,p\neq2,3\\
\sum_{2^\beta\in\mathscr{B}_N\colon c_{2^\beta}\equiv1\bmod2}\beta\equiv0\pmod{2},\quad \sum_{3^\beta\in\mathscr{B}_N\colon c_{3^\beta}\equiv1\bmod2}\beta\equiv0\pmod{2}.
\end{dcases}
\end{equation}
Then the element \eqref{eq:elementSNrepr} is the zero in $\mathscr{S}(N)$ if and only if $(c_n)_{n\in\mathscr{B}_N}$ satisfies \eqref{eq:congruencesZN} according to the definition of $Z(N)$.
\begin{deff}
\label{deff:DeltaSequence}
Let $N\in\numgeq{Z}{2}$. For any ordering $b_1,b_2,\dots,b_t$ of $\mathscr{B}_N$ (which means $\mathscr{B}_N=\{b_1,b_2,\dots,b_t\}$ and $b_i\neq b_j$ for $i\neq j$), we define a sequence of positive integers $\Delta_1,\Delta_2,\dots,\Delta_t$ as follows:
\begin{equation*}
\Delta_i=\min\{1\leq m\leq 24\colon\exists(c_n)_{n\in\mathscr{B}_N}\text{ s.t. \eqref{eq:congruencesZN} holds and }c_{b_j}=0\,(j<i),\,c_{b_i}=m\}.
\end{equation*}
\end{deff}
Note that $\Delta_i$ is well defined since the set in the right-hand side always contains $24$. In addition, be careful that the definition \emph{does not} imply $\sum_{1\leq j\leq t}\Delta_{j}e_{j}-\left(\sum_{1\leq j\leq t}\Delta_{j}\right)e_1\in Z(N)$ unless $\abs{\mathscr{B}_N}=1$.
\begin{thm}
\label{thm:SNrepr}
With the notation of Definition \ref{deff:DeltaSequence}, a complete system of representatives of $\mathscr{S}(N)$ is given by
\begin{equation}
\label{eq:SNrepr}
\sum_{1\leq j\leq t}c_{b_j}e_{b_j}-\left(\sum_{1\leq j\leq t}c_{b_j}\right)e_1,\quad 0\leq c_{b_j}<\Delta_j.
\end{equation}
\end{thm}
\begin{proof}
Let $\mathbf{r}+Z(N)\in\mathscr{S}(N)$ be arbitrary. According to Lemma \ref{lemm:SNsimpleForm} we can assume that $\mathbf{r}=\sum_{n\in\mathscr{B}_N}r_ne_n-\left(\sum_{n\in\mathscr{B}_N}r_n\right)e_1$ with $r_n\in\numZ$. By the definition of $\Delta_1$, there exists $\mathbf{d}_1=\sum_{n\in\mathscr{B}_N}d_ne_n-\left(\sum_{n\in\mathscr{B}_N}d_n\right)e_1\in Z(N)$ with $d_{b_1}=\Delta_1$. Hence subtracting a multiple of $\mathbf{d}_1$ from $\mathbf{r}$ we may assume that $0\leq r_{b_1}<\Delta_1$. Inductively, suppose we have adjusted $\mathbf{r}$ such that $0\leq r_{b_j}<\Delta_j$ for $j=1,2,\dots,J$. Let $\mathbf{d}_{J+1}=\sum_{n\in\mathscr{B}_N}d_ne_n-\left(\sum_{n\in\mathscr{B}_N}d_n\right)e_1\in Z(N)$ with $d_{b_j}=0$ ($j\leq J$) and $d_{b_{J+1}}=\Delta_{J+1}$. This exists by the definition of $\Delta_{J+1}$. Subtracting a multiple of $\mathbf{d}_{J+1}$ from $\mathbf{r}$ we can assume that $0\leq r_{b_j}<\Delta_j$ for $j=1,2,\dots,J+1$. Therefore by induction $\mathbf{r}$ takes the form of \eqref{eq:SNrepr} modulo $Z(N)$. We have proved that each coset in $\mathscr{S}(N)$ can be represented by an element in \eqref{eq:SNrepr}.

Let $\mathbf{r}_1=\sum_{1\leq j\leq t}c_{b_j}e_{b_j}-\left(\sum_{1\leq j\leq t}c_{b_j}\right)e_1$ and $\mathbf{r}_2=\sum_{1\leq j\leq t}d_{b_j}e_{b_j}-\left(\sum_{1\leq j\leq t}d_{b_j}\right)e_1$ be two different representatives of the form \eqref{eq:SNrepr}. Let $1\leq J \leq t$ satisfy $c_{b_j}=d_{b_j}$ for $j<J$ and $c_{b_J}\neq d_{b_J}$. Without loss of generality assume that $c_{b_J}<d_{b_J}$. It follows that $0<d_{b_J}-c_{b_J}<\Delta_J$. Then $\mathbf{r}_2-\mathbf{r}_1\not\in Z(N)$ by the definition of $\Delta_J$. We have proved that the vectors in \eqref{eq:SNrepr} are pairwise inequivalent modulo $Z(N)$, which concludes the whole proof.
\end{proof}
\begin{coro}
\label{coro:AllEtaQuoChars}
Let $N\in\numgeq{Z}{1}$ and $k\in\frac{1}{2}\numZ$. Then the characters \eqref{eq:charEtaQuotient} of eta-quotients of level $N$, weight $k$ and cover index $2$ are in one-to-one correspondence with the vectors
\begin{equation*}
\sum_{1\leq j\leq t}r_{b_j}e_{b_j}+\left(2k-\sum_{1\leq j\leq t}r_{b_j}\right)e_1,\quad 0\leq r_{b_j}<\Delta_j.
\end{equation*}
\end{coro}
\begin{proof}
According to the discussion followed by Remark \ref{rema:SNsubgroup}, the characters of eta-quotients of level $N$, weight $k$ and cover index $2$ can be regarded as the set $\mathscr{S}(N,k)$ and the map $\mathscr{S}(N)\to\mathscr{S}(N,k)$, $\mathbf{r}+Z(N)\mapsto\mathbf{r}+2ke_1+Z(N)$ is a bijection. The corollary follows from this and Theorem \ref{thm:SNrepr}.
\end{proof}

In the remainder of this section, we give examples of concrete sequences $(\Delta_i)_{1\leq i\leq t}$ ($t=\abs{\mathscr{B}_N}$) for some special $N$.
\begin{examp}
\label{examp:classifyChar1}
Let $N=p^\alpha$ where $p\geq 5$ is a prime and $\alpha\in\numgeq{Z}{1}$. Then $\mathscr{B}_N=\{p\}$ and the sequence $(\Delta_i)_{1\leq i\leq t}$ ($t=\abs{\mathscr{B}_N}$) contains a single element $\Delta_1$. The system of congruences \eqref{eq:congruencesZN} is equivalent to $(p-1)c_p\equiv0\bmod{24}$ and $c_p\equiv0\bmod{2}$ in this case. Thus, $\Delta_1=\frac{24}{(12,p-1)}$ and a complete system of representatives of $\mathscr{S}(p^\alpha)$ is given by
\begin{equation*}
r_{p}(e_{p}-e_1),\quad 0\leq r_{p}<\frac{24}{(12,p-1)}.
\end{equation*}
As a consequence, for any $k\in\frac{1}{2}\numZ$ there are exactly $\frac{24}{(12,p-1)}$ characters $\chi\colon\widetilde{\Gamma_0(p^\alpha)}\to\numC^\times$ of eta-quotients of level $p^\alpha$, weight $k$ and cover index $2$. It should be noted that, the character of an eta-quotient, say, of cover index $4$ may also be regarded as a character on $\widetilde{\Gamma_0(p^\alpha)}$ (see the third fact of Remark \ref{rema:mulSysProp}). For instance, the character $\chi\colon\Dcover{\Gamma_0(p^\alpha)}{8}\to\numC^\times$ of $\eta(\tau)^{1/4}\eta(p\tau)^{3/4}$ descends to a character on $\widetilde{\Gamma_0(p^\alpha)}=\Dcover{\Gamma_0(p^\alpha)}{2}$. Such characters are not in the list above.
\end{examp}

\begin{examp}
\label{examp:classifyChar2}
Let $N=2^\alpha$ with $\alpha \geq3$. Then $\mathscr{B}_N=\{2^{\alpha-2},2^{\alpha-1},2^{\alpha}\}$ which is arranged in the increasing order. We need to calculate the sequence $\Delta_1,\Delta_2,\Delta_3$. The system of congruences \eqref{eq:congruencesZN} is equivalent to
\begin{equation*}
\begin{dcases}
(2^{\alpha-2}-1)c_{2^{\alpha-2}}+(2^{\alpha-1}-1)c_{2^{\alpha-1}}+(2^{\alpha}-1)c_{2^{\alpha}}\equiv0\pmod{24}\\
(4-2^{\alpha})c_{2^{\alpha-2}}+(2-2^{\alpha})c_{2^{\alpha-1}}+(1-2^{\alpha})c_{2^{\alpha}}\equiv0\pmod{24}\\
\sum_{\twoscript{\beta\in\{\alpha-2,\alpha-1,\alpha\}}{c_{2^\beta}\equiv1\bmod 2}}\beta\equiv0\pmod{2}.
\end{dcases}
\end{equation*}
One can verify directly that
\begin{equation*}
(\Delta_1,\Delta_2,\Delta_3)=\begin{dcases}
(2,24,8)  &\text{ if }\alpha\geq3\text{ and }2\mid\alpha\\
(2,8,24)  &\text{ if }\alpha\geq3\text{ and }2\nmid\alpha.
\end{dcases}
\end{equation*}
Therefore, a complete system of representatives of $\mathscr{S}(2^\alpha)$ with $\alpha\geq3$ is given by
\begin{equation*}
r_{2^{\alpha-2}}e_{2^{\alpha-2}}+r_{2^{\alpha-1}}e_{2^{\alpha-1}}+r_{2^{\alpha}}e_{2^{\alpha}}-(r_{2^{\alpha-2}}+r_{2^{\alpha-1}}+r_{2^{\alpha}})e_1,
\end{equation*}
where $r_{\alpha-2}\in\{0,1\}$ and $0\leq r_{2^{\alpha-1}}<24$, $0\leq r_{2^{\alpha}}<8$ if $2\mid\alpha$; $0\leq r_{2^{\alpha-1}}<8$, $0\leq r_{2^{\alpha}}<24$ if $2\nmid\alpha$. As a consequence, for any $k\in\frac{1}{2}\numZ$ there are exactly $384$ characters $\chi\colon\widetilde{\Gamma_0(2^\alpha)}\to\numC^\times$ of eta-quotients of level $2^\alpha$, weight $k$ and cover index $2$ when $\alpha\geq3$. In addition, it is immediate that for $N=2^2$ we have $\mathscr{B}_{4}=\{2^1,2^2\}$ and $(\Delta_1,\Delta_2)=(24,8)$ and for $N=2^1$ we have $\mathscr{B}_{2}=\{2^1\}$ and $\Delta_1=24$. Thus, there are exactly $192$ ($24$ respectively) characters of eta-quotients of level $4$ ($2$ respectively), cover index $2$ and any fixed weight $\in\frac{1}{2}\numZ$.
\end{examp}

\begin{examp}
\label{examp:classifyChar3}
Let $N=3^\alpha$ with $\alpha \geq2$. Then $\mathscr{B}_N=\{3^{\alpha-1},3^{\alpha}\}$ which is arranged in the increasing order. We need to calculate the sequence $\Delta_1,\Delta_2$. The system of congruences \eqref{eq:congruencesZN} is equivalent to
\begin{equation*}
\begin{dcases}
(3^{\alpha-1}-1)c_{3^{\alpha-1}}+(3^{\alpha}-1)c_{3^{\alpha}}\equiv0\pmod{24}\\
(3-3^{\alpha})c_{3^{\alpha-1}}+(1-3^{\alpha})c_{3^{\alpha}}\equiv0\pmod{24}\\
\sum_{\twoscript{\beta\in\{\alpha-1,\alpha\}}{c_{3^\beta}\equiv1\bmod 2}}\beta\equiv0\pmod{2}.
\end{dcases}
\end{equation*}
One can verify directly that
\begin{equation*}
(\Delta_1,\Delta_2)=\begin{dcases}
(12,3)  &\text{ if }\alpha\geq2\text{ and }2\mid\alpha\\
(3,12)  &\text{ if }\alpha\geq2\text{ and }2\nmid\alpha.
\end{dcases}
\end{equation*}
Therefore, a complete system of representatives of $\mathscr{S}(3^\alpha)$ with $\alpha\geq2$ is given by
\begin{equation*}
r_{3^{\alpha-1}}e_{3^{\alpha-1}}+r_{3^{\alpha}}e_{3^{\alpha}}-(r_{3^{\alpha-1}}+r_{3^{\alpha}})e_1,
\end{equation*}
where $0\leq r_{3^{\alpha-1}}<12$, $0\leq r_{3^{\alpha}}<3$ if $2\mid\alpha$; $0\leq r_{3^{\alpha-1}}<3$, $0\leq r_{3^{\alpha}}<12$ if $2\nmid\alpha$. As a consequence, for any $k\in\frac{1}{2}\numZ$ there are exactly $36$ characters $\chi\colon\widetilde{\Gamma_0(3^\alpha)}\to\numC^\times$ of eta-quotients of level $3^\alpha$, weight $k$ and cover index $2$ when $\alpha\geq2$. In addition, it is direct that for $N=3^1$ we have $\mathscr{B}_{3}=\{3^1\}$ and $\Delta_1=12$. Thus, there are exactly $12$ characters of eta-quotients of level $3$, cover index $2$ and any fixed weight $\in\frac{1}{2}\numZ$.
\end{examp}

\begin{examp}
\label{examp:Nis4palpha}
Let $N=4p^\alpha$ where $p\geq5$ is a prime and $\alpha\in\numgeq{Z}{1}$. Then $\mathscr{B}_N=\{2^1,2^2,p\}$ which is arranged in the increasing order. We need to calculate the sequence $\Delta_1,\Delta_2,\Delta_3$. The system of congruences \eqref{eq:congruencesZN} is equivalent to
\begin{equation*}
\begin{dcases}
c_2+3c_4+(p-1)c_p\equiv0\pmod{24}\\
-2p^\alpha c_2-3p^\alpha c_4+(4p^{\alpha-1}-4p^\alpha)c_p\equiv0\pmod{24}\\
c_2,\,c_p\equiv0\pmod{2}.
\end{dcases}
\end{equation*}
Using elementary number theory one can show that
\begin{equation*}
\Delta_1=2\cdot(12,p-1),\quad \Delta_2=8,\quad \Delta_3=\frac{24}{(12,p-1)}.
\end{equation*}
Therefore, a complete system of representatives of $\mathscr{S}(4p^\alpha)$ with $p\geq5$ and $\alpha\geq1$ is given by
\begin{equation*}
r_{2}e_{2}+r_{4}e_{4}+r_pe_p-(r_2+r_4+r_p)e_1
\end{equation*}
with
\begin{equation*}
0\leq r_2<2\cdot(12,p-1),\quad 0\leq r_{4}<8,\quad 0\leq r_{p}<\frac{24}{(12,p-1)}.
\end{equation*}
As a consequence, for any $k\in\frac{1}{2}\numZ$ there are exactly $384$ characters $\chi\colon\widetilde{\Gamma_0(4p^\alpha)}\to\numC^\times$ of eta-quotients of level $4p^\alpha$, weight $k$ and cover index $2$ when $p\geq5$ and $\alpha\geq1$.
\end{examp}

\begin{examp}
\label{examp:classifyChar5}
Let $N=p_1^{\alpha_1}p_2^{\alpha_2}$ where $p_1,p_2\geq5$ are distinct primes and $\alpha_1,\alpha_2\in\numgeq{Z}{1}$. Then $\mathscr{B}_N=\{p_1,p_2\}$. We need to calculate the sequence $\Delta_1,\Delta_2$. The system of congruences \eqref{eq:congruencesZN} is equivalent to
\begin{equation*}
\begin{dcases}
(p_1-1)c_{p_1}+(p_2-1)c_{p_2}\equiv0\pmod{24}\\
(p_1^{\alpha_1-1}p_2^{\alpha_2}-p_1^{\alpha_1}p_2^{\alpha_2})c_{p_1}+(p_1^{\alpha_1}p_2^{\alpha_2-1}-p_1^{\alpha_1}p_2^{\alpha_2})c_{p_2}\equiv0\pmod{24}\\
c_{p_1},\,c_{p_2}\equiv0\pmod{2}.
\end{dcases}
\end{equation*}
Using elementary number theory one can show that
\begin{align*}
\Delta_1&=\frac{24\cdot(12,p_2-1)}{(12,(p_1-1)(p_2-1))}\cdot\left(\frac{12\cdot(12,p_2-1)}{(12,(p_1-1)(p_2-1))},p_1-p_2\right)^{-1},\\
\Delta_2&=\frac{24}{(12,p_2-1)}.
\end{align*}
The expression for $\Delta_1$ seems to be complicated but it actually depends only on $p_1\bmod{12}$ and $p_2\bmod{12}$. For instance, if $p_2\equiv1\bmod{12}$, then $\Delta_1$ is simplified to $\frac{24}{(12,p_1-1)}$. A complete system of representatives of $\mathscr{S}(p_1^{\alpha_1}p_2^{\alpha_2})$ is given by
\begin{equation*}
r_{p_1}e_{p_1}+r_{p_2}e_{p_2}-(r_{p_1}+r_{p_2})e_1
\end{equation*}
with
\begin{equation*}
0\leq r_{p_1}<\Delta_1,\quad 0\leq r_{p_2}<\Delta_2.
\end{equation*}
As a consequence, for any $k\in\frac{1}{2}\numZ$ there are exactly $\Delta_1\Delta_2$ characters $\chi\colon\widetilde{\Gamma_0(p_1^{\alpha_1}p_2^{\alpha_2})}\to\numC^\times$ of eta-quotients of level $p_1^{\alpha_1}p_2^{\alpha_2}$, weight $k$ and cover index $2$ when $p_1, p_2\geq5$ are distinct primes and $\alpha_1,\alpha_2\geq1$.
\end{examp}

\section{Dimension formulas for weights $1/2$, $1$, $3/2$ and small levels}
\label{sec:Dimension formulas for weight 1/2, 1, 3/2 and small levels}
In this section, we give values of all $\dim_\numC M_{k}(\Gamma_0(N),\chi)$ that can be calculated by Theorem \ref{thm:main} (throughout this section, $t$ in this theorem always takes the value $0$) for $k=1/2$, $1$ and $3/2$, $N\in\numgeq{Z}{1}$ and $\chi$ being the character of an eta-quotient of level $N$, weight $k$ and cover index $2$. There are only finitely many such formulas since the right-hand side of \eqref{eq:mainThmCondtion2} tends to $2$ when $N\to+\infty$ and the set of $\chi$ is finite for each $N$ and $k$ (see Corollary \ref{coro:AllEtaQuoChars}). The following lemma gives an upper bound of $N$ that should be considered.
\begin{lemm}
\label{lemm:NupperBound}
Let $N$ be a positive integer and let $m$, $\varepsilon_2$, $\varepsilon_3$ and $\varepsilon_\infty$ be defined in \eqref{eq:mGamma0N}, \eqref{eq:e2Gamma0N}, \eqref{eq:e3Gamma0N} and \eqref{eq:einftyGamma0N} respectively. If $N\geq1024$, we have
\begin{equation*}
\varepsilon_2<\frac{m}{96},\quad\varepsilon_3<\frac{m}{128},\quad\varepsilon_\infty\leq\frac{m}{32}.
\end{equation*}
Consequently, in Theorem \ref{thm:main}, if $k\leq3/2$, $t=0$ and $N\geq1024$, then \eqref{eq:mainThmCondtion2} does not hold.
\end{lemm}
For $\varepsilon_\infty$, a stronger estimate\footnote{This is the best estimate if we want it to hold for all $N$ since it becomes equality when $N$ is a square.} is $\varepsilon_\infty\leq m\cdot N^{-1/2}$ with $N\in\numgeq{Z}{1}$ which we will prove below.
\begin{proof}
Suppose $N=\prod_{1\leq j\leq s}p_j^{\alpha_j}$ where $p_1<p_2<\dots<p_s$ are primes, $s\in\numgeq{Z}{0}$ and $\alpha_j\in\numgeq{Z}{1}$. Note that for $p=p_j$, $\alpha=\alpha_j$ we have
\begin{equation*}
\frac{p^{[\alpha/2]}+p^{[\alpha/2-1/2]}}{p^{\alpha}+p^{\alpha-1}}=\begin{dcases}
\frac{2}{p^{\alpha/2}(p^{1/2}+p^{-1/2})} &\text{ if } 2\nmid\alpha,\\
\frac{1}{p^{\alpha/2}} &\text{ if } 2\mid\alpha.
\end{dcases}
\end{equation*}
Thereby, according to \cite[Corollary 6.3.24(b)]{CS17}
\begin{equation}
\label{eq:einftymN}
\frac{\varepsilon_\infty}{m}=\prod_{1\leq j\leq s}\frac{p_j^{[\alpha_j/2]}+p_j^{[\alpha_j/2-1/2]}}{p_j^{\alpha_j}+p_j^{\alpha_j-1}}\leq\prod_{1\leq j \leq s}\frac{1}{p_j^{\alpha_j/2}}=N^{-1/2}.
\end{equation}

Now we begin to prove the inequalities for $\varepsilon_2$ and $\varepsilon_3$, so suppose $N\geq1024$. If $s=1,2,3$, then $\varepsilon_2/m\leq2^3/1025<1/96$, $\varepsilon_3/m\leq2^3/1025<1/128$. If $s\geq4$, we consider two subcases: $\varepsilon_2=0$ and $\varepsilon_2\neq0$. If $\varepsilon_2=0$, obviously $\varepsilon_2/m<1/96$. If $\varepsilon_2\neq0$, then $\legendre{-4}{p_j}\neq-1$ and hence $p_j=2$ or $p_j\equiv1\bmod 4$ for $1\leq j\leq s$. If $p_1\neq2$, then $p_1^{\alpha_1}\geq 5$, $p_2^{\alpha_2}\geq 13$ and $p_j^{\alpha_j}\geq 17$ for $j\geq3$. Therefore
\begin{equation*}
\frac{\varepsilon_2}{m}\leq\frac{2}{5+1}\cdot\frac{2}{13+1}\cdot\left(\frac{2}{17+1}\right)^{s-2}<\frac{1}{96}.
\end{equation*}
If $p_1=2$, then $p_2^{\alpha_2}\geq 5$, $p_3^{\alpha_3}\geq 13$ and $p_j^{\alpha_j}\geq 17$ for $j\geq4$. Therefore
\begin{equation*}
\frac{\varepsilon_2}{m}\leq\frac{1}{2+1}\cdot\frac{2}{5+1}\cdot\frac{2}{13+1}\cdot\left(\frac{2}{17+1}\right)^{s-3}<\frac{1}{96}.
\end{equation*}
This proves the inequality for $\varepsilon_2$. To prove the inequality for $\varepsilon_3$ when $s\geq4$, we consider the subcases $\varepsilon_3=0$ and $\varepsilon_3\neq0$ separately and proceed as in the proof for $\varepsilon_2$.

Finally, with the notation of Theorem \ref{thm:main} suppose that $k\leq3/2$, $t=0$ and $N\geq1024$. Assume by contradiction that \eqref{eq:mainThmCondtion2} holds, then
\begin{equation*}
\frac{3}{2} \geq k > 2-\frac{6}{m}\varepsilon_2-\frac{8}{m}\varepsilon_3-\frac{12}{m}\varepsilon_\infty>2-\frac{6}{96}-\frac{8}{128}-\frac{12}{32}
\end{equation*}
which is absurd. Hence \eqref{eq:mainThmCondtion2} does not hold.
\end{proof}

The algorithm for presenting all $\dim_\numC M_{k}(\Gamma_0(N),\chi)$ described above is as follows:
\begin{enumerate}
  \item Let $k$ range over $\{1/2,1,3/2\}$ and let $N$ range over the positive integers $\leq1023$. We do not miss anything according to Lemma \ref{lemm:NupperBound}.
  \item Given $k,N$, we arrange $\mathscr{B}_N$ in any order and calculate the sequence $\Delta_1,\dots,\Delta_t$ (see Definition \ref{deff:DeltaSequence}).
  \item Let $\mathbf{r}=(r_n)_{n\mid N}$ range over the vectors described in Corollary \ref{coro:AllEtaQuoChars}. According to this corollary, the corresponding $\chi_\mathbf{r}$ ranges over exactly all characters of eta-quotients of level $N$, weight $k$ and cover index $2$.
  \item Given $k,N$ and $\mathbf{r}$, we calculate $\mathbf{x}=(x_c)_{c\mid N}$ via \eqref{eq:rntoxc}.
  \item We determine whether \eqref{eq:mainThmCondtion2} holds. If it holds, then we calculate $\dim_\numC M_{k}(\Gamma_0(N),\chi_{\mathbf{r}})$ by \eqref{eq:dimFormula} where $t=0$. If it does not hold, $\dim_\numC M_{k}(\Gamma_0(N),\chi_{\mathbf{r}})$ can not be calculated by Theorem \ref{thm:main}. Nevertheless we can obtain an upper bound and a lower bound of $\dim_\numC M_{k}(\Gamma_0(N),\chi_{\mathbf{r}})$ (see Remark \ref{rema:uboundlbound}).
\end{enumerate}
\begin{rema}
\label{rema:rangeN}
For a fixed $k$, one can shrink the range of $N$ by picking out those $N\leq1023$ such that $k>2-\frac{6}{m}\varepsilon_2-\frac{8}{m}\varepsilon_3-\frac{12}{m}\varepsilon_\infty$. Hence, for $k=1/2$, $N$ need only to range over $\{1\leq N\leq21\}\cup\{24,25,27,32,36,49,50\}$; for $k=1$, $N$ need only to range over $\{1\leq N \leq 22\}\cup\{24\leq N \leq 32\}\cup\{34,36,37,39,40,45,48,49,50,54,64,72,75,81,98,100,121,169\}$; for $k=3/2$, $N$ need only to range over a subset of $\{1\leq N \leq529\}$.
\end{rema}
All the above steps can be easily implemented in any computer algebra system. We program them using SageMath. See Appendix \ref{apx:Usage of SageMath code} for the usage of the code.

\subsection{The weight $1/2$}
\label{subsec:weight1/2}
The dimensions obtained by the above algorithm are summarized in Table \ref{table:wt1/2}.
\begin{table}[ht]
\centering
\caption{For each $N$, $a$ means the total number of characters (see Corollary \ref{coro:AllEtaQuoChars}) and $v$ means the number of characters $\chi$ such that $\dim_\numC M_{1/2}(\Gamma_0(N),\chi)$ can be computed using Theorem \ref{thm:main}, $d_j$ means the number of spaces of dimension $j$ for $j=0,1,2$. \label{table:wt1/2}}
\begin{tabular}{llllll|llllll}
\toprule
$N$ & $a$ & $v$ & $d_0$ & $d_1$ & $d_2$ & $N$ & $a$ & $v$ & $d_0$ & $d_1$ & $d_2$ \\
\midrule
$1$ & $1$ & $1$ & $0$ & $1$ & $0$ & $2$ & $24$ & $24$ & $20$ & $4$ & $0$\\
$3$ & $12$ & $12$ & $10$ & $2$ & $0$ & $4$ & $192$ & $146$ & $136$ & $10$ & $0$\\
$5$ & $6$ & $6$ & $4$ & $2$ & $0$ & $6$ & $288$ & $160$ & $148$ & $12$ & $0$\\
$7$ & $4$ & $4$ & $2$ & $2$ & $0$ & $8$ & $384$ & $214$ & $198$ & $16$ & $0$\\
$9$ & $36$ & $18$ & $15$ & $3$ & $0$ & $10$ & $48$ & $32$ & $24$ & $8$ & $0$\\
$12$ & $1152$ & $318$ & $286$ & $32$ & $0$ & $13$ & $2$ & $2$ & $0$ & $2$ & $0$\\
$16$ & $384$ & $127$ & $108$ & $18$ & $1$ & $18$ & $288$ & $56$ & $42$ & $12$ & $2$\\
$20$ & $384$ & $6$ & $0$ & $6$ & $0$ & $24$ & $2304$ & $24$ & $0$ & $24$ & $0$\\
$25$ & $6$ & $3$ & $1$ & $1$ & $1$ & $27$ & $36$ & $2$ & $0$ & $2$ & $0$\\
$32$ & $384$ & $10$ & $0$ & $8$ & $2$ & $36$ & $1152$ & $12$ & $0$ & $6$ & $6$\\
$49$ & $4$ & $1$ & $0$ & $0$ & $1$ & $50$ & $48$ & $4$ & $0$ & $0$ & $4$\\
\bottomrule
\end{tabular}
\end{table}
\begin{examp}
Let us consider the case $N=50$. We have $\mathscr{B}_{50}=\{2,5\}$ and $(\Delta_1,\Delta_2)=(8,6)$ as one can verify directly. Thus, there are totally $48$ characters of eta-quotients of level $50$, weight $1/2$ and cover index $2$. They are exactly the characters on $\Dcover{\Gamma_0(50)}{2}$ of the following eta-quotients:
\begin{equation*}
\eta(\tau)^{1-r_2-r_5}\eta(2\tau)^{r_2}\eta(5\tau)^{r_5},\quad 0\leq r_2<8,\,0\leq r_5<6.
\end{equation*}
Among the $48$ characters, for precisely the four characters $\chi$ given by
\begin{equation}
\label{eq:wt1/2gamma50fourChar}
(r_2,r_5)=(0,0),\,(1,0),\,(2,0),\,(7,4)
\end{equation}
one can compute $\dim_\numC M_{1/2}(\Gamma_0(50),\chi)$ using Theorem \ref{thm:main} (that is, \eqref{eq:mainThmCondtion2} holds with $t=0$) and the four dimensions all equal $2$. For the other $44$ characters $\chi$ our method only gives that $0\leq\dim_\numC M_{1/2}(\Gamma_0(50),\chi)\leq4$ (see Remark \ref{rema:uboundlbound}). Let $\chi_1,\chi_2,\chi_3,\chi_4$ denote respectively the characters given by \eqref{eq:wt1/2gamma50fourChar}. Then according to \eqref{eq:divEtaQuotient} and Lemma \ref{lemm:whenCharSame} we have
\begin{align*}
M_{1/2}(\Gamma_0(50),\chi_1)&=\numC\eta(\tau)\oplus\numC\eta(25\tau),\\
M_{1/2}(\Gamma_0(50),\chi_2)&=\numC\eta(2\tau)\oplus\numC\eta(50\tau),\\
M_{1/2}(\Gamma_0(50),\chi_3)&=\numC\eta(\tau)^{-1}\eta(2\tau)^2\oplus\numC\eta(25\tau)^{-1}\eta(50\tau)^{2}.
\end{align*}
Finding explicit generators of $M_{1/2}(\Gamma_0(50),\chi_4)$ seems to be a bit harder since $\eta(\tau)^{-10}\eta(2\tau)^7\eta(5\tau)^4$ is only a weakly holomorphic modular form. However, we have $7e_2+4e_5-11e_1\sim-e_2+e_1$ module $Z(50)$ which can be verified by \eqref{eq:congruencesZN}. It follows that
\begin{equation*}
M_{1/2}(\Gamma_0(50),\chi_4)=\numC\eta(\tau)^{2}\eta(2\tau)^{-1}\oplus\numC\eta(25\tau)^{2}\eta(50\tau)^{-1}.
\end{equation*}
Finally note that $M_{1/2}(\Gamma_0(50),\chi_j)=S_{1/2}(\Gamma_0(50),\chi_j)$ for $j=1,2$, while the Eisenstein subspace of $M_{1/2}(\Gamma_0(50),\chi_j)$ is nontrivial for $j=3,4$.
\end{examp}

\begin{examp}
We consider another example $N=20$ in which there are totally $384$ characters. They have been described in Example \ref{examp:Nis4palpha} with $p=5$, $\alpha=1$. We have $\mathscr{B}_{20}=\{2,4,5\}$ and $(\Delta_1,\Delta_2,\Delta_3)=(8,8,6)$. Therefore the characters are exactly those on $\Dcover{\Gamma_0(20)}{2}$ of the following eta-quotients:
\begin{equation*}
\eta(\tau)^{1-r_2-r_4-r_5}\eta(2\tau)^{r_2}\eta(4\tau)^{r_4}\eta(5\tau)^{r_5},\quad 0\leq r_2<8,\,0\leq r_4<8,\,0\leq r_5<6.
\end{equation*}
\end{examp}
Among the $384$ characters, for precisely the six characters $\chi$ given by
\begin{equation}
\label{eq:wt1/2gamma20sixChar}
(r_2,r_4,r_5)=(1,6,1),\,(3,0,5),\,(3,2,5),\,(5,6,0),\,(7,0,4),\,(7,2,4)
\end{equation}
one can compute $\dim_\numC M_{1/2}(\Gamma_0(20),\chi)$ using Theorem \ref{thm:main} and the six dimensions all equal $1$. For the other $378$ characters $\chi$ our method only gives that $0\leq\dim_\numC M_{1/2}(\Gamma_0(20),\chi)\leq2$ (see Remark \ref{rema:uboundlbound}). Let $\chi_1,\chi_2,\chi_3,\chi_4,\chi_5,\chi_6$ denote respectively the characters given by \eqref{eq:wt1/2gamma20sixChar}. One can verify using \eqref{eq:divEtaQuotient} and Lemma \ref{lemm:whenCharSame} that
\begin{align*}
M_{1/2}(\Gamma_0(20),\chi_1)&=\numC\eta(5\tau)^{-2}\eta(10\tau)^5\eta(20\tau)^{-2},\quad &M_{1/2}(\Gamma_0(20),\chi_2)&=\numC\eta(5\tau)^{2}\eta(10\tau)^{-1},\\
M_{1/2}(\Gamma_0(20),\chi_3)&=\numC\eta(10\tau)^{-1}\eta(20\tau)^2,\quad &M_{1/2}(\Gamma_0(20),\chi_4)&=\numC\eta(\tau)^{-2}\eta(2\tau)^5\eta(4\tau)^{-2},\\
M_{1/2}(\Gamma_0(20),\chi_5)&=\numC\eta(\tau)^{2}\eta(2\tau)^{-1},\quad &M_{1/2}(\Gamma_0(20),\chi_6)&=\numC\eta(2\tau)^{-1}\eta(4\tau)^{2}.
\end{align*}

\begin{rema}
There are totally $188$ nonzero spaces $M_{1/2}(\Gamma_0(N),\chi)$ whose dimensions can be computed by Theorem \ref{thm:main} according to Table \ref{table:wt1/2}. Note that we restrict ourselves to the characters $\chi$ of eta-quotients of cover index $2$, so the dimension formulas summarized in Table \ref{table:wt1/2} are not direct consequences of the well known theorem of Serre and Stark \cite{SS77} that gives an explicit basis of $M_{1/2}(\Gamma_0(N),\chi_1\chi)$ where $4 \mid N$, $\chi_1$ is the character of Euler theta function $\theta(\tau)=\eta(\tau)^{-2}\eta(2\tau)^5\eta(4\tau)^{-2}$ and $\chi$ is the character induced by any even Dirichlet character $(\numZ/N\numZ)^\times\to\numC^\times$. For instance, when $N=12$, there are two even Dirichlet characters and hence the theorem of Serre and Stark gives explicit bases of two spaces $M_{1/2}(\Gamma_0(12),\chi)$. As a comparison, there are totally $1152$ characters of eta-quotients of weight $1/2$, level $12$ and cover index $2$ among which the dimensions of $318$ spaces can be computed by Theorem \ref{thm:main}---$286$ spaces are zero-dimensional and $32$ spaces are one-dimensional.
\end{rema}
\begin{rema}
Zagier observed that there are exactly $14$ primitive (which means the greatest common divisor of all the $n$ with $r_n\neq 0$ is equal to $1$) eta-quotients of integral exponents that are modular forms of weight $1/2$. This and the similar conjectures on higher weights were proved by Mersmann in his Master's thesis. See \cite{Bha17} for a simpler proof. It seems that all the generators of the $188$ nonzero spaces $M_{1/2}(\Gamma_0(N),\chi)$ whose dimensions can be computed by Theorem \ref{thm:main} (see Table \ref{table:wt1/2}) take the form $f(m\tau)$ where $f$ is one of the $14$ primitive eta-quotients and $m$ is a suitable positive integer. This is verified by the SageMath program for $N=4,20,50$. See Appendix \ref{apx:Usage of SageMath code} for how to find the code doing this.
\end{rema}


\subsection{The weight $1$}
\label{subsec:weight1}
The dimensions obtained by the algorithm described above are summarized in Table \ref{table:wt1}. See Appendix \ref{apx:Usage of SageMath code} for the SageMath code that generates this kind of tables. Note that the characters are originally defined on $\Dcover{\Gamma_0(N)}{2}$ since we are considering eta-quotients of cover index $2$ but they actually descend to characters on the matrix groups $\Gamma_0(N)$ since the weight is $1$.
\begin{longtable}{llllllllllll}
\caption{For each $N$, $a$ means the total number of characters (see Corollary \ref{coro:AllEtaQuoChars}) and $v$ means the number of characters $\chi$ such that $\dim_\numC M_{1}(\Gamma_0(N),\chi)$ can be computed using Theorem \ref{thm:main}, $d_j$ means the number of spaces of dimension $j$ for $0\leq j \leq 8$. \label{table:wt1}}\\
\toprule
$N$ & $a$ & $v$ & $d_0$ & $d_1$ & $d_2$ & $d_3$ & $d_4$ & $d_5$ & $d_6$ & $d_7$ & $d_8$ \\
\midrule
\endfirsthead
$N$ & $a$ & $v$ & $d_0$ & $d_1$ & $d_2$ & $d_3$ & $d_4$ & $d_5$ & $d_6$ & $d_7$ & $d_8$ \\
\midrule
\endhead
$1$ & $1$ & $1$ & $0$ & $1$ & $0$ & $0$ & $0$ & $0$ & $0$ & $0$ & $0$\\
$2$ & $24$ & $24$ & $17$ & $7$ & $0$ & $0$ & $0$ & $0$ & $0$ & $0$ & $0$\\
$3$ & $12$ & $12$ & $7$ & $5$ & $0$ & $0$ & $0$ & $0$ & $0$ & $0$ & $0$\\
$4$ & $192$ & $173$ & $142$ & $31$ & $0$ & $0$ & $0$ & $0$ & $0$ & $0$ & $0$\\
$5$ & $6$ & $6$ & $3$ & $3$ & $0$ & $0$ & $0$ & $0$ & $0$ & $0$ & $0$\\
$6$ & $288$ & $250$ & $189$ & $60$ & $1$ & $0$ & $0$ & $0$ & $0$ & $0$ & $0$\\
$7$ & $4$ & $4$ & $1$ & $3$ & $0$ & $0$ & $0$ & $0$ & $0$ & $0$ & $0$\\
$8$ & $384$ & $331$ & $247$ & $83$ & $1$ & $0$ & $0$ & $0$ & $0$ & $0$ & $0$\\
$9$ & $36$ & $29$ & $17$ & $11$ & $1$ & $0$ & $0$ & $0$ & $0$ & $0$ & $0$\\
$10$ & $48$ & $46$ & $20$ & $22$ & $4$ & $0$ & $0$ & $0$ & $0$ & $0$ & $0$\\
$11$ & $12$ & $1$ & $0$ & $1$ & $0$ & $0$ & $0$ & $0$ & $0$ & $0$ & $0$\\
$12$ & $1152$ & $947$ & $617$ & $305$ & $24$ & $1$ & $0$ & $0$ & $0$ & $0$ & $0$\\
$13$ & $2$ & $2$ & $0$ & $1$ & $1$ & $0$ & $0$ & $0$ & $0$ & $0$ & $0$\\
$14$ & $96$ & $21$ & $0$ & $20$ & $1$ & $0$ & $0$ & $0$ & $0$ & $0$ & $0$\\
$15$ & $72$ & $18$ & $0$ & $16$ & $2$ & $0$ & $0$ & $0$ & $0$ & $0$ & $0$\\
$16$ & $384$ & $293$ & $166$ & $108$ & $18$ & $1$ & $0$ & $0$ & $0$ & $0$ & $0$\\
$17$ & $6$ & $4$ & $0$ & $4$ & $0$ & $0$ & $0$ & $0$ & $0$ & $0$ & $0$\\
$18$ & $288$ & $217$ & $113$ & $75$ & $25$ & $3$ & $1$ & $0$ & $0$ & $0$ & $0$\\
$19$ & $4$ & $3$ & $0$ & $3$ & $0$ & $0$ & $0$ & $0$ & $0$ & $0$ & $0$\\
$20$ & $384$ & $96$ & $0$ & $72$ & $24$ & $0$ & $0$ & $0$ & $0$ & $0$ & $0$\\
$21$ & $24$ & $14$ & $0$ & $8$ & $6$ & $0$ & $0$ & $0$ & $0$ & $0$ & $0$\\
$22$ & $96$ & $1$ & $0$ & $0$ & $1$ & $0$ & $0$ & $0$ & $0$ & $0$ & $0$\\
$24$ & $2304$ & $658$ & $0$ & $480$ & $176$ & $0$ & $2$ & $0$ & $0$ & $0$ & $0$\\
$25$ & $6$ & $4$ & $1$ & $1$ & $1$ & $1$ & $0$ & $0$ & $0$ & $0$ & $0$\\
$26$ & $48$ & $4$ & $0$ & $0$ & $4$ & $0$ & $0$ & $0$ & $0$ & $0$ & $0$\\
$27$ & $36$ & $15$ & $0$ & $9$ & $5$ & $1$ & $0$ & $0$ & $0$ & $0$ & $0$\\
$28$ & $384$ & $11$ & $0$ & $0$ & $10$ & $1$ & $0$ & $0$ & $0$ & $0$ & $0$\\
$30$ & $576$ & $2$ & $0$ & $0$ & $0$ & $0$ & $2$ & $0$ & $0$ & $0$ & $0$\\
$32$ & $384$ & $146$ & $0$ & $100$ & $36$ & $8$ & $2$ & $0$ & $0$ & $0$ & $0$\\
$36$ & $1152$ & $412$ & $0$ & $198$ & $168$ & $30$ & $15$ & $0$ & $1$ & $0$ & $0$\\
$37$ & $2$ & $1$ & $0$ & $0$ & $1$ & $0$ & $0$ & $0$ & $0$ & $0$ & $0$\\
$40$ & $768$ & $2$ & $0$ & $0$ & $0$ & $0$ & $2$ & $0$ & $0$ & $0$ & $0$\\
$45$ & $72$ & $2$ & $0$ & $0$ & $0$ & $0$ & $2$ & $0$ & $0$ & $0$ & $0$\\
$48$ & $2304$ & $58$ & $0$ & $0$ & $0$ & $16$ & $40$ & $0$ & $2$ & $0$ & $0$\\
$49$ & $4$ & $2$ & $0$ & $0$ & $1$ & $0$ & $1$ & $0$ & $0$ & $0$ & $0$\\
$50$ & $48$ & $26$ & $0$ & $0$ & $19$ & $0$ & $0$ & $7$ & $0$ & $0$ & $0$\\
$54$ & $288$ & $1$ & $0$ & $0$ & $0$ & $0$ & $0$ & $0$ & $1$ & $0$ & $0$\\
$64$ & $384$ & $18$ & $0$ & $0$ & $0$ & $0$ & $16$ & $0$ & $2$ & $0$ & $0$\\
$72$ & $2304$ & $2$ & $0$ & $0$ & $0$ & $0$ & $0$ & $0$ & $0$ & $0$ & $2$\\
$75$ & $72$ & $2$ & $0$ & $0$ & $0$ & $0$ & $0$ & $0$ & $2$ & $0$ & $0$\\
$81$ & $36$ & $1$ & $0$ & $0$ & $0$ & $0$ & $0$ & $0$ & $1$ & $0$ & $0$\\
$98$ & $96$ & $1$ & $0$ & $0$ & $0$ & $0$ & $0$ & $0$ & $0$ & $0$ & $1$\\
$121$ & $12$ & $1$ & $0$ & $0$ & $0$ & $0$ & $0$ & $0$ & $1$ & $0$ & $0$\\
\bottomrule
\end{longtable}

We focus on the level $98$ in the rest of this subsection. According to Table \ref{table:wt1}, there are totally $96$ characters $\chi$ of eta-quotients of weight $1$, level $98$ and cover index $2$. They correspond to $96$ independent spaces $M_1(\Gamma_0(98),\chi)$ among which the dimension of exactly one space can be computed by Theorem \ref{thm:main}, namely, the space where $\chi$ is the character of $\eta(\tau)^{-15}\eta(2\tau)^{16}\eta(7\tau)$. For this $\chi$ we have $\dim_\numC M_1(\Gamma_0(98),\chi)=8$. One can show that $\chi\tbtmat{a}{b}{c}{d}=\legendre{-7}{d}$ by, for instance, a result of Gordon, Hughes and Newman (cf. \cite[Theorem 1.64]{Ono04}) or by checking directly for $\tbtmat{a}{b}{c}{d}$ being in a set of generators of $\Gamma_0(98)$. Thus, an explicit basis can be given using Eisenstein series constructed by Weisinger \cite{Wei77}. See also \cite[Theorem 8.5.23]{CS17}. This shows there is no cusp form in $M_1(\Gamma_0(98),\chi)$, which can also be proved using the tools developed in \cite{DS74}.

One may ask: is the space $M_1(\Gamma_0(98),\chi)$ generated by holomorphic eta-quotients? The answer is no. According to \eqref{eq:divEtaQuotient}, the functions
\begin{align*}
\eta(\tau)^{-1}\eta(2\tau)^2\eta(7\tau)^{-1}\eta(14\tau)^{2}&=q+q^2+q^4+q^7+q^8+q^9+2q^{11}+q^{14}+\dots\\
\eta(\tau)^{2}\eta(2\tau)^{-1}\eta(7\tau)^{2}\eta(14\tau)^{-1}&=1-2q+2q^4-2q^7+4q^8-2q^9-4q^{11}+6q^{16}+\dots\\
\eta(7\tau)^{-1}\eta(14\tau)^2\eta(49\tau)^{-1}\eta(98\tau)^{2}&=q^7+q^{14}+q^{28}+q^{49}+q^{56}+q^{63}+2q^{77}+q^{98}+\dots\\
\eta(7\tau)^{2}\eta(14\tau)^{-1}\eta(49\tau)^{2}\eta(98\tau)^{-1}&=1-2q^7+2q^{28}-2q^{49}+4q^{56}-2q^{63}-4q^{77}+6q^{112}+\dots
\end{align*}
are holomorphic eta-quotients of level $98$. Using Lemma \ref{lemm:whenCharSame} the characters of these four modular forms are equal to $\chi$ and hence they belong to $M_1(\Gamma_0(98),\chi)$. By their $q$-coefficients these four functions are linearly independent. However there are no other holomorphic eta-quotients of cover index $2$ in $M_1(\Gamma_0(98),\chi)$ so this space can not be generated by a set consisting of all eta-quotients. This can be seen as follows. Let $f(\tau)=\prod_{n\mid98}\eta(n\tau)^{r_n}$ be in $M_1(\Gamma_0(98),\chi)$ with $r_n\in\numZ$. Then $x_c\in\numgeq{Z}{0}$ for any $c\mid 98$ where $x_c$ is given by \eqref{eq:rntoxc}. According to Theorem \ref{thm:valence} (applied to $f$), \eqref{eq:divEtaQuotient} and the complete set of representatives of $\Gamma_0(98)\backslash\projQ$ described in the proof of Proposition \ref{prop:etaEisenstein} we have $x_1+x_2+6x_7+6x_{14}+x_{49}+x_{98}=336$ of which there are only finitely many solutions $(x_c)_{c\mid 98}$. For each solution, we obtain $(r_n)_{n\mid98}$ via \eqref{eq:xctorn} and then check whether $r_n\in\numZ$ and the character corresponding to $(r_n)_{n\mid98}$ is the same as $\chi$ using Lemma \ref{lemm:whenCharSame}. In this way we find that there are exactly four holomorphic eta-quotients of cover index $2$ in $M_1(\Gamma_0(98),\chi)$ (totally $69$ holomorphic eta-quotients, $57$ of which are primitive).

\subsection{The weight $3/2$}
\label{subsec:wt3/2}
There are totally $17862$ spaces $M_{3/2}(\Gamma_0(N),\chi)$ whose dimensions can be calculated using Theorem \ref{thm:main}. Among these spaces, the largest level is $N=400$ and the largest dimension is $48$. See Appendix \ref{apx:Usage of SageMath code} for the usage of SageMath code that gives these data. The following observation will be used later: the dimensions of these $17862$ spaces are all greater than $1$ when $N>36$.

\section{Application: one-dimensional spaces and Eisenstein series of rational weights}
\label{sec:one-dimensional spaces and Eisenstein series of rational weights}
In \cite[Section 7]{Zhu23}, the author gives infinitely many identities whose left-hand sides are Eisenstein series and right-hand sides are eta-quotients of rational weights for the levels $2$ and $3$. Using Theorem \ref{thm:main} (with $t=0$), we find new identities of this kind for the level $4$ in this section.

First let us rewrite Theorem \ref{thm:main} more explicitly for $N=4$. We have $m=6$, $\varepsilon_2=\varepsilon_3=0$. Let $r_1$, $r_2$, $r_4$ be rational numbers and set $k=\frac{1}{2}(r_1+r_2+r_4)$. Let
\begin{equation}
\label{eq:x1x2x4r1r2r4}
\begin{pmatrix}
x_1 \\ x_2 \\ x_4
\end{pmatrix}=
\begin{pmatrix}
4 & 2 & 1 \\
1 & 2 & 1\\
1 & 2 & 4
\end{pmatrix}\cdot
\begin{pmatrix}
r_1 \\ r_2 \\ r_4
\end{pmatrix}.
\end{equation}
Our aim is to find an Eisenstein series $E(\tau)$ of weight $k$ equal to $f(\tau)=\eta(\tau)^{r_1}\eta(2\tau)^{r_2}\eta(4\tau)^{r_4}$. We only consider the Eisenstein series defined at the cusp $\rmi\infty$, so assume $x_4=0$ which means $f$ does not vanish at the cusp $\rmi\infty$. Since $f$ must be a holomorphic eta-quotient, we assume $x_1,x_2\geq0$. Therefore,
\begin{equation*}
x_1=-6r_2-15r_4\geq0,\quad x_2=-3r_4\geq0,\quad\text{hence }k=\frac{1}{2}(-r_2-3r_4)\geq0.
\end{equation*}
Now the first inequality in \eqref{eq:mainThmCondtion2} becomes $k>-4+2\left(\left\{\frac{x_1}{24}\right\}+\left\{\frac{x_2}{24}\right\}\right)$, which holds since $k\geq0$. Consequently, \eqref{eq:dimFormula} with $t=0$ is applicable. We summarize:
\begin{prop}
\label{prop:level4x4zero}
Let $r_2$ and $r_4$ be rational numbers satisfying $-2r_2-5r_4\geq0$ and $-r_4\geq0$. Let $D\in\numgeq{Z}{1}$ such that $Dr_2,\,Dr_4\in2\numZ$. Then we have
\begin{equation*}
\dim_\numC M_k(\Gamma_0(4), \chi)=\left[\frac{-2r_2-5r_4}{8}\right]+\left[\frac{-r_4}{8}\right]+1,
\end{equation*}
where $\chi\colon\Dcover{\Gamma_0(4)}{D}\to\numC^\times$ is the character of $\eta(\tau)^{-2r_2-4r_4}\eta(2\tau)^{r_2}\eta(4\tau)^{r_4}$ and $k=\frac{1}{2}(-r_2-3r_4)$.
\end{prop}
\begin{proof}
We have shown that \eqref{eq:dimFormula} with $t=0$ holds. A simplification gives the desired formula.
\end{proof}
\begin{coro}
\label{coro:level4x4zeroDim0}
Let the notation and assumptions be as in Proposition \ref{prop:level4x4zero}. Then $\dim_\numC M_k(\Gamma_0(4), \chi)=1$ if and only if
\begin{equation*}
0\leq-2r_2-5r_4<8\quad\text{and}\quad0\leq-r_4<8.
\end{equation*}
Moreover, suppose $\dim_\numC M_k(\Gamma_0(4), \chi)=1$; then $k>2$ if and only if $(-2r_2-5r_4)+(-r_4)>8$.
\end{coro}
Note that when $\dim_\numC M_k(\Gamma_0(4), \chi)=1$, then $M_k(\Gamma_0(4), \chi)=\numC\eta(\tau)^{-2r_2-4r_4}\eta(2\tau)^{r_2}\eta(4\tau)^{r_4}$. Hence if $M_k(\Gamma_0(4), \chi)$ contains an Eisenstein series, then an identity that relates the eta-quotient and the Eisenstein series follows. In the rest we only consider the case $k>2$ in which nonholomorphic Eisenstein series are avoided. Similar identities can also be derived in the cases $k<2$, which we will do in the future work.

Recall the definition of Eisenstein series of rational weights greater than $2$:
\begin{deff}[Definition 3.1, \cite{Zhu23}]
\label{deff:Eis}
Let $D,N$ be positive integers, and let $k\in\frac{1}{D}\numZ$ that is greater than $2$. Let $\chi\colon \Dcover{\Gamma_0(N)}{D} \to \numC^\times$ be a finite index character such that $\chi(\widetilde{-I})=\etp{-k/2}$. Let $s \in \projQ$ and $\gamma_s \in \slZ$ such that $\gamma_s(\rmi\infty)=s$. If $\chi(\widetilde{\gamma_s}\widetilde{T}^{w_{s}}\widetilde{\gamma_s}^{-1})=1$, then we define the Eisenstein series $E_{\gamma_s,k}$ on the group $\Dcover{\Gamma_0(N)}{D}$, of weight $k$, with character $\chi$ and at cusp $\gamma_s$ as
\begin{equation*}
E_{\gamma_s,k}(\tau)=\sum_{\gamma \in \widetilde{\gamma_s}\langle \widetilde{T}^{w_{s}}\rangle\widetilde{\gamma_s}^{-1}\backslash\Dcover{\Gamma_0(N)}{D}}\chi(\gamma)^{-1}\cdot 1\vert_k\widetilde{\gamma_s}^{-1}\gamma.
\end{equation*}
\end{deff}
In the above definition, $w_{s}$ is the width (see Section \ref{sec:Divisors of modular forms}). If $s=a/c$ where $a,c$ are coprime integers, then $w_s=\frac{N}{(N,c^2)}$. The basic properties of these $E_{\gamma_s,k}$ are collected in \cite[Section 3]{Zhu23}.

Now let $N=4$ and let $\chi$ be as in Proposition \ref{prop:level4x4zero}. Since $\chi(\widetilde{T})=1$ the Eisenstein series $E_{I,k}(\tau)$ is well-defined provided that $k>2$, that is, $(-2r_2-5r_4)+(-r_4)>8$. The main result of this section is the following:
\begin{thm}
\label{thm:etaEisIdentity}
Let $r_2$ and $r_4$ be rational numbers satisfying
\begin{equation}
\label{eq:r2r4inequality}
0\leq-2r_2-5r_4<8,\quad0\leq-r_4<8,\quad(-2r_2-5r_4)+(-r_4)>8.
\end{equation}
Let $D$ be a positive integer such that $Dr_2,\,Dr_4\in2\numZ$. Let $\chi\colon\Dcover{\Gamma_0(4)}{D}\to\numC^\times$ be the character of $\eta(\tau)^{-2r_2-4r_4}\eta(2\tau)^{r_2}\eta(4\tau)^{r_4}$. Set $k=\frac{1}{2}(-r_2-3r_4)$ and
\begin{equation*}
E_{I,k}(\tau)=E_{I,k}(\tau;r_2,r_4)=\sum_{\gamma \in \langle \widetilde{T}\rangle\backslash\Dcover{\Gamma_0(4)}{D}}\chi(\gamma)^{-1}\cdot 1\vert_k\gamma.
\end{equation*}
Then we have
\begin{equation}
\label{eq:etaEis}
\eta(\tau)^{-2r_2-4r_4}\eta(2\tau)^{r_2}\eta(4\tau)^{r_4}=\frac{1}{2D}E_{I,k}(\tau;r_2,r_4).
\end{equation}
\end{thm}
\begin{proof}
Corollary \ref{coro:level4x4zeroDim0} shows that $\dim_\numC M_k(\Gamma_0(4), \chi)=1$. Let $x_1$, $x_2$ and $x_4$ be as in \eqref{eq:x1x2x4r1r2r4} with $r_1=-2r_2-4r_4$. Then the assumptions imply that $x_1\geq0$, $x_2\geq0$ and $x_4=0$. Thus $\eta(\tau)^{-2r_2-4r_4}\eta(2\tau)^{r_2}\eta(4\tau)^{r_4}$ is holomorphic at all cusps and hence is in $M_k(\Gamma_0(4), \chi)$. By \cite[Theorem 3.3]{Zhu23} we have $E_{I,k}\in M_k(\Gamma_0(4), \chi)$. It follows that $\eta(\tau)^{-2r_2-4r_4}\eta(2\tau)^{r_2}\eta(4\tau)^{r_4}=c\cdot E_{I,k}$. Let $\Im\tau\to+\infty$; we find that $c=\frac{1}{2D}$ again by \cite[Theorem 3.3]{Zhu23}.
\end{proof}

In the rest of this section we will write out the Fourier expansion of $\frac{1}{2D}E_{I,k}(\tau)$ which will give us a more explicit form of the identity \eqref{eq:etaEis}.
\begin{lemm}
For any rational numbers $r_2,r_4$ with $k=\frac{1}{2}(-r_2-3r_4)>2$ we have
\begin{equation*}
\frac{1}{2D}E_{I,k}(\tau;r_2,r_4)=1+\sum_{\twoscript{4\mid c>0,d\in\numZ}{(c,d)=1}}\etp{-\frac{1}{24}P(c,d;r_2,r_4)}\cdot(c\tau+d)^{-k}
\end{equation*}
where
\begin{align}
P(c,d;r_2,r_4)&=r_2\cdot\left(\Psi\tbtMat{a}{2b}{c/2}{d}-2\Psi\tbtMat{a}{b}{c}{d}\right)+r_4\cdot\left(\Psi\tbtMat{a}{4b}{c/4}{d}-4\Psi\tbtMat{a}{b}{c}{d}\right)\notag\\
&=r_2\cdot\left(12s(-d,c/2)-24s(-d,c)+3\right)+r_4\cdot\left(12s(-d,c/4)-48s(-d,c)+9\right)\label{eq:Pcdr2r4}
\end{align}
with $a,b$ being any integers such that $ad-bc=1$.
\end{lemm}
\begin{proof}
According to \cite[Lemma 3.2]{Zhu23},
\begin{equation*}
E_{I,k}(\tau;r_2,r_4)=D\cdot\sum_{\twoscript{4\mid c\in\numZ,d\in\numZ}{(c,d)=1}}\chi\widetilde{\tbtmat{a}{b}{c}{d}}^{-1}\cdot(c\tau+d)^{-k}
\end{equation*}
where $a,b$ are any integers such that $ad-bc=1$. By Remark \ref{rema:mulSysProp}, $\chi\left(\widetilde{-I}\right)=\etp{-k/2}$. It follows that
\begin{equation*}
\chi\widetilde{\tbtmat{a}{b}{c}{d}}^{-1}\cdot(c\tau+d)^{-k}=\chi\widetilde{\tbtmat{-a}{-b}{-c}{-d}}^{-1}\cdot(-c\tau-d)^{-k}.
\end{equation*}
Thus
\begin{equation*}
\frac{1}{2D}E_{I,k}(\tau;r_2,r_4)=1+\sum_{\twoscript{4\mid c>0,d\in\numZ}{(c,d)=1}}\chi\widetilde{\tbtmat{a}{b}{c}{d}}^{-1}\cdot(c\tau+d)^{-k}.
\end{equation*}
Substituting \eqref{eq:charEtaQuotient} and \eqref{eq:Psi} in the above expression gives the desired formula.
\end{proof}
\begin{prop}
\label{prop:EIkr2r4Fourier}
With the notation being as above, we have
\begin{equation*}
\frac{1}{2D}E_{I,k}(\tau;r_2,r_4)=1+\rme^{-\uppi\rmi k/2}\frac{(2\uppi)^k}{\Gamma(k)}\sum_{n \in \numgeq{Z}{1}}n^{k-1}a(n)\cdot q^n,
\end{equation*}
where
\begin{equation*}
a(n)=\sum_{c\in\numgeq{Z}{1}}\frac{1}{(4c)^k}\sum_{\twoscript{0\leq d<4c}{(d,4c)=1}}\etp{\frac{dn}{4c}-\frac{1}{24}P(4c,d;r_2,r_4)}.
\end{equation*}
\end{prop}
\begin{proof}
For $d\in\numZ$, $c\in\numgeq{Z}{1}$ set
\begin{equation*}
f(d,c)=\begin{dcases}
\etp{-\frac{1}{24}P(c,d;r_2,r_4)} & \text{ if }4\mid c \text{ and }(d,c)=1;\\
0 & \text{ if }4\nmid c \text{ and }(d,c)=1;\\
f\left(\frac{d}{(d,c)},\frac{c}{(d,c)}\right) &\text{ if } (d,c)>1.
\end{dcases}
\end{equation*}
Thus, $\frac{1}{2D}E_{I,k}(\tau;r_2,r_4)=1+\zeta(k)^{-1}\cdot\sum_{c\in\numgeq{Z}{1},d\in\numZ}f(d,c)(c\tau+d)^{-k}$ since $\sum_{t \mid (c,d)}\mu(t) = 0$ if $(c,d) > 1$ and $\sum_{t}\frac{\mu(t)}{t^k}=\zeta(k)^{-1}$, where $\zeta$ is the Riemann zeta function and $\mu$ is the M\"obius function. Note that $f(d+c,c)=f(d,c)$. This fact, together with the Lipschitz summation formula
\begin{equation*}
\sum_{n \in \numZ}\frac{1}{(\tau+n)^s}=\rme^{-\uppi\rmi s/2}\frac{(2\uppi)^s}{\Gamma(s)}\sum_{n \geq 1}n^{s-1}\rme^{2\uppi\rmi n\tau},\quad \tau\in\uhp,\,\Re(s)>1
\end{equation*}
implies that
\begin{align*}
\sum_{d \in \numZ}f(d,c)(c\tau+d)^{-k} &= c^{-k}\sum_{0 \leq d_1 < c}f(d_1,c)\sum_{d_0 \in \numZ}(\tau+\frac{d_1}{c}+d_0)^{-k} \\
&=\rme^{-\uppi\rmi k/2}\frac{(2\uppi)^k}{\Gamma(k)}c^{-k}\sum_{n \in \numgeq{Z}{1}}n^{k-1}\sum_{0 \leq d < c}f(d,c)\etp{\frac{dn}{c}}q^n.
\end{align*}
Therefore,
\begin{equation}
\label{eq:proofFourierEisInfty}
\frac{1}{2D}E_{I,k}(\tau;r_2,r_4)=1+\rme^{-\uppi\rmi k/2}\frac{(2\uppi)^k}{\Gamma(k)\zeta(k)}\sum_{n \in \numgeq{Z}{1}}n^{k-1}\sum_{c \in \numgeq{Z}{1}}\frac{1}{c^k}\sum_{0 \leq d < c}f(d,c)\etp{\frac{dn}{c}}q^n.
\end{equation}
By the definition of $f(d,c)$ we find that
\begin{align*}
\sum_{c \in \numgeq{Z}{1}}\frac{1}{c^k}\sum_{0 \leq d < c}f(d,c)\etp{\frac{dn}{c}}&=\sum_{t\in\numgeq{Z}{1}}\sum_{\twoscript{c\in\numgeq{Z}{1},0\leq d<c}{(d,c)=t}}\frac{1}{(tc)^k}f(td,tc)\etp{\frac{dn}{c}}\\
&=\zeta(k)\cdot\sum_{\twoscript{c\in\numgeq{Z}{1},0\leq d<4c}{(d,4c)=1}}\frac{1}{(4c)^k}f(d,4c)\etp{\frac{dn}{4c}}\\
&=\zeta(k)\cdot\sum_{\twoscript{c\in\numgeq{Z}{1},0\leq d<4c}{(d,4c)=1}}\frac{1}{(4c)^k}\etp{\frac{dn}{4c}-\frac{1}{24}P(4c,d;r_2,r_4)}.
\end{align*}
Inserting this into \eqref{eq:proofFourierEisInfty} gives the desired formula.
\end{proof}
\begin{rema}
The reader may compare the above proposition with \cite[Theorem 6.2]{Zhu23}.
\end{rema}

We can now give an equivalent statement of Theorem \ref{thm:etaEisIdentity} in a form that gives Fourier coefficients of certain infinite $q$-products:
\begin{coro}
\label{coro:etaQuotientCoefficient}
Let $r_2$ and $r_4$ be rational numbers satisfying \eqref{eq:r2r4inequality}. Let $D$ be a positive integer such that $Dr_2,\,Dr_4\in2\numZ$. Let $\sum_{n\in\numgeq{Z}{0}}c(n)q^n$ be the holomorphic branch of the $D$th root of the infinite product
\begin{equation*}
\prod_{n\in\numgeq{Z}{1}}(1-q^n)^{-2Dr_2-4Dr_4}\prod_{n\in\numgeq{Z}{1}}(1-q^{2n})^{Dr_2}\prod_{n\in\numgeq{Z}{1}}(1-q^{4n})^{Dr_4}
\end{equation*}
with $c(0)=1$. Then for $n\geq 1$ we have
\begin{equation*}
c(n)=\rme^{-\uppi\rmi k/2}\frac{(2\uppi)^k}{\Gamma(k)}n^{k-1}\sum_{c\in\numgeq{Z}{1}}\frac{1}{(4c)^k}\sum_{\twoscript{0\leq d<4c}{(d,4c)=1}}\etp{\frac{dn}{4c}-\frac{1}{24}P(4c,d;r_2,r_4)}.
\end{equation*}
(See \eqref{eq:Pcdr2r4} for the definition of $P(4c,d;r_2,r_4)$.)
\end{coro}
\begin{proof}
There are exactly $D$ holomorphic branches and the one with the leading term $1$ is the eta-quotient $\eta(\tau)^{-2r_2-4r_4}\eta(2\tau)^{r_2}\eta(4\tau)^{r_4}$. Now the assertion follows from Theorem \ref{thm:etaEisIdentity} and Proposition \ref{prop:EIkr2r4Fourier}.
\end{proof}

\begin{examp}
\label{examp:App1Examp1}
Let us consider the special case $r_2=0$, $r_4=-\frac{3}{2}$ for which \eqref{eq:r2r4inequality} holds. Note that $k=\frac{1}{2}(-r_2-3r_4)=\frac{9}{4}$. Set $D=4$. Then Corollary \ref{coro:etaQuotientCoefficient} shows
\begin{equation*}
\frac{\prod_{n\in\numgeq{Z}{1}}(1-q^n)^6}{\sqrt{\prod_{n\in\numgeq{Z}{1}}(1-q^{4n})^3}}=1+\rme^{-9\uppi\rmi /8}\frac{(2\uppi)^{\frac{9}{4}}}{\Gamma(\frac{9}{4})}\sum_{n\in\numgeq{Z}{1}}n^{\frac{5}{4}}\left(\sum_{c\in\numgeq{Z}{1}}\frac{1}{(4c)^{\frac{9}{4}}}\sum_{\twoscript{0\leq d<4c}{(d,4c)=1}}\etp{\frac{dn}{4c}-\frac{1}{24}P(4c,d;0,-\frac{3}{2})}\right)\cdot q^n,
\end{equation*}
where the square root is the principal branch $\sqrt{\prod_{n\in\numgeq{Z}{1}}(1-q^{4n})^3}=1+o(1)$ and
\begin{equation*}
P(4c,d;0,-\frac{3}{2})=-18s(-d,c)+72s(-d,4c)-\frac{27}{2}.
\end{equation*}
(See \eqref{eq:DedekindSum} for the definition of $s(-d,c)$.) For numerical check (of this and the next examples) using SageMath programs, see Appendix \ref{apx:Usage of SageMath code}.
\end{examp}
\begin{examp}
\label{examp:N4rn274}
Corollary \ref{coro:etaQuotientCoefficient} contains as well some identities involving ordinary eta-quotients, that is, eta-quotients with integral exponents. For instance, set $r_2=7$, $r_4=-4$ and hence $k=\frac{1}{2}(-r_2-3r_4)=\frac{5}{2}$ and $D=2$. For this setting, \eqref{eq:r2r4inequality} holds. The identity reads
\begin{multline*}
\prod_{n\in\numgeq{Z}{1}}(1-q^n)^2(1-q^{2n})^7(1-q^{4n})^{-4}\\
=1+\rme^{-5\uppi\rmi/4}\frac{(2\uppi)^{\frac{5}{2}}}{\Gamma(\frac{5}{2})}\sum_{n\in\numgeq{Z}{1}}n^{\frac{3}{2}}\sum_{c\in\numgeq{Z}{1}}\frac{1}{(4c)^{\frac{5}{2}}}\sum_{\twoscript{0\leq d<4c}{(d,4c)=1}}\etp{\frac{dn}{4c}-\frac{1}{24}P(4c,d;7,-4)}\cdot q^n
\end{multline*}
where
\begin{equation*}
P(4c,d;7,-4)=-48s(-d,c)+84s(-d,2c)+24s(-d,4c)-15.
\end{equation*}
We will return to this function, i.e., $\eta(\tau)^2\eta(2\tau)^7\eta(4\tau)^{-4}$ in the next section (see Section \ref{subsuesec:closedFourier}) and show that it is, among others, a Hecke eigenform in a generalized sense. (In \cite[Section 4]{ZZ23}, the authors obtained some Hecke eigenforms in the generalized sense. However, for $N=4$, only Hecke eigenforms of weight $1/2$, $1$, $3/2$ are obtained. The current example gives a Hecke eigenform of weight $5/2$ which can not be achieved using methods in \cite{ZZ23}.)
\end{examp}

\section{Application: an extension of Martin's list of multiplicative eta-quotients}
\label{sec:an extension of Martin's list of multiplicative eta-quotients}
In 1996 Martin \cite{Mar96} obtained the complete list of integral-weight holomorphic eta-quotients that are eigenforms for all Hecke operators. This generalized a previous result of Dummit, Kisilevsky and McKay \cite{DKM85}. The term ``eta-quotient'' in \cite{Mar96} refers to one whose multiplier system is induced by a Dirichlet character (necessarily real) and the Hecke operators are in the sense of \cite{AL70}.

In this section, we provide an algorithm for finding out all holomorphic eta-quotients $f(\tau)=\prod_{n\geq1}\eta(n\tau)^{r_n}$ of integral exponents $r_n$ (integral or half-integral weights) and of arbitrary multiplier systems such that either of the following conditions holds:
\begin{enumerate}
  \item $\dim_\numC M_k(\Gamma_0(N),\chi)$ can be computed using Theorem \ref{thm:main} (with $t=0$) and is equal to $1$,
  \item $\dim_\numC M_k(\Gamma_0(N),\chi)$ can be computed using Theorem \ref{thm:main} (with $t=0$) and is greater than $1$, $f$ is a cusp form, $k>2$ and there exist $\dim_\numC M_k(\Gamma_0(N),\chi)-1$ linearly independent Eisenstein series in $M_k(\Gamma_0(N),\chi)$ (see Definition \ref{deff:Eis}),
\end{enumerate}
where $k=\frac{1}{2}\sum_nr_n$ is the weight, $N$ is the least common multiple of all $n$ with $r_n\neq 0$ (when $r_n=0$ for all $n$, let $N=1$) and $\chi$ is the character $\widetilde{\Gamma_0(N)}\rightarrow\numC^\times$ of $f$. (Note that throughout this section $\widetilde{\Gamma_0(N)}=\Dcover{\Gamma_0(N)}{2}$ as in Section \ref{sec:The classification of characters induced by eta-quotients}.)

If $f$ satisfies either of the above conditions, we say it is \emph{admissible}. If $f$ satisfies the condition (a) then we say it is \emph{admissible of type I}; if it satisfies (b) then we say it is \emph{admissible of type II}. Moreover, if $f$ satisfies either of the conditions
\begin{enumerate}
  \item[(a')] $\dim_\numC M_k(\Gamma_0(N),\chi)=1$,
  \item[(b')] $\dim_\numC M_k(\Gamma_0(N),\chi)\geq2$ but $\dim_\numC S_k(\Gamma_0(N),\chi)=1$ and $f$ is a cusp form,
\end{enumerate}
we say $f$ is \emph{weakly admissible}. Thus an admissible eta-quotient must be a weakly admissible eta-quotient. Similarly, we can define the concept of weakly admissible eta-quotients of type I and type II. For each weakly admissible $f$, we will show that it is a Hecke eigenform in the sense of Wohlfahrt \cite{Woh57}; see Theorem \ref{thm:mainApp2}. Therefore, our list of admissible eta-quotients can be considered as an extension of Martin's list.

\subsection{The algorithm for finding admissible eta-quotients}
\label{subsec:The algorithm for finding admissible eta-quotients}
Throughout this subsection, let $f(\tau)=\prod_{n\geq1}\eta(n\tau)^{r_n}$ be a holomorphic eta-quotient where $r_n$ are integers and are equal to $0$ for all but finitely many $n$. Let $k=\frac{1}{2}\sum_nr_n$ be the weight and $N$, the level\footnote{The reader should notice that this usage of ``eta quotient of level $N$'' is different from that in the second half of Section \ref{sec:Modular forms of rational weight}. In the usage here any eta-quotient has a uniquely determined level while in the usage before, an eta-quotient, say $\eta(\tau)\eta(2\tau)$, may be regarded as of level $2$, or of level $4$, etc.}, be the least common multiple of $\{n\in\numgeq{Z}{1}\colon r_n\neq0\}$ (with the convention $N=1$ if $r_n=0$ for all $n$). Let $\chi\colon\widetilde{\Gamma_0(N)}\rightarrow\numC^\times$ be the character of $f$ (see \eqref{eq:charEtaQuotient} with $D=2$). Since $f$ is holomorphic at all cusps we have $x_c\geq0$ for all $c\mid N$ ($x_c$ are defined in \eqref{eq:rntoxc}).  By Remark \ref{rema:uboundlbound} we have
\begin{align*}
\dim_\numC M_k(\Gamma_0(N), \chi)&\geq\frac{k-1}{12}m+\frac{1}{4}\varepsilon_2+\frac{1}{3}\varepsilon_3+\sum_{c\mid N}\phi(c,N/c)\cdot\left(\frac{1}{2}-\left\{\frac{x_c}{24}\right\}\right)\\
&>\frac{k-1}{12}m-\frac{1}{2}\sum_{c\mid N}\phi(c,N/c)
\end{align*}
where $m$, $\varepsilon_2$, $\varepsilon_3$ are given by \eqref{eq:mGamma0N}, \eqref{eq:e2Gamma0N}, \eqref{eq:e3Gamma0N} respectively. 
\begin{lemm}
\label{lemm:admissible}
If $\left(\frac{(k-1)N}{12}-\frac{3\sqrt{N}}{2}\right)\cdot\prod_{p\mid N}\left(1+\frac{1}{p}\right)\geq1$, then $f$ is not admissible. As a consequence, there are only finitely many admissible eta-quotients.
\end{lemm}
\begin{proof}
Assume by contradiction that $f$ is admissible. If the condition (a) holds, then
\begin{equation}
\label{eq:dimInequalityconda}
1=\dim_\numC M_k(\Gamma_0(N), \chi)>\frac{k-1}{12}m-\frac{1}{2}\sum_{c\mid N}\phi(c,N/c).
\end{equation}
By \eqref{eq:einftymN} we have $\varepsilon_\infty=\sum_{c\mid N}\phi(c,N/c)\leq m\cdot N^{-1/2}$. This, together with \eqref{eq:dimInequalityconda} and \eqref{eq:mGamma0N}, implies that $\left(\frac{(k-1)N}{12}-\frac{\sqrt{N}}{2}\right)\cdot\prod_{p\mid N}\left(1+\frac{1}{p}\right)<1$ which contradicts the assumption. On the other hand, if the condition (b) holds. Since $k>2$, there are totally\footnote{These $n_0$ functions, in the rational weight ($k>2$) situation, constitute a basis of the Eisenstein space  which, as in the integral weight case, is defined as the orthogonal complement of the space of cusp forms under the Petersson inner product. In addition, one can prove that $n_0=\abs{R}$ where $R$ is \eqref{eq:R} with $f_0$ arbitrarily chosen.} $n_0$ Eisenstein series in $M_k(\Gamma_0(N), \chi)$, where
\begin{equation}
\label{eq:n0EisSeries}
n_0=\#\left\{s \in \Gamma_0(N)\backslash\projQ\colon \chi(\widetilde{\gamma_s}\widetilde{T}^{w_{s}}\widetilde{\gamma_s}^{-1})=1 \right\},
\end{equation}
($w_s$ is the width; see Section \ref{sec:Divisors of modular forms}, $\gamma_s$ is any matrix in $\slZ$ such that $\gamma_s(\rmi\infty)=s$) and these $n_0$ functions are $\numC$-linearly independent. See \cite[Theorem 3.3]{Zhu23}. By condition (b) we have $n_0\geq\dim_\numC M_k(\Gamma_0(N), \chi)-1$. Since in the current case $f$ is a cusp form, these $n_0$ Eisenstein series and $f$ together constitute a basis of $M_k(\Gamma_0(N), \chi)$ and $n_0=\dim_\numC M_k(\Gamma_0(N), \chi)-1$. Hence
\begin{equation}
\label{eq:dimInequalitycondb}
n_0+1=\dim_\numC M_k(\Gamma_0(N), \chi)>\frac{k-1}{12}m-\frac{1}{2}\sum_{c\mid N}\phi(c,N/c).
\end{equation}
This inequality, together with \eqref{eq:einftymN}, \eqref{eq:mGamma0N} and the fact $n_0\leq\varepsilon_\infty=\sum_{c\mid N}\phi(c,N/c)$, implies that $\left(\frac{(k-1)N}{12}-\frac{3\sqrt{N}}{2}\right)\cdot\prod_{p\mid N}\left(1+\frac{1}{p}\right)<1$ which contradicts the assumption. Therefore, we have shown that both of conditions (a) and (b) lead to contradictions, so $f$ is not admissible.

Now we begin to prove that there are only finitely many admissible eta-quotients. Since $f\in M_k(\Gamma_0(N), \chi)$ is a holomorphic eta-quotient of integral exponents, $2k\in\numgeq{Z}{0}$. For $k=0$, we must have $f=1$. For $k=1/2$ and $1$, we have shown in Remark \ref{rema:rangeN} that there are only finitely many $N$ such that $\dim_\numC M_k(\Gamma_0(N), \chi)$ can be calculated using Theorem \ref{thm:main}. For each pair $(k,N)$, the holomorphic eta-quotients of level $N$ and weight $k$ constitute a finite set since the cardinality of this set does not exceed the cardinality of the solution set of $\sum_{c\mid N}\phi(c,N/c)x_c=2mk$, $x_c\in\numgeq{Z}{0}$ (see Proposition \ref{prop:rnxc} and Remark \ref{rema:ANcolumn}) which is finite. This proves that there are only finitely many admissible eta-quotients of weights $0$, $1/2$ and $1$. For each weight $k\geq 3/2$, there exists an $N_0$ such that when $N>N_0$ we have $\left(\frac{(k-1)N}{12}-\frac{3\sqrt{N}}{2}\right)\cdot\prod_{p\mid N}\left(1+\frac{1}{p}\right)\geq1$. Thus, the level of an admissible eta-quotient of weight $k$ does not exceed $N_0$ and hence there are finitely many of them. Finally, if $k\geq19$ and $N\geq2$, we have
\begin{equation*}
\left(\frac{(k-1)N}{12}-\frac{3\sqrt{N}}{2}\right)\cdot\prod_{p\mid N}\left(1+\frac{1}{p}\right)\geq\frac{3}{2}(N-\sqrt{N})\cdot\prod_{p\mid N}\left(1+\frac{1}{p}\right)\geq1,
\end{equation*}
and if $k\geq31$, $N=1$, then the left-hand side still $\geq1$. It follows that if the weight of $f$ is greater than or equal to $31$, then $f$ is not admissible. Therefore, there are only finitely many admissible eta-quotients.
\end{proof}
\begin{coro}
\label{coro:admmissible}
If $N>400$, then $f$ is not admissible. Moreover, if $N>36$, then $f$ is not admissible of type I.
\end{coro}
\begin{proof}
Suppose $N>400$; we have $k\neq0$. If $k=1/2$, $1$ or $3/2$, then by Tables \ref{table:wt1/2}, \ref{table:wt1} and Section \ref{subsec:wt3/2}, $\dim_\numC M_k(\Gamma_0(N), \chi)$ can not be calculated using Theorem \ref{thm:main} with $t=0$ (that is, the first inequality in \eqref{eq:mainThmCondtion2} does not hold) for any $\chi$. Thus, $f$ is not admissible\footnote{Although $f$ is possibly weakly admissible.}. Otherwise, if $k\geq2$, then obviously $\left(\frac{(k-1)N}{12}-\frac{3\sqrt{N}}{2}\right)\cdot\prod_{p\mid N}\left(1+\frac{1}{p}\right)\geq1$. It follows from this and Lemma \ref{lemm:admissible} that $f$ is not admissible.

Now we begin to prove the latter assertion; suppose $N>36$. Certainly $k\neq0$. According to Table \ref{table:wt1/2}, Table \ref{table:wt1} and Subsection \ref{subsec:wt3/2}, if $k=1/2$, $1$ or $3/2$, then $\dim_\numC M_k(\Gamma_0(N),\chi)\geq2$ provided it can be calculated using Theorem \ref{thm:main} (with $t=0$). Hence $f$ is not admissible of type I. If $k\geq2$, then for $36<N\leq400$ but $N\neq 49$ we have
\begin{equation*}
\dim_\numC M_k(\Gamma_0(N), \chi)>\frac{k-1}{12}m-\frac{1}{2}\sum_{c\mid N}\phi(c,N/c)>1.
\end{equation*}
See Appendix \ref{apx:Usage of SageMath code} for the SageMath code verifying this. If $N=49$, then $\varepsilon_3=2$ (see \eqref{eq:e3Gamma0N}), and hence
\begin{align*}
\dim_\numC M_k(\Gamma_0(49), \chi)&\geq\frac{k-1}{12}m+\frac{1}{4}\varepsilon_2+\frac{1}{3}\varepsilon_3+\sum_{c\mid N}\phi(c,N/c)\cdot\left(\frac{1}{2}-\left\{\frac{x_c}{24}\right\}\right)\\
&>\frac{2-1}{12}\cdot56+\frac{2}{3}-\frac{1}{2}\sum_{c\mid 49}\phi(c,49/c)>1.
\end{align*}
Therefore $f$ is not admissible of type I for $k\geq2$, $36<N\leq400$.
\end{proof}
According to the above corollary and Lemma \ref{lemm:admissible}, the following algorithm can be used to present all admissible nonconstant eta-quotients.
\begin{enumerate}
  \item Let $N$ range over $\{1,2,3,4,\dots,400\}$. 
  \item For a fixed $N$, set $K_N=\left\{k\in\frac{1}{2}\numZ\colon\frac{1}{2}\leq k<1+\frac{12}{m}+\frac{18}{\sqrt{N}}\right\}$, where $m$ is defined in \eqref{eq:mGamma0N}. 
  \item Let $N$ be fixed as above. For any $c\mid N$, let $x_c$ range over nonnegative integers such that $\sum_{c\mid N}\phi(c,N/c)x_c\in2m\cdot K_N$.
  \item Let $N$ and $(x_c)_{c\mid N}$ be fixed. Determine the sequence $(r_n)_{n\mid N}$ by \eqref{eq:xctorn} and set $k=\frac{1}{2}\sum_{n\mid N}r_n$, $g=\prod_{n\mid N}\eta(n\tau)^{r_n}$. Since $\sum_{c\mid N}\phi(c,N/c)x_c=m\cdot\sum_{n\mid N}r_n$ (see Remark \ref{rema:ANcolumn}) we have $k\in K_N$.
  \item If $r_n$ are not all integers, or they are all integers but $N$ is not the least common multiple of $\{r_n\neq 0\colon n\mid N\}$, then we continue to the next value of $(x_c)_{c\mid N}$ in Step (d) since in these cases either $g$ is not an eta-quotient of integral exponents or $g$ has been considered in a previous step.
  \item We check whether \eqref{eq:mainThmCondtion2} with $t=0$ holds. If it does not hold, $g$ is not admissible so we go to Step (d) and consider the next value of $(x_c)_{c\mid N}$.
  \item We check whether $\dim_\numC M_k(\Gamma_0(N), \chi)=1$ via \eqref{eq:dimFormula}. If this is the case, then we find an admissible $g$. We record it, go to Step (d) and consider the next value of $(x_c)_{c\mid N}$.
  \item We check whether $k>2$, $g$ is a cusp form (equivalent to $x_c>0$ for all $c\mid N$) and $\dim_\numC M_k(\Gamma_0(N), \chi)=n_0+1$ where $n_0$ is given in \eqref{eq:n0EisSeries}. If this is the case, then $g$ is admissible which we record; otherwise $g$ is not admissible. We have completely determined whether this $g$ is admissible or not so we go to Step (d) and consider the next value of $(x_c)_{c\mid N}$.
  \item Since for each $N$ there are only finitely many values of $(x_c)_{c\mid N}$, this algorithm halts in finitely many steps.
\end{enumerate}
\begin{prop}
The above algorithm gives all admissible nonconstant eta-quotients with no repetition.
\end{prop}
\begin{proof}
First, it is immediate that the functions $g$ recorded in Steps (g) or (h) are admissible nonconstant eta-quotients. Second, suppose $g_1$ and $g_2$ are functions recorded in Steps (g) or (h) corresponding to different pairs $\left(N_1,(x_c)_{c\mid N_1}\right)$ and $\left(N_2,(x'_c)_{c\mid N_2}\right)$. If $N_1\neq N_2$, then the exponents of $g_1$ are different from those of $g_2$ since their least common multiples are different. Hence $g_1\neq g_2$. Otherwise $N_1=N_2$ but their is some $c\mid N_1$ such that $x_{c}\neq x'_{c}$. Thus the sequences of exponents of $g_1$ and $g_2$ are different by Proposition \ref{prop:rnxc}, and hence $g_1\neq g_2$ as well. Finally, let $f=\prod_{n\mid N_0}\eta(n\tau)^{r_n}$ be any admissible nonconstant eta-quotient whose level is $N_0$ and weight is $k_0$. By Corollary \ref{coro:admmissible} we have $1\leq N_0\leq 400$; by Lemma \ref{lemm:admissible} we have $k_0\in K_{N_0}$. Define $(y_c)_{c\mid N_0}$ by \eqref{eq:rntoxc} (with $N=N_0$ and $x_c$ replaced by $y_c$). It is immediate that $y_c\in\numZ$. Moreover, $y_c\geq0$ since $f$ is holomorphic at all cusps. We have $\sum_{c\mid N_0}\phi(c,N_0/c)y_c=2mk_0$ (see Remark \ref{rema:ANcolumn}) and hence $\sum_{c\mid N_0}\phi(c,N_0/c)y_c\in2m\cdot K_{N_0}$ where $m=N_0\prod_{p\mid N_0}(1+1/p)^{-1}$. It follows that in Steps (a) and (c) the variable $N$ can actually take the value $N_0$ and $(x_c)_{c\mid N}$ can take the value $(y_c)_{c\mid N_0}$. The corresponding $g$ in Step (d) is equal to $f$ according to Proposition \ref{prop:rnxc}. Therefore, $f$ is recorded in some step of the algorithm which concludes the proof.
\end{proof}
\begin{rema}
If Step (h) is removed, then the remaining part of the algorithm gives all admissible nonconstant eta-quotients of type I. If this is the task, then we can shrink the set that $N$ ranges over in Step (a) to $\{1,2,3,\dots,36\}$ by the latter assertion of Corollary \ref{coro:admmissible}. In the following subsections, we will only consider such eta-quotients because finding those of type II need more effort and the program takes rather long time\footnote{We have provided the SageMath code that can produce admissible eta-quotients of type II. See Appendix \ref{apx:Usage of SageMath code} for the usage.} which we will do in a future paper. The simplest example of type II is $\eta(\tau)^{24}$ since it is a cusp form, $\dim_\numC M_{12}(\slZ)=2$ and there is an Eisenstein series in $M_{12}(\slZ)$. As a comparison, $\eta(\tau)^{r}$ is an admissible eta-quotient of type I for $r=0,1,2,\dots,23$ since $\dim_\numC M_{r/2}(\slZ,\chi_{\eta^r})=1$. See \cite[Section 5]{ZZ23}.
\end{rema}

\subsection{The table of admissible eta-quotients of type I}
For the SageMath program realizing the algorithm for finding admissible eta-quotients of type I in the last subsection, see Appendix \ref{apx:Usage of SageMath code}. Here (Table \ref{table:admissibleTypeI} below) we present all such eta-quotients for levels $N=1,\,p_1,\,p_1^2,\,p_1^3$ and $p_1p_2$ where $p_1$, $p_2$ are distinct primes.

\begin{longtable}{llllll}
\caption{All admissible eta-quotients of type I of level $N=1,p_1,\,p_1^2,\,p_1^3$ or $p_1p_2$ where $p_1$, $p_2$ are distinct primes. An entry $n_1^{r_1}n_2^{r_2}\cdots n_t^{r_t}$ represents the eta-quotient $\eta(n_1\tau)^{r_1}\eta(n_2\tau)^{r_2}\dots\eta(n_t\tau)^{r_t}$ and $k$ is the weight.\label{table:admissibleTypeI}} \\
\toprule
\endfirsthead
\multicolumn{6}{l}{Table \ref{table:admissibleTypeI}. (continued)} \\
\toprule
\endhead
\bottomrule
\multicolumn{6}{r}{To be continued} \\[2ex]
\endfoot
\bottomrule
\endlastfoot
\multicolumn{6}{l}{$N=1$}\\
\multicolumn{6}{l}{$1^r$ where $r=0,1,2,\dots,23$}\\
\midrule
\multicolumn{6}{l}{$N=2$, $k=1/2$}\\
$1^{-1}2^{2}$ & $1^{0}2^{1}$ & $1^{2}2^{-1}$ & &&\\
\midrule
\multicolumn{6}{l}{$N=2$, $k=1$}\\
$1^{-2}2^{4}$ & $1^{-1}2^{3}$ & $1^{0}2^{2}$ & $1^{1}2^{1}$ & $1^{3}2^{-1}$ & $1^{4}2^{-2}$\\
\midrule
\multicolumn{6}{l}{$N=2$, $k=3/2$}\\
$1^{-3}2^{6}$ & $1^{-2}2^{5}$ & $1^{-1}2^{4}$ & $1^{0}2^{3}$ & $1^{1}2^{2}$ & $1^{2}2^{1}$\\
\end{longtable}

\begin{rema}
Each entry in Table \ref{table:admissibleTypeI} is different from others. For instance, although the eta-quotient $\eta(\tau)^3$, when considered as $\eta(\tau)^3\eta(2\tau)^0$, generates the one-dimensional space $M_{3/2}(\Gamma_0(2),\chi)$ where $\chi\colon\widetilde{\Gamma_0(3)}\rightarrow\numC^\times$ is the character of $\eta(\tau)^3\eta(2\tau)^0$, we list $\eta(\tau)^3$ only in the group $N=1$, not $N=2,k=3/2$ to avoid repetition. As a comparison, Tables 2 and 3 of \cite{ZZ23}, whose entries constitute a subset of Table \ref{table:admissibleTypeI} above, contain repeated eta-quotients. For example, $\eta(\tau)^{-1}\eta(2\tau)^2=\eta(\tau)^{-1}\eta(2\tau)^2\eta(4\tau)^0$ is listed both in the groups corresponding to $N=2,k=1/2$ and $N=4,k=1/2$ there.
\end{rema}

\begin{rema}
One should notice, e.g., that the formal product $1^{0}3^{0}5^{0}15^{2}$ which represents $\eta(15\tau)^2$ is in Table \ref{table:admissibleTypeI} but $1^{0}3^{0}5^{0}15^{1}$ is not. This is because $\dim_\numC M_{1/2}(\Gamma_0(15),\chi)$ can not be computed using Theorem \ref{thm:main} (with $t=0$) for any $\chi$ but $\dim_\numC M_{1}(\Gamma_0(15),\chi)$ can with $\chi$ being the character of $\eta(15\tau)^2$ as one may directly see by comparing both sides of the first inequality of \eqref{eq:mainThmCondtion2}. See also Table \ref{table:wt1/2} and Table \ref{table:wt1}. Nevertheless, we can prove that $\eta(15\tau)$ is weakly admissible, as the following proposition shows.
\end{rema}

\begin{prop}
Let $f$ be an eta-quotient (of integral exponents). If there is a positive integer $n$ such that $f^n$ is weakly admissible, then $f$ is weakly admissible. As a consequence, The $m$th root, where $m\in\numgeq{Z}{1}$, of an eta-quotient $\prod\eta(n\tau)^{r_n}$ in Table \ref{table:admissibleTypeI} is weakly admissible provided that $m\mid r_n$ for all $n$.
\end{prop}
\begin{proof}
Let $N$, $k$ be the level and weight of $f$ respectively. Let $\chi\colon\widetilde{\Gamma_0(N)}\rightarrow\numC^\times$ be the character of $f$. Then $\chi^n$ is the character\footnote{Here $\chi^n$ means the character on $\widetilde{\Gamma_0(N)}$ that sends $\left(\gamma,\varepsilon\right)$ to $\chi\left(\gamma,\varepsilon\right)^n$ where $\gamma\in\Gamma_0(N)$ and $\varepsilon\in\{\pm1\}$.} of $f^n$. First suppose $f^n$ is weakly admissible of type I, which means $f^n$ is holomorphic at all cusps and $\dim_\numC M_{nk}(\Gamma_0(N),\chi^n)=1$. Then obviously $f$ is holomorphic at all cusps so $f\in M_{k}(\Gamma_0(N),\chi)$. To prove $\dim_\numC M_{k}(\Gamma_0(N),\chi)=1$, let $0\neq g\in M_{k}(\Gamma_0(N),\chi)$ be arbitrary. Then $g^n\in M_{nk}(\Gamma_0(N),\chi^n)$ and hence $g^n=c\cdot f^n$ for some $c\in\numC^\times$. It follows that $g=c_1\cdot f$ where $c_1$ is an $n$th root of $c$ since $\uhp$ is connected and $f$ is holomorphic on $\uhp$. This proves that $f$ is weakly admissible of type I. The case of type II is proved in a similar manner with spaces such as $M_{k}(\Gamma_0(N),\chi)$ replaced by $S_{k}(\Gamma_0(N),\chi)$.
\end{proof}
\begin{rema}
It should be noted that if $f^n$ is weakly admissible of type II, then we can only assert that $f$ is weakly admissible; we can say nothing about whether $f$ is of type I or II. For instance, consider $f(\tau)=\eta(\tau)$, $n=24$.
\end{rema}

\subsection{Generalized double coset operators}
We recall the theory of generalized double coset operators developed in \cite[Section 3]{ZZ23}. This is a generalization of Wohlfahrt's extension of Hecke operators; cf. \cite{Woh57}. Let $l$, $N$ be positive integers. Let $\chi$ be the character \eqref{eq:charEtaQuotient} where $D=2$ and $r_n$ are integers. Let $(r_n')_{n\mid N}$ be another sequence of integers such that $\sum_nr_n=\sum_nr_n'$ and let $\chi'$ be the character \eqref{eq:charEtaQuotient} with $r_n$ replaced by $r_n'$ (again $D=2$). We define a formal expression (see \cite[Eq. (28) and (30)]{ZZ23})
\begin{equation}
\label{eq:Tlf}
T_lf(\tau)=l^{-\frac{k}{2}}\cdot\sum_{\twoscript{a\mid l}{(N,a)=1}}a^k\sum_{\twoscript{0\leq b<d}{(a,b,d)=1}}\chi\widetilde{\tbtmat{-Nb+ax}{z}{-Nd}{y}}^{-1}\chi'\widetilde{\tbtmat{ay}{by-dz}{N}{x}}^{-1}f\left(\frac{a\tau+b}{d}\right)
\end{equation}
where $k\in\frac{1}{2}\numZ$, $d=\frac{l}{a}$ and $x,y,z$ are any integers (depending on $a$, $b$ and $d$) such that $(Nd,-Nb+ax)=1$ and $(-Nb+ax)y+Ndz=1$. If the dependence on $k$, $N$, $\chi$ and $\chi'$ is important, then we also denote $T_l$ by $T_{l;k,N,\chi,\chi'}$ or $T_{l;\chi,\chi'}$.
\begin{thm}
\label{thm:doubleCosetOperator}
For $f$ being a meromorphic modular form of weight $k\in\frac{1}{2}\numZ$ for the group $\Gamma_0(N)$ with character $\chi$, $T_lf$ is well-defined (i.e., independent of the choices of $x,y,z$) and is a meromorphic modular form of the same weight for the same group with character $\chi'$ provided that the following three conditions hold:
\begin{gather} 
l\sum_{n \mid N}\frac{N}{n}r_n \equiv \sum_{n \mid N}\frac{N}{n}r_n' \pmod{24},\label{eq:charCom1}\\
\sum_{n \mid N}nr_n \equiv l\sum_{n \mid N}nr_n' \pmod{24},\label{eq:charCom2}\\
l^{2\abs{k'}}\cdot\prod_{2 \nmid r_n-r_n'}n \text{ is a perfect square,}\label{eq:charCom3}
\end{gather}
where $k'=\frac{1}{2}\sum_nr_n$. Moreover, if the above conditions hold, then $T_l$ maps $M_k(\Gamma_0(N),\chi)$ into $M_k(\Gamma_0(N),\chi')$ and $S_k(\Gamma_0(N),\chi)$ into $S_k(\Gamma_0(N),\chi')$ respectively.
\end{thm}
\begin{proof}
This is a combination of the last assertion of \cite[Theorem 3.9]{ZZ23}, \cite[Proposition 3.2(4),(5)]{ZZ23} and \cite[Eq. (28), (30)]{ZZ23}.
\end{proof}
\begin{thm}
\label{thm:Tlf1f2}
Let $f_1=\prod_{n}\eta(n\tau)^{r_n}$ and $f_2=\prod_{n}\eta(n\tau)^{r'_n}$ be weakly admissible eta-quotients of the same type (e.g., functions listed in Table \ref{table:admissibleTypeI}) of the same weight $k$ and the same level $N$. Let $\chi$ and $\chi'$ be the characters of $f_1$ and $f_2$ on $\widetilde{\Gamma_0(N)}$ respectively. Then for any $l\in\numgeq{Z}{1}$ satisfying \eqref{eq:charCom1}, \eqref{eq:charCom2} and \eqref{eq:charCom3} (with $k'=k$), there exists a $c_l\in\numC$ such that
\begin{equation*}
T_{l;\chi,\chi'}f_1=c_l\cdot f_2.
\end{equation*}
\end{thm}
\begin{proof}
An immediate consequence of Theorem \ref{thm:doubleCosetOperator}.
\end{proof}
\begin{rema}
This theorem is still valid in the following two cases:
\begin{itemize}
  \item $f_1$ is weakly admissible of type II and $f_2$ is weakly admissible of type I,
  \item $f_1$ is a cusp form that is weakly admissible of type I and $f_2$ is weakly admissible of type II.
\end{itemize}
\end{rema}
\begin{rema}
In the case $N=1$ (hence $r_1=r'_1$, $k=\frac{1}{2}r_1$ and $\chi=\chi'$) the identity reads $T_{l}\eta^{r_1}=c_l\cdot\eta^{r_1}$ where $r_1=0,1,2,\dots24$. This has been investigated in detail in \cite[Section 5]{ZZ23} where $c_l$ and \eqref{eq:Tlf} are given in explicit forms.
\end{rema}

We are mainly interested in the case $f_1=f_2$ in Theorem \ref{thm:Tlf1f2}.
\begin{coro}
\label{coro:Lfmf}
Let $f(\tau)=\prod_{n}\eta(n\tau)^{r_n}$ be a weakly admissible eta-quotient (e.g., a function listed in Table \ref{table:admissibleTypeI}) of weight $k=\frac{1}{2}\sum_nr_n$ and level $N$. We define a submonoid of the multiplicative monoid of positive integers as follows:
\begin{equation*}
L_f=\begin{dcases}
\left\{l\in\numgeq{Z}{1}\colon l\equiv1\bmod{m_f}\right\} & \text{ if } k\in\numZ,\\
\left\{l\in\numgeq{Z}{1}\colon l\equiv1\bmod{m_f}\text{ and }l\text{ is a perfect square}\right\} & \text{ if } k\in\frac{1}{2}+\numZ,
\end{dcases}
\end{equation*}
where $$m_f=\frac{24}{(24,\sum_n(N/n)r_n,\sum_nnr_n)}.$$
Then for any $l\in L_f$ we have $T_lf=c_l\cdot f$ for some $c_l\in\numC$.
\end{coro}
\begin{proof}
Elementary number theory shows that $l\in L_f$ if and only if \eqref{eq:charCom1}, \eqref{eq:charCom2} and \eqref{eq:charCom3} with $r'_n=r_n$ and $k'=k$ hold. Therefore the assertion follows from Theorem \ref{thm:Tlf1f2} with $f_1=f_2=f$.
\end{proof}
\begin{rema}
\label{rema:mrLr}
The monoid $L_f$ and the integer $m_f$ actually depend only on $\mathbf{r}=(r_n)_{n\mid N}$, so we sometimes write $L_\mathbf{r}$ and $m_\mathbf{r}$ instead.
\end{rema}
In another words, any weakly admissible eta-quotient $f$ is an eigenform for the operators $T_l$ indexed by the infinite monoid $L_f$. When $k\in\numZ$ and the character of $f$ is induced by some Dirichlet character, then by a result of Gordon, Hughes and Newman (cf. \cite[Theorem 1.64]{Ono04}) we have $m_f=1$ and hence $L_f=\numgeq{Z}{1}$. Moreover, in this case, $T_l$ is proportional to the usual Hecke operator of index $l$. Consequently, the functions in Table \ref{table:admissibleTypeI} are all Hecke eigenforms for appropriate infinite operator monoids in a generalized sense.

The identities $T_lf=c_l\cdot f$ become interesting only when we can express \eqref{eq:Tlf} in a more explicit form, at least one where $x,y,z$ do not appear. The following theorem gives such a formula which is key to the current topic. The reader may compare this with a previous result of Wohlfahrt \cite[Eq. (7$\vert$3) and (10$\vert$3)]{Woh57}.
\begin{thm}
\label{thm:cTlfbycf}
Let $N\in\numgeq{Z}{1}$. For each $n\mid N$, let $r_n$ be an integer and set 
\begin{equation}
\label{eq:notationsrn}
\mathbf{r}=(r_n)_{n\mid N},\quad k'=\frac{1}{2}\sum_{n\mid N}r_n,\quad x_N=\sum_{n\mid N}nr_n,\quad \varPi=\prod_{n\mid N}(Nn^{-1})^{\abs{r_n}}.
\end{equation}
Let $\chi_\mathbf{r}\colon\widetilde{\Gamma_0(N)}\rightarrow\numC^\times$ be the character of $\prod_{n\mid N}\eta(n\tau)^{r_n}$, i.e. \eqref{eq:charEtaQuotient}. Let $f$ be a meromorphic modular form on $\Gamma_0(N)$ of weight $k\in k'+2\numZ$ with character $\chi_{\mathbf{r}}$ and let $l\in L_\mathbf{r}$. Then
\begin{enumerate}
  \item we have the expansions (on $\Im\tau>Y_0$)
  \begin{equation*}
  f(\tau)=\sum_{n\in\frac{x_N}{24}+\numZ}c_f(n)q^n,\quad T_lf(\tau)=\sum_{n\in\frac{x_N}{24}+\numZ}c_{T_lf}(n)q^n,
  \end{equation*}
  where $Y_0$ is some nonnegative number, $c_f(n),\,c_{T_lf}(n)\in\numC$ are uniquely determined and $T_l$ is the operator defined in \eqref{eq:Tlf} with $\chi=\chi'=\chi_\mathbf{r}$;
  \item for $n\in\frac{x_N}{24}+\numZ$ we have
  \begin{equation}
  \label{eq:cTlfbycf}
  c_{T_lf}(n)=l^{-\frac{k}{2}}\sum_{\twoscript{a\mid l,\,d=l/a}{(a,N)=1}}\legendre{a}{\varPi}a^k\cdot c_f\left(\frac{ln}{a^2}\right)\sum_{\twoscript{0\leq b<d}{(a,b,d)=1}}\legendre{-Nb}{(a,d)}^{2k}\etp{bd\left(\frac{n}{l}-\frac{x_N}{24}\right)}\psi_{l,\mathbf{r}}(a,b)
  \end{equation}
  where
  \begin{equation*}
  \psi_{l,\mathbf{r}}(a,b)=\begin{dcases}
  \etp{-\frac{k+\delta}{4}(d-1)+\frac{(k+\delta)(l-1)(N-1)}{4}} &\text{ if } 2\nmid l,\\
  \etp{-\frac{k+\delta}{4}(a-1)-\frac{N(k+\delta)(1+\delta_1)}{4}b} &\text{ if } 2\mid l,\,2\mid N,\\
  1 &\text{ if } 2\mid l,\,2\nmid N,
  \end{dcases}
  \end{equation*}
  with $\delta=0$ ($\delta=1$ respectively) if $\varPi$ takes the form $2^\alpha\cdot(4m+1)$ ($2^\alpha\cdot(4m+3)$ respectively), $\alpha,\,m\in\numgeq{Z}{0}$, $\delta_1=1$ if $2\mid l$, $4\mid N$, $k\in\frac{1}{2}+\numZ$, $v_2(\varPi)\equiv1\bmod{2}$ and $\delta_1=0$ otherwise.
\end{enumerate}
\end{thm}
A few words of caution are in order. Note that $\legendre{a}{\varPi}$ and $\legendre{-Nb}{(a,d)}$ refer to Kronecker-Jacobi symbols; they are not fractions. Also note that in \eqref{eq:cTlfbycf} we have set $c_f\left(\frac{ln}{a^2}\right)=0$ if $\frac{ln}{a^2}\not\in\frac{x_N}{24}+\numZ$. In addition, $k$ can be replaced by $k'$ in the definition of $\psi_{l,\mathbf{r}}(a,b)$ and if $f$ is the eta-quotient $\prod_{n\mid N}\eta(n\tau)^{r_n}$ itself, which is the case we will use below, then $k=k'$. Finally, $Y_0=0$ if $f$ is holomorphic on $\uhp$.
\begin{rema}
\label{rema:Tlfexplicit}
The formula \eqref{eq:cTlfbycf} is equivalent to
\begin{equation}
\label{eq:Tlfexplicit}
T_lf(\tau)=l^{-\frac{k}{2}}\sum_{\twoscript{a\mid l,\,d=l/a}{(a,N)=1}}\legendre{a}{\varPi}a^k\sum_{\twoscript{0\leq b<d}{(a,b,d)=1}}\legendre{-Nb}{(a,d)}^{2k}\etp{-\frac{x_N}{24}bd}\psi_{l,\mathbf{r}}(a,b)\cdot f\left(\frac{a\tau+b}{d}\right)
\end{equation}
which can be considered as an explicit form of \eqref{eq:Tlf} with $\chi=\chi'$.
\end{rema}
\begin{proof}[Proof of Theorem \ref{thm:cTlfbycf} and Remark \ref{rema:Tlfexplicit}]
We begin with the proof of \eqref{eq:Tlfexplicit}. Set
$$
R:=\chi\widetilde{\tbtmat{-Nb+ax}{z}{-Nd}{y}}\chi'\widetilde{\tbtmat{ay}{by-dz}{N}{x}}=\chi_\mathbf{r}\widetilde{\tbtmat{(-Nb+ax)ay+Nz}{(-Nb+ax)(by-dz)+xz}{-Nady+Ny}{-Nd(by-dz)+xy}}
$$
in \eqref{eq:Tlf}. (Theorem \ref{thm:doubleCosetOperator} and the assumption $l\in L_\mathbf{r}$ ensure that $T_l$ is well-defined.) Since $ad=l$ and
\begin{equation}
\label{eq:xyz1}
(-Nb+ax)y+Ndz=1
\end{equation}
we have
\begin{equation*}
R=\chi_\mathbf{r}\widetilde{\tbtmat{a-N(l-1)z}{b-(l-1)xz}{-N(l-1)y}{d-(l-1)xy}}.
\end{equation*}
To derive \eqref{eq:Tlfexplicit} from \eqref{eq:Tlf}, we need only to prove
\begin{equation}
\label{eq:toProveIncTlfbycf}
R^{-1}=\legendre{a}{\varPi}\legendre{-Nb}{(a,d)}^{2k}\etp{-\frac{x_N}{24}bd}\psi_{l,\mathbf{r}}(a,b)
\end{equation}
for which we need the Petersson's formula of $\chi_\eta$:
\begin{equation}
\label{eq:etaCharPetersson}
\chi_\eta\left(\tbtmat{a}{b}{c}{d},\varepsilon\right)=\begin{cases}
\varepsilon\cdot\legendre{d}{\abs{c}}\etp{\frac{1}{24}\left((a-2d)c-bd(c^2-1)+(3d-3)c\right)}   & \text{if }2 \nmid c, \\
\varepsilon\cdot\legendre{c}{d}\etp{\frac{1}{24}\left((a-2d)c-bd(c^2-1)+3d-3\right)}   & \text{if }2 \mid c.
\end{cases}
\end{equation}
For a proof of \eqref{eq:etaCharPetersson} see \cite{Kno70}. Since $r_n\in\numZ$ we have
\begin{equation}
\label{eq:Rprod}
R=\prod_{n\mid N}\chi_\eta^{r_n}\widetilde{\tbtmat{a-N(l-1)z}{n(b-(l-1)xz)}{-\frac{N}{n}(l-1)y}{d-(l-1)xy}}.
\end{equation}
The strategy\footnote{The $N=1$ case has been successfully worked out in the proof of \cite[Lemma 5.2]{ZZ23}. The strategy here is essentially the one there.} is to insert \eqref{eq:etaCharPetersson} into \eqref{eq:Rprod} and then eliminate $x,\,y,\,z$ using properties of Kronecker-Jacobi symbols. Inserting \eqref{eq:etaCharPetersson} into \eqref{eq:Rprod} gives
\begin{multline}
\label{eq:Rexpression}
R=\prod_{\twoscript{n\mid N}{2\nmid\frac{N}{n}(l-1)y}}\legendre{d-(l-1)xy}{Nn^{-1}(l-1)\abs{y}}^{r_n}\cdot\prod_{\twoscript{n\mid N}{2\mid\frac{N}{n}(l-1)y}}\legendre{-Nn^{-1}(l-1)y}{d-(l-1)xy}^{r_n}\cdot\etp{\frac{x_N}{24}bd}\\
\cdot\prod_{\twoscript{n\mid N}{2\nmid\frac{N}{n}(l-1)y}}\etp{\frac{r_n}{8}(d-1-(l-1)xy)(-Nn^{-1})(l-1)y}\cdot\prod_{\twoscript{n\mid N}{2\mid\frac{N}{n}(l-1)y}}\etp{\frac{r_n}{8}(d-1-(l-1)xy)}
\end{multline}
where we have used the fact
\begin{equation}
\label{eq:24midxNandx1}
24\mid(l-1)\sum_{n\mid N}nr_n,\quad24\mid(l-1)\sum_{n\mid N}Nn^{-1}r_n
\end{equation}
which are consequences of the assumption $l\in L_{\mathbf{r}}$. Now the proof splits into three cases.

\textbf{The case $2\nmid l$.} We must have $2\mid\frac{N}{n}(l-1)y$ for any $n\mid N$. Therefore
\begin{equation}
\label{eq:inProofOfThmctlf_1}
R=\etp{\frac{x_N}{24}bd}\prod_{n\mid N}\legendre{-Nn^{-1}(l-1)y}{d-(l-1)xy}^{r_n}\prod_{n\mid N}\etp{\frac{r_n}{8}(d-1-(l-1)xy)}.
\end{equation}
Note that if $k'\in\frac{1}{2}+\numZ$ then $l$ is an odd square and hence
\begin{equation*}
\legendre{-Nn^{-1}(l-1)y}{d-(l-1)xy}=\legendre{Nn^{-1}}{d-(l-1)xy}\cdot\legendre{-(l-1)y}{d-(l-1)xy}=\legendre{Nn^{-1}}{d-(l-1)xy}\cdot\legendre{y}{d}.
\end{equation*}
Inserting this into \eqref{eq:inProofOfThmctlf_1} we find that
\begin{equation}
\label{eq:inProofOfThmctlf_2}
R=\etp{\frac{x_N}{24}bd}\cdot\legendre{\varPi}{d-(l-1)xy}\legendre{y}{d}^{2k'}\etp{\frac{k'}{4}(d-1-(l-1)xy)}.
\end{equation}
If $k'\in\numZ$, the above expression is still valid. By \eqref{eq:xyz1} if $k'\in\frac{1}{2}+\numZ$ we have
\begin{equation}
\label{eq:inProofOfThmctlf_3}
\legendre{y}{d}=\legendre{-Nb+ax}{d}=\legendre{-Nb}{(a,d)}.
\end{equation}
For the last equality, see \cite[p. 567, line -3]{ZZ23}. If $k'\in\numZ$, then trivially $\legendre{y}{d}^{2k'}=\legendre{-Nb}{(a,d)}^{2k'}$. For the calculation of $\legendre{\varPi}{d-(l-1)xy}$, we set
\begin{equation}
\label{eq:defvarPip}
\varPi'=\prod_{p\mid \varPi,\,2\nmid v_p(\varPi)}p.
\end{equation}
There are four cases to consider. For the first case, $\varPi'\equiv1\bmod{4}$, we have
\begin{equation}
\label{eq:inProofOfThmctlf_3h}
\legendre{\varPi}{d-(l-1)xy}=\legendre{\varPi'}{d-(l-1)xy}=\legendre{d-(l-1)xy}{\varPi'}=\legendre{a}{\varPi'}=\legendre{a}{\varPi}.
\end{equation}
The third equality above follows from
\begin{equation*}
\legendre{d-(l-1)xy}{p}=\legendre{d-lxy+xy}{p}=\legendre{xy}{p}=\legendre{a}{p},\quad p\mid \varPi'
\end{equation*}
since by \eqref{eq:xyz1} we have $axy\equiv1\bmod{p}$ and $lxy\equiv d\bmod{p}$. Moreover, we have
\begin{equation}
\label{eq:inProofOfThmctlf_4}
\etp{-\frac{k'(l-1)xy}{4}}=\etp{-\frac{k'(l-1)(N-1)}{4}}.
\end{equation}
(When $k'\in\numZ$, $l\equiv3\bmod{4}$ and $2\nmid N$, \eqref{eq:inProofOfThmctlf_4} does not necessarily hold, but we can always choose $x,y,z$ such that $2\mid y$ for which \eqref{eq:inProofOfThmctlf_4} holds. We always make such a choice.) Inserting \eqref{eq:inProofOfThmctlf_3}, \eqref{eq:inProofOfThmctlf_3h} and \eqref{eq:inProofOfThmctlf_4} into \eqref{eq:inProofOfThmctlf_2} and replacing $k'$ by $k$ we obtain \eqref{eq:toProveIncTlfbycf} as required. For the other three cases $\varPi'\equiv3\bmod{4}$, $\varPi'\equiv2\bmod{8}$ and $\varPi'\equiv6\bmod{8}$, we proceed similarly and obtain \eqref{eq:toProveIncTlfbycf} likewise. However, if $2\mid \varPi'$ there is an extra difficulty: a factor $\legendre{2}{ad-(l-1)axy}=\legendre{2}{1+(l-1)N(dz-by)}$ appears and we should prove it equals $1$. This is immediate if $l\equiv1\bmod{4}$ since $2\mid \varPi'$ implies $2\mid N$. Now suppose $l\equiv3\bmod{4}$. Assume by contradiction that $4\nmid N$; we have $N\equiv2\bmod{4}$. By the first relation of \eqref{eq:24midxNandx1} we have $4\mid\sum_nnr_n$ from which $2\mid\sum_{2\nmid n}r_n$. Since $2\mid\varPi'$ we have
\begin{equation*}
v_2(\varPi)=\sum_{n\mid N}v_2(N/n)\cdot\abs{r_{n}}=\sum_{2\nmid n}\abs{r_n}\equiv1\bmod{2}
\end{equation*}
which contradicts $2\mid\sum_{2\nmid n}r_n$. Therefore, $4\mid N$ and hence $\legendre{2}{1+(l-1)N(dz-by)}=1$.

\textbf{The case $2\mid l$, $2\nmid N$.} By \eqref{eq:xyz1} if $2\mid d$ then $2\nmid y$ and if $2\nmid d$ then we can choose $x,y,z$ such that $2\nmid y$. (If $(x,y,z)$ is a solution to \eqref{eq:xyz1} with $2\mid y$, then $(x,y+Nd,z-(-Nb+ax))$ is another solution with $2\nmid y+Nd$.) We always make such a choice. Thus \eqref{eq:Rexpression} becomes
\begin{equation*}
R=\etp{\frac{x_N}{24}bd}\prod_{n\mid N}\legendre{d-(l-1)xy}{Nn^{-1}(l-1)\abs{y}}^{r_n}\prod_{n\mid N}\etp{\frac{r_n}{8}(d-1-(l-1)xy)(-Nn^{-1})(l-1)y}.
\end{equation*}
Since $l\equiv1\bmod{m_{\mathbf{r}}}$ we have $2\nmid m_{\mathbf{r}}$. By the definition of $m_{\mathbf{r}}$ (see Remark \ref{rema:mrLr}) we have
\begin{equation}
\label{eq:8midx1xN}
8\mid\sum_nnr_n,\quad 8\mid\sum_nNn^{-1}r_n.
\end{equation}
It follows that $k'=\frac{1}{2}\sum_nr_n\in\numZ$ and
\begin{equation}
\label{eq:inProofOfThmctlf_5}
R=\etp{\frac{x_N}{24}bd}\prod_{n\mid N}\legendre{d-(l-1)xy}{Nn^{-1}(l-1)\abs{y}}^{r_n}.
\end{equation}
Note that
\begin{equation*}
\legendre{d-(l-1)xy}{Nn^{-1}(l-1)\abs{y}}=\legendre{d-(l-1)xy}{Nn^{-1}}\cdot\legendre{d-(l-1)xy}{(l-1)\abs{y}}=\legendre{d-(l-1)xy}{Nn^{-1}}\cdot\legendre{d}{(l-1)\abs{y}}
\end{equation*}
where the last equality follows from the fact $(l-1)\abs{y}$ is odd and positive. Inserting this and $2k'\equiv0\bmod{2}$ into \eqref{eq:inProofOfThmctlf_5} we obtain
\begin{equation*}
R=\etp{\frac{x_N}{24}bd}\cdot\legendre{d-(l-1)xy}{\varPi}=\etp{\frac{x_N}{24}bd}\cdot\legendre{a}{\varPi}
\end{equation*}
where the last equality follows just as in the deduction of \eqref{eq:inProofOfThmctlf_3h}. We thus have arrived at \eqref{eq:toProveIncTlfbycf}.

\textbf{The case $2\mid l$, $2\mid N$.} As in the last case, we have $8\mid\sum_nnr_n$ and $8\mid\sum_nNn^{-1}r_n$. Moreover, $y$ must be odd by \eqref{eq:xyz1}. For any $n\mid N$ with $2\nmid\frac{N}{n}(l-1)y$, that is, with $2\nmid\frac{N}{n}$, we have
\begin{equation*}
\legendre{d-(l-1)xy}{Nn^{-1}(l-1)\abs{y}}=\legendre{Nn^{-1}(l-1)\abs{y}}{d-(l-1)xy}\cdot\etp{-\frac{1}{8}(d-1-(l-1)xy)\cdot(Nn^{-1}(l-1)\abs{y}-1)}.
\end{equation*}
Choosing a solution $(x,y,z)$ to \eqref{eq:xyz1} with $y<0$ and inserting the above identity into \eqref{eq:Rexpression}, we find that
\begin{multline*}
R=\etp{\frac{x_N}{24}bd}\cdot\prod_{n\mid N}\legendre{-Nn^{-1}(l-1)y}{d-(l-1)xy}^{r_n}\cdot\prod_{\twoscript{n\mid N}{2\nmid\frac{N}{n}}}\etp{-\frac{r_n}{8}(d-1-(l-1)xy)(-Nn^{-1}(l-1)y-1)}\\
\cdot\prod_{\twoscript{n\mid N}{2\nmid\frac{N}{n}}}\etp{\frac{r_n}{8}(d-1-(l-1)xy)(-Nn^{-1})(l-1)y}\cdot\prod_{\twoscript{n\mid N}{2\mid\frac{N}{n}}}\etp{\frac{r_n}{8}(d-1-(l-1)xy)}
\end{multline*}
We deal with the factor $\prod_{n\mid N}\legendre{-Nn^{-1}(l-1)y}{d-(l-1)xy}^{r_n}$ as in previous cases and thus simplify the expression for $R$ as
\begin{equation}
\label{eq:R2midl2midN}
R=\etp{\frac{x_N}{24}bd}\cdot\legendre{\varPi}{d-(l-1)xy}\cdot\legendre{d}{-(l-1)y}^{2k'}\cdot\etp{\frac{k'}{4}(d-1-(l-1)xy)\cdot(-(l-1)y)}. 
\end{equation}
There are six subcases: (see \eqref{eq:defvarPip} for $\varPi'$)
\begin{enumerate}
  \item $2\mid l$, $4\mid N$, $k'\in\numZ$ and $2\nmid\varPi'$.
  \item $2\mid l$, $4\mid N$, $k'\in\numZ$ and $2\mid\varPi'$.
  \item $2\mid l$, $4\mid N$, $k'\in\frac{1}{2}+\numZ$ and $2\nmid\varPi'$.
  \item $2\mid l$, $4\mid N$, $k'\in\frac{1}{2}+\numZ$ and $2\mid\varPi'$. (Only in this subcase we have $\delta_1=1$.)
  \item $2\mid l$, $N\equiv2\bmod{4}$ and $2\nmid\varPi'$. (We must have $k'\in\numZ$ in this subcase.)
  \item $2\mid l$, $N\equiv2\bmod{4}$ and $2\mid\varPi'$. (We must have $k'\in\numZ$ in this subcase.)
\end{enumerate}
We will only give the proof of the more difficult subcases (b) and (d) and omit the other four since the strategies are the same. For (b), we first prove $8\mid N$. Assume by contradiction that $8\nmid N$; then $2^2\parallel N$. It follows that
\begin{equation*}
v_2(\varPi)\equiv\sum_{\twoscript{n\mid N}{v_2(n)=1}}v_2(N/n)\cdot\abs{r_{n}}=\sum_{\twoscript{n\mid N}{v_2(n)=1}}\abs{r_n}\equiv1\bmod{2}.
\end{equation*}
On the other hand, \eqref{eq:8midx1xN} still holds in this case, and hence
\begin{equation*}
\sum_{\twoscript{n\mid N}{v_2(n)=0}}r_n\equiv\sum_{\twoscript{n\mid N}{v_2(n)=2}}r_n\equiv0\bmod{2}.
\end{equation*}
Therefore $\sum_{n\mid N}r_n\equiv1\bmod{2}$, that is, $k'\in\frac{1}{2}+\numZ$, which contradicts the assumption in (b). Hence $8\mid N$. Note that both $\varPi'/2$ and $d-(l-1)xy$ are odd and positive, from which we deduce that
\begin{align}
\legendre{\varPi}{d-(l-1)xy}&=\legendre{\varPi'}{d-(l-1)xy}\notag\\
&=\legendre{2}{d-(l-1)xy}\legendre{d-(l-1)xy}{\varPi'/2}\etp{\frac{1}{4}(d-1-(l-1)xy)\cdot\frac{\varPi'/2-1}{2}}\notag\\
&=\legendre{2}{ad-(l-1)axy}\legendre{a}{\varPi'}\etp{\frac{1}{4}(d-1-(l-1)xy)\cdot\frac{\varPi'/2-1}{2}}\notag\\
&=\legendre{2}{1+(l-1)N(dz-by)}\legendre{a}{\varPi}\etp{\frac{1}{4}(d-1-(l-1)xy)\cdot\frac{\varPi'/2-1}{2}}\notag\\
&=\legendre{a}{\varPi}\etp{\frac{1}{4}(a-1)\cdot\frac{\varPi'/2-1}{2}}.\label{eq:inProofOfThmctlf_6}
\end{align}
The last equality holds since by \eqref{eq:xyz1} and $8\mid N$ we have $axy\equiv1\bmod{4}$ and hence
\begin{equation*}
(a-1)-(d-1-(l-1)xy)\equiv a-d+(l-1)a=(a^2-1)d\equiv0\bmod{4}.
\end{equation*}
An argument similar to the above deduction shows that
$$
\etp{\frac{k'}{4}(d-1-(l-1)xy)\cdot(-(l-1)y)}=\etp{\frac{k'}{4}(a-1)(l-1)}.
$$
Inserting this and \eqref{eq:inProofOfThmctlf_6} into \eqref{eq:R2midl2midN} and noting that $a$ is odd, we obtain \eqref{eq:toProveIncTlfbycf} in the subcase (b). For the subcase (d), $l$ is a square. Since $(N,a)=1$ and $ad=l$ we have $v_2(d)=v_2(l)\equiv0\bmod{2}$ and hence $4\mid d$. Thus if we let $d'=d/2^{v_2(d)}$ then
\begin{equation*}
\legendre{d}{-(l-1)y}=\legendre{d}{y}=\legendre{d'}{y}=\etp{\frac{y-1}{4}\cdot\frac{d'-1}{2}}\cdot\legendre{y}{d}.
\end{equation*}
Combining this and
\begin{equation*}
\legendre{-Nb+ax}{d}\legendre{y}{d}=\legendre{(-Nb+ax)y}{d}=\legendre{1-Ndz}{d}=1
\end{equation*}
we find that
\begin{equation}
\label{eq:inProofOfThmctlf_7}
\legendre{d}{-(l-1)y}=\etp{\frac{y-1}{4}\cdot\frac{d'-1}{2}}\cdot\legendre{-Nb+ax}{d}=\etp{\frac{y-1}{4}\cdot\frac{d'-1}{2}}\cdot\legendre{-Nb}{(a,d)}.
\end{equation}
The calculation of $\legendre{\varPi}{d-(l-1)xy}$ is similar to that in above cases (e.g. \eqref{eq:inProofOfThmctlf_3h} or \eqref{eq:inProofOfThmctlf_6}) and the result is
\begin{equation}
\label{eq:inProofOfThmctlf_8}
\legendre{\varPi}{d-(l-1)xy}=\legendre{2}{1+Nb}\legendre{a}{\varPi}\etp{\frac{1}{4}(l-1)(a-1)\cdot\frac{\varPi'/2-1}{2}}.
\end{equation}
For the factor $\etp{\frac{k'}{4}(d-1-(l-1)xy)\cdot(-(l-1)y)}$ in \eqref{eq:R2midl2midN}, note that $l-d=(a-1)d\equiv0\bmod{8}$, $2\nmid y$ and hence
\begin{align*}
\etp{\frac{k'}{4}(d-1-(l-1)xy)\cdot(-(l-1)y)}&=\etp{\frac{k'}{4}(l-1)^2(1-xy)(-y)}\\
&=\etp{\frac{k'}{4}(1-xy)(-y)}=\etp{\frac{k'}{4}(x-y)}.
\end{align*}
Inserting this equality, \eqref{eq:inProofOfThmctlf_7} and \eqref{eq:inProofOfThmctlf_8} into \eqref{eq:R2midl2midN} we obtain
\begin{multline}
\label{eq:inProofOfThmctlf_8hh}
R=\legendre{a}{\varPi}\legendre{-Nb}{(a,d)}\etp{\frac{x_N}{24}bd}\legendre{2}{1+Nb}\\
\cdot\etp{\frac{1}{4}(l-1)(a-1)\cdot\frac{\varPi'/2-1}{2}}\etp{\frac{y-1}{4}\cdot\frac{d'-1}{2}+\frac{k'}{4}(x-y)}.
\end{multline}
Now we simplify the factor $\etp{\frac{y-1}{4}\cdot\frac{d'-1}{2}+\frac{k'}{4}(x-y)}$ in above. By \eqref{eq:xyz1} and the facts $2\mid d$, $4\mid N$ we find that $(-Nb+ax)y\equiv1\bmod{8}$. This, together with $2\nmid y$, implies $x\equiv ay+Nab\bmod{8}$. It follows that
\begin{align}
\etp{\frac{y-1}{4}\cdot\frac{d'-1}{2}+\frac{k'}{4}(x-y)}&=\etp{\frac{-2k'}{4}(y-1)\cdot\frac{d'-1}{2}+\frac{k'}{4}(ay-y+Nab)}\notag\\
&=\etp{\frac{k'}{4}Nab}\etp{\frac{k'}{4}y(a-d')}\etp{\frac{k'}{4}(d'-1)}.\label{eq:inProofOfThmctlf_8h}
\end{align}
Since $ad'=l/2^{v_2(d)}=l/2^{v_2(l)}$ is a square, $2\nmid a$ and $2\nmid d'$ we have
\begin{equation}
\label{eq:inProofOfThmctlf_9}
a-d'\equiv0\bmod{8}.
\end{equation}
In addition, since $k'\in\frac{1}{2}+\numZ$, $2\nmid a$ and $4\mid N$,
\begin{equation}
\label{eq:inProofOfThmctlf_10}
\etp{\frac{k'}{4}Nab}=\etp{\frac{1}{8}Nab}=\etp{\frac{1}{8}Nb}.
\end{equation}
Inserting \eqref{eq:inProofOfThmctlf_9} and \eqref{eq:inProofOfThmctlf_10} into \eqref{eq:inProofOfThmctlf_8h}, then combining \eqref{eq:inProofOfThmctlf_8hh} and noting that $\etp{\frac{1}{8}Nb}\legendre{2}{1+Nb}=1$, we arrive at \eqref{eq:toProveIncTlfbycf} as required. This concludes the proof of \eqref{eq:Tlfexplicit}.

Finally, we begin to prove Theorem \ref{thm:cTlfbycf}. For Theorem \ref{thm:cTlfbycf}(a), elementary Fourier analysis shows that
\begin{equation*}
f(\tau)=\sum_{n\in\frac{1}{24}\numZ}c_f(n)q^n,\quad T_lf(\tau)=\sum_{n\in\frac{1}{24}\numZ}c_{T_lf}(n)q^n,
\end{equation*}
which converge on $\Im\tau>Y_0$ for some $Y_0\geq0$ since $f(\tau+24)=f(\tau)$ and $T_lf(\tau+24)=T_lf(\tau)$. By modularity $f\vert_k\widetilde{\tbtmat{1}{1}{0}{1}}=\chi_{\mathbf{r}}\widetilde{\tbtmat{1}{1}{0}{1}}\cdot f=\etp{\frac{x_N}{24}}\cdot f$ and $T_lf\vert_k\widetilde{\tbtmat{1}{1}{0}{1}}=\chi_{\mathbf{r}}\widetilde{\tbtmat{1}{1}{0}{1}}\cdot T_lf=\etp{\frac{x_N}{24}}\cdot T_lf$. In terms of Fourier coefficients,
\begin{equation*}
c_f(n)\etp{n}=c_f(n)\etp{\frac{x_N}{24}},\quad c_{T_lf}(n)\etp{n}=c_{T_lf}(n)\etp{\frac{x_N}{24}}.
\end{equation*}
It follows that $c_f(n)=0$ unless $n\in\frac{x_N}{24}+\numZ$ and so is $c_{T_lf}(n)$. Theorem \ref{thm:cTlfbycf}(b), that is, \eqref{eq:cTlfbycf}, follows immediately from \eqref{eq:Tlfexplicit} whose details are omitted.
\end{proof}

In the important case $x_N\equiv0\mod{24}$ we have $c_f(n)=0$ and $c_{T_lf}(n)=0$ unless $n\in\numZ$. Therefore, \eqref{eq:cTlfbycf} becomes
\begin{equation}
\label{eq:cTlfbycfxN0}
c_{T_lf}(n)=l^{-\frac{k}{2}}\sum_{\twoscript{a\mid l,\,a^2\mid ln}{(a,N)=1,\,d=l/a}}\legendre{a}{\varPi}a^k\cdot c_f\left(\frac{ln}{a^2}\right)\sum_{\twoscript{0\leq b<d}{(a,b,d)=1}}\legendre{-Nb}{(a,d)}^{2k}\etp{\frac{bdn}{l}}\psi_{l,\mathbf{r}}(a,b)
\end{equation}
for $n\in\numZ$. From this some elegant formulas for special cases follow:
\begin{coro}
We keep the notations and assumptions of Theorem \ref{thm:cTlfbycf} and assume that $x_N\equiv0\mod{24}$. Then
\begin{equation*}
c_{T_lf}(0)=c_f(0)\cdot l^{-\frac{k}{2}}\sum_{\twoscript{a\mid l,\,d=l/a}{(a,N)=1}}\legendre{a}{\varPi}a^k\sum_{\twoscript{0\leq b<d}{(a,b,d)=1}}\legendre{-Nb}{(a,d)}^{2k}\psi_{l,\mathbf{r}}(a,b).
\end{equation*}
Moreover, if for any prime $p$ not dividing $N$ we have $p^2\nmid l$ (e.g. $l$ is square-free or $\rad(l)\mid\rad(N)$), then
\begin{equation}
\label{eq:cTlfwhenlSquareFree}
c_{T_lf}(1)=c_f(l)\cdot l^{-\frac{k}{2}}\sum_{0\leq b<l}\psi_{l,\mathbf{r}}(1,b).
\end{equation}
Finally, if $l=p^\alpha$ where $p$ is a prime not dividing $N$ and $\alpha\in\numgeq{Z}{1}$, then
\begin{equation*}
c_{T_{p^\alpha}f}(1)=p^{-\frac{\alpha k}{2}}\sum_{0\leq\beta\leq[\alpha/2]}\legendre{p}{\varPi}^\beta p^{k\beta}\cdot c_f(p^{\alpha-2\beta})\sum_{\twoscript{0\leq b<p^{\alpha-\beta}}{(p^\beta,b)=1}}\legendre{-Nb}{p^\beta}^{2k}\etp{\frac{b}{p^\beta}}\psi_{l,\mathbf{r}}(p^\beta,b).
\end{equation*}
\end{coro}
\begin{rema}
The sum over $b$ in the above formulas, \eqref{eq:cTlfbycfxN0} or \eqref{eq:cTlfbycf} is either a partial sum of a geometric sequence or a Gauss sum according to whether $k$ is integral or half-integral. It is not hard to derive an explicit formula for this sum but we exclude this task here. For the $N=1$ case the reader may consult \cite[Eq. (58)]{ZZ23} and \cite[Lemma 5.6]{ZZ23} for such formulas.
\end{rema}

\subsection{Identities involving Fourier coefficients of weakly admissible eta-quotients}
\label{subsec:identitiesFourier}
The main theorem of Section \ref{sec:an extension of Martin's list of multiplicative eta-quotients} is the following:
\begin{thm}
\label{thm:mainApp2}
Let $f(\tau)=\prod_{n\mid N}\eta(n\tau)^{r_n}=\sum_{n}c_f(n)q^n$ be a weakly admissible eta-quotient (e.g., a function listed in Table \ref{table:admissibleTypeI}) where $N$ is the level. We keep the notations in \eqref{eq:notationsrn} with $k'=k$. Then for any $l\in L_f$ (see Corollary \ref{coro:Lfmf}) there exists a unique $c_l\in\numC$ such that
\begin{equation}
\label{eq:mainApp2}
l^{-\frac{k}{2}}\sum_{\twoscript{a\mid l,\,d=l/a}{(a,N)=1}}\legendre{a}{\varPi}a^k\cdot c_f\left(\frac{ln}{a^2}\right)\sum_{\twoscript{0\leq b<d}{(a,b,d)=1}}\legendre{-Nb}{(a,d)}^{2k}\etp{bd\left(\frac{n}{l}-\frac{x_N}{24}\right)}\psi_{l,\mathbf{r}}(a,b)=c_l\cdot c_f(n)
\end{equation}
holds for any $n\in\frac{x_N}{24}+\numZ$.
Moreover, the number $c_l$ is equal to the left-hand side of the above identity with $n$ replaced by $\frac{x_N}{24}$.
\end{thm}
\begin{proof}
By Corollary \ref{coro:Lfmf} $T_lf=c_l\cdot f$ and hence $c_{T_lf}(n)=c_l\cdot c_f(n)$ for any $n\in\frac{x_N}{24}+\numZ$. Inserting \eqref{eq:cTlfbycf} into this identity gives the desired formula. Letting $n=\frac{x_N}{24}$ and noting that $c_f\left(\frac{x_N}{24}\right)=1$ we obtain the formula for $c_l$.
\end{proof}

The identities \eqref{eq:mainApp2} have been verified numerically for $l\leq121$, $n-\frac{x_N}{24}\leq40$ and $f$ being any function listed in Table \ref{table:admissibleTypeI} by a SageMath program. See Appendix \ref{apx:Usage of SageMath code} for the code. The rest of this section is devoted to useful consequences of Theorem \ref{thm:mainApp2} and concrete examples.

\subsubsection{Closed formulas for Fourier coefficients of eta-quotients.}\label{subsuesec:closedFourier} Let $f(\tau)=\prod_{n\mid N}\eta(n\tau)^{r_n}$ be a weakly admissible eta-quotient (e.g., a function listed in Table \ref{table:admissibleTypeI}) where $N$ is the level. We keep the notations in \eqref{eq:notationsrn} with $k'=k$. Here let us consider the situation $x_N=0$ where we have
\begin{equation*}
f(\tau)=\prod_{n\mid N}\prod_{m\in\numgeq{Z}{1}}(1-q^{nm})^{r_n}=\sum_{n\in\numgeq{Z}{0}}c_f(n)q^n,\quad\tau\in\uhp.
\end{equation*}
We have given a series expression of $c_f(n)$ involving Dedekind sums for $N=4$, $r_1=2$, $r_2=7$, $r_4=-4$ in Example \ref{examp:N4rn274}. This example is of course also listed in Table \ref{table:admissibleTypeI}. In fact, there exist very simple closed formulas for $c_f(n)$ only involving Kronecker-Jacobi symbols at least for $n$ square-free---thanks to Theorem \ref{thm:mainApp2}.
\begin{thm}
\label{thm:closedFormulaclf}
Let the notations and assumptions be as above. Let $l\in L_f$ (see Corollary \ref{coro:Lfmf}) satisfy
\begin{equation}
\label{eq:pnp2l}
p\nmid N\implies p^2\nmid l,\quad\text{for any prime } p.
\end{equation}
(e.g. $l$ is square-free or $\rad(l)\mid\rad(N)$.) We require that at least one of $l$ and $N$ is odd. Then
\begin{equation*}
c_f(l)=\begin{dcases}
-r_1\cdot\sum_{a\mid l,\,(a,N)=1}\legendre{a}{\varPi}a^{k-1} &\text{ if } 2\mid l,\,2\nmid N,\\
-r_1\cdot\sum_{a\mid l,\,(a,N)=1}\legendre{a}{\varPi}a^{k-1}\etp{\frac{(k+\delta)(l-d)}{4}} &\text{ if } 2\nmid l,
\end{dcases}
\end{equation*}
where $d=l/a$, $\varPi=\prod_{n\mid N}(Nn^{-1})^{\abs{r_n}}$ and $\delta=0$ ($\delta=1$ respectively) if $\varPi$ takes the form $2^\alpha\cdot(4m+1)$ ($2^\alpha\cdot(4m+3)$ respectively). In particular, if $\rad(l)\mid\rad(N)$, then $c_f(l)=-r_1$.
\end{thm}
\begin{proof}
A combination of \eqref{eq:cTlfwhenlSquareFree} and \eqref{eq:mainApp2} (with $n=1$) gives $c_f(l)\cdot\sum_{0\leq b<l}\psi_{l,\mathbf{r}}(1,b)=l^{\frac{k}{2}}c_l\cdot c_f(1)$. By the last assertion of Theorem \ref{thm:mainApp2} we have
$$
c_l=l^{-\frac{k}{2}}\sum_{a\mid l,\,(a,N)=1}\legendre{a}{\varPi}a^k\cdot c_f(0)\sum_{0\leq b<d}\psi_{l,\mathbf{r}}(a,b).
$$
Since $c_f(0)=1$ and $c_f(1)=-r_1$ we find that
\begin{equation}
\label{eq:cflIdentityxN0}
c_f(l)\cdot\sum_{0\leq b<l}\psi_{l,\mathbf{r}}(1,b)=-r_1\cdot\sum_{a\mid l,\,(a,N)=1}\legendre{a}{\varPi}a^k\cdot\sum_{0\leq b<d}\psi_{l,\mathbf{r}}(a,b).
\end{equation}
The assumption at least one of $l$ and $N$ is odd ensures that $\sum_{0\leq b<l}\psi_{l,\mathbf{r}}(1,b)\neq0$. Therefore the desired formulas follow by expanding $\sum_{0\leq b<l}\psi_{l,\mathbf{r}}(1,b)$ and $\sum_{0\leq b<d}\psi_{l,\mathbf{r}}(a,b)$ in \eqref{eq:cflIdentityxN0}.
\end{proof}
\begin{rema}
The formulas for $c_f(l)$ in the above theorem seem like coefficients of Eisenstein series. This is actually the case since $x_N=0$ implies $f$ is not a cusp form. Since $f$ is weakly admissible, then it must be of type I, that is, $\dim_\numC M_k(\Gamma_0(N),\chi_\mathbf{r})=1$. Therefore $f$ is in the Eisenstein space.
\end{rema}
\begin{rema}
In the case of half-integral weights, Theorem \ref{thm:closedFormulaclf} can be used only for those $l$ with $\rad(l)\mid\rad(N)$. However, there is no half-integral weight eta-quotient $f$ in Table \ref{table:admissibleTypeI} such that $c_f(l)$ ($l>1$) can be computed using Theorem \ref{thm:closedFormulaclf}. Nevertheless, it is still worthwhile to include the case of half-integral weights in Theorem \ref{thm:closedFormulaclf} since there may exist weakly admissible eta-quotients not listed in Table \ref{table:admissibleTypeI} whose coefficients $c_f(l)$ can be computed using this theorem.
\end{rema}
\begin{examp}
\label{examp:level12}
Let us consider an example of level $12$ which is not listed in Table \ref{table:admissibleTypeI}. The reader may use SageMath programs described in Appendix \ref{apx:Usage of SageMath code} to produce this eta-quotient and check the assertions given here. Set
\begin{equation*}
f(\tau)=\eta(\tau)^{1}\eta(2\tau)^{-1}\eta(3\tau)^{-1}\eta(4\tau)^{1}\eta(6\tau)^4\eta(12\tau)^{-2}=\sum_{n\in\numgeq{Z}{0}}c_f(n)q^n.
\end{equation*}
Let $(x_c)_{c\mid 12}$ be as in Proposition \ref{prop:rnxc}. Hence $\div_{a/c}f=\frac{x_c}{24}$ for $c\mid 12$ and $a$ coprime to $c$; see \eqref{eq:divEtaQuotient}. We have
\begin{equation*}
\begin{pmatrix}
x_1 \\ x_2 \\ x_3 \\ x_4 \\ x_6 \\ x_{12}
\end{pmatrix}=
\begin{pmatrix}
12 & 6 & 4 & 3 & 2 & 1 \\
3  & 6 & 1 & 3 & 2 & 1 \\
4 & 2 & 12 & 1 & 6 & 3 \\
3 & 6 & 1 & 12 & 2 & 4 \\
1 & 2 & 3 & 1  & 6 & 3 \\
1 & 2 & 3 & 4  & 6 & 12
\end{pmatrix}
\begin{pmatrix}
r_1 \\ r_2 \\ r_3 \\ r_4 \\ r_6 \\ r_{12}
\end{pmatrix}=
\begin{pmatrix}
11 \\ 5 \\ 9 \\ 8 \\ 15 \\ 0
\end{pmatrix},
\end{equation*}
from which we see that $f$ is a holomorphic eta-quotient. Moreover, $f$ is admissible of type I. Thus Theorem \ref{thm:closedFormulaclf} can be applied to $f$. Note that $m_f$ in Corollary \ref{coro:Lfmf} is equal to $24$ and hence
$$
L_f=\{l\in\numgeq{Z}{1}\colon l\equiv1\bmod{24}\}.
$$
Therefore $(a,12)=1$ for any $a\mid l$ where $l\in L_f$. Moreover,
\begin{equation*}
k=1,\quad\varPi=2^9\cdot3^3=\text{square}\cdot6,\quad\delta=1.
\end{equation*}
It follows from Theorem \ref{thm:closedFormulaclf} that
\begin{equation*}
c_f(l)=-\sum_{a\mid l}\legendre{a}{6}=-\prod_{p\mid l}\left(1+\legendre{p}{6}\right)
\end{equation*}
where $1\leq l\equiv1\bmod{24}$ is square-free. Consequently, for such $l$, $-c_f(l)$ is either zero or a power of $2$. Moreover, for any $\alpha\in\numgeq{Z}{1}$ there exist infinitely many $l$ such that $c_f(l)=-2^\alpha$.
\end{examp}

\subsubsection{Multiplicativity of weakly admissible eta-quotients.} The Fourier coefficients $c_f(n)$ of an eta-quotient $f$ in Martin's list \cite[Table I]{Mar96} satisfy $c_f(nm)c_f(1)=c_f(n)c_f(m)$ whenever $n$, $m$ are coprime positive integers. This is the reason they are called multiplicative eta-quotients. The proper generalization to weakly admissible eta-quotients is stated in the following theorem.
\begin{thm}
\label{thm:multiplicativeEtaQuotients}
Let us keep the notations and assumptions of Theorem \ref{thm:mainApp2}. Then the map $L_f\rightarrow\numC$ that sends $l$ to $c_l$ is a multiplicative function, that is, we have $c_{l_1l_2}=c_{l_1}\cdot c_{l_2}$ for any $l_1,l_2\in L_f$ with\footnote{The gcd $(l_1,l_2)$ is the usual greatest common divisor in the monoid of positive integers, \emph{not} in the monoid $L_f$. These two concepts are different. For instance in the monoid $L=\{l\in\numgeq{Z}{1}\colon l\equiv1\bmod{24}\}$, $(5\cdot29,5\cdot53)=5$ in the usual sense, but the unique common divisor of $5\cdot29$ and $5\cdot53$ in $L$ is $1$.} $(l_1,l_2)=1$.
\end{thm}
This theorem is an immediate consequence of the following general fact which in the special case $N=1$ is due to Wohlfahrt \cite{Woh57}.
\begin{prop}
\label{prop:compositionTl}
Let $N\in\numgeq{Z}{1}$ and $\chi_1$, $\chi_2$, $\chi_3$ be complex linear characters on the double cover $\widetilde{\Gamma_0(N)}$. Let $l_1,l_2\in\numgeq{Z}{1}$ such that $\chi_1$ and $\chi_2$ are $\widetilde{\tbtmat{1}{0}{0}{l_1}}$-compatible, $\chi_2$ and $\chi_3$ are $\widetilde{\tbtmat{1}{0}{0}{l_2}}$-compatible. Then $\chi_1$ and $\chi_3$ are $\widetilde{\tbtmat{1}{0}{0}{l_1l_2}}$-compatible. Moreover, if $(l_1,l_2)=1$ then
\begin{equation}
\label{eq:Tl1l2Multiplicative}
T_{l_2;\chi_2,\chi_3}\circ T_{l_1;\chi_1,\chi_2}f=T_{l_1l_2;\chi_1,\chi_3}f
\end{equation}
for any meromorphic modular form $f$ on $\Gamma_0(N)$ of weight $k\in\frac{1}{2}\numZ$ and with character $\chi_1$.
\end{prop}
For the concept of compatibility of multiplier systems, see \cite[Lemma 3.1]{ZZ23} for details. We say $\chi$ and $\chi'$ are $\widetilde{\tbtmat{1}{0}{0}{l}}$-compatible if for any $\gamma\in\widetilde{\tbtmat{1}{0}{0}{l}}^{-1}\widetilde{\Gamma_0(N)}\widetilde{\tbtmat{1}{0}{0}{l}}\cap\widetilde{\Gamma_0(N)}$ we have $\chi\left(\widetilde{\tbtmat{1}{0}{0}{l}}\gamma\widetilde{\tbtmat{1}{0}{0}{l}}^{-1}\right)=\chi'(\gamma)$.
\begin{proof}[Proof of Proposition \ref{prop:compositionTl} and Theorem \ref{thm:multiplicativeEtaQuotients}]
First let us prove $\chi_1$ and $\chi_3$ are $\widetilde{\tbtmat{1}{0}{0}{l_1l_2}}$-compatible. Let $g\in\widetilde{\tbtmat{1}{0}{0}{l_1l_2}}^{-1}\widetilde{\Gamma_0(N)}\widetilde{\tbtmat{1}{0}{0}{l_1l_2}}\cap\widetilde{\Gamma_0(N)}$ be arbitrary. We have
\begin{align*}
\chi_1\left(\widetilde{\tbtmat{1}{0}{0}{l_1l_2}}g\widetilde{\tbtmat{1}{0}{0}{l_1l_2}}^{-1}\right)&=\chi_1\left(\widetilde{\tbtmat{1}{0}{0}{l_1}}\widetilde{\tbtmat{1}{0}{0}{l_2}}g\widetilde{\tbtmat{1}{0}{0}{l_2}}^{-1}\widetilde{\tbtmat{1}{0}{0}{l_1}}^{-1}\right)\\
&=\chi_2\left(\widetilde{\tbtmat{1}{0}{0}{l_2}}g\widetilde{\tbtmat{1}{0}{0}{l_2}}^{-1}\right)\\
&=\chi_3(g).
\end{align*}
In the second equality we have used the assumption $\chi_1$ and $\chi_2$ are $\widetilde{\tbtmat{1}{0}{0}{l_1}}$-compatible and in the third equality the assumption $\chi_2$ and $\chi_3$ are $\widetilde{\tbtmat{1}{0}{0}{l_2}}$-compatible. This proves the $\widetilde{\tbtmat{1}{0}{0}{l_1l_2}}$-compatibility of $\chi_1$ and $\chi_3$.

Assume that $(l_1,l_2)=1$. To prove \eqref{eq:Tl1l2Multiplicative} we need an equivalent definition of $T_lf$ via double coset actions. For the proof of this equivalence see \cite[Eq. (28) and (30)]{ZZ23}. Write
\begin{equation*}
\widetilde{\Gamma_0(N)}\widetilde{\tbtmat{1}{0}{0}{l_i}}\widetilde{\Gamma_0(N)}=\bigsqcup_{\gamma\in A_i}\widetilde{\Gamma_0(N)}\widetilde{\gamma}
\end{equation*}
for $i=1,2$. Then
$$
T_{l_1;\chi_1,\chi_2}f=\sum_{\gamma\in A_1}(\chi_1\vert_{l_1}\chi_2)^{-1}(\widetilde{\gamma})\cdot f\vert_k\widetilde{\gamma},
$$
where $\chi_1\vert_{l_1}\chi_2(h)=\chi_1(h_1)\chi_2(h_2)$ if we write $h=h_1\widetilde{\tbtmat{1}{0}{0}{l_1}}h_2$ with $h_1,h_2\in\widetilde{\Gamma_0(N)}$. The compatibility ensures that $\chi_1\vert_{l_1}\chi_2(h)$ is well-defined. Therefore
\begin{equation}
\label{eq:TcompExpand}
T_{l_2;\chi_2,\chi_3}\circ T_{l_1;\chi_1,\chi_2}f=\sum_{\gamma_1\in A_1,\gamma_2\in A_2}(\chi_2\vert_{l_2}\chi_3)^{-1}(\widetilde{\gamma_2})(\chi_1\vert_{l_1}\chi_2)^{-1}(\widetilde{\gamma_1})\cdot f\vert_k\widetilde{\gamma_1}\widetilde{\gamma_2}.
\end{equation}
By \cite[Eq. (25)]{ZZ23} with $m=1$ we can choose the following $A_i$:
\begin{equation}
\label{eq:Airepr}
A_i=\left\{\tbtmat{a}{b}{0}{d}\colon a,d>0,\,0\leq b<d,\,ad=l_i,\,(N,a)=1,\,(a,b,d)=1\right\}.
\end{equation}
It follows from the assumption $(l_1,l_2)=1$ and elementary number theory that $\widetilde{\gamma_1}\widetilde{\gamma_2}\in\widetilde{\Gamma_0(N)}\widetilde{\tbtmat{1}{0}{0}{l_1l_2}}\widetilde{\Gamma_0(N)}$ if $\gamma_1\in A_1$ and $\gamma_2\in A_2$. Therefore the map
\begin{align*}
\sigma\colon A_1\times A_2&\rightarrow\widetilde{\Gamma_0(N)}\backslash\widetilde{\Gamma_0(N)}\widetilde{\tbtmat{1}{0}{0}{l_1l_2}}\widetilde{\Gamma_0(N)}\\
(\gamma_1,\gamma_2) &\mapsto \widetilde{\Gamma_0(N)}\cdot\widetilde{\gamma_1}\widetilde{\gamma_2}
\end{align*}
is well-defined. By the theory of classical Hecke algebras, or more precisely, by Theorem 2.8.1(2) and Eq. (2.7.2) of \cite{Miy06}, $\sigma$ is a surjection and each inverse image $\sigma^{-1}(x)$ has the same cardinality. Since $\sigma^{-1}\left(\widetilde{\Gamma_0(N)}\widetilde{\tbtmat{1}{0}{0}{l_1l_2}}\right)=\left\{(\tbtmat{1}{0}{0}{l_1},\tbtmat{1}{0}{0}{l_2})\right\}$ is a singleton, $\sigma$ is a bijection which means the range of summation in the right-hand side of \eqref{eq:TcompExpand} is essentially a set of representatives of the coset space $\widetilde{\Gamma_0(N)}\backslash\widetilde{\Gamma_0(N)}\widetilde{\tbtmat{1}{0}{0}{l_1l_2}}\widetilde{\Gamma_0(N)}$. Now we can define a map, which is key to the proof,
\begin{align*}
\xi\colon\widetilde{\Gamma_0(N)}\widetilde{\tbtmat{1}{0}{0}{l_1l_2}}\widetilde{\Gamma_0(N)}&\rightarrow\numC\\
x&\mapsto\chi_1\vert_{l_1}\chi_2(x\widetilde\gamma_2^{-1})\cdot\chi_2\vert_{l_2}\chi_3(\widetilde\gamma_2)
\end{align*}
where $(\gamma_1,\gamma_2)$ is the uniquely determined pair in $A_1\times A_2$ such that $\widetilde{\Gamma_0(N)}x=\widetilde{\Gamma_0(N)}\widetilde{\gamma_1}\widetilde{\gamma_2}$. It is not hard to see that
\begin{equation}
\label{eq:xiProperty}
\xi(gxh)=\chi_1(g)\xi(x)\chi_3(h),\quad g,h\in\widetilde{\Gamma_0(N)},\,x\in\widetilde{\Gamma_0(N)}\widetilde{\tbtmat{1}{0}{0}{l_1l_2}}\widetilde{\Gamma_0(N)},\quad\xi\widetilde{\tbtmat{1}{0}{0}{l_1l_2}}=1
\end{equation}
and that
\begin{equation}
\label{eq:xig1g2}
\xi(\widetilde{\gamma_1}\widetilde{\gamma_2})=\chi_1\vert_{l_1}\chi_2(\widetilde{\gamma_1})\chi_2\vert_{l_2}\chi_3(\widetilde{\gamma_2}).
\end{equation}
Since the map $x\mapsto\chi_1\vert_{l_1l_2}\chi_3(x)$ also satisfies \eqref{eq:xiProperty}, we have $\xi=\chi_1\vert_{l_1l_2}\chi_3$. Inserting this and \eqref{eq:xig1g2} into \eqref{eq:TcompExpand} gives
\begin{equation*}
T_{l_2;\chi_2,\chi_3}\circ T_{l_1;\chi_1,\chi_2}f=\sum_{\gamma_1\in A_1,\gamma_2\in A_2}\xi^{-1}(\widetilde{\gamma_1}\widetilde{\gamma_2})\cdot f\vert_k\widetilde{\gamma_1}\widetilde{\gamma_2}=\sum_{\gamma_1\in A_1,\gamma_2\in A_2}(\chi_1\vert_{l_1l_2}\chi_3)^{-1}(\widetilde{\gamma_1}\widetilde{\gamma_2})\cdot f\vert_k\widetilde{\gamma_1}\widetilde{\gamma_2}.
\end{equation*}
Since $\sigma$ is a bijection, we obtain \eqref{eq:Tl1l2Multiplicative} from the above equalities.

Finally we prove Theorem \ref{thm:multiplicativeEtaQuotients}. Let $f$ be as in this theorem. Then $T_{l_1}f=c_{l_1}\cdot f$, $T_{l_2}f=c_{l_2}\cdot f$ and $T_{l_1l_2}f=c_{l_1l_2}\cdot f$ by Corollary \ref{coro:Lfmf}. Since $(l_1,l_2)=1$ we have
\begin{equation*}
c_{l_1l_2}\cdot f=T_{l_1l_2}f=T_{l_2}(T_{l_1}f)=T_{l_2}(c_{l_1}\cdot f)=c_{l_1}c_{l_2}\cdot f
\end{equation*}
where we have used Proposition \ref{prop:compositionTl} in the second equality. Since $f$ is not identically zero we find $c_{l_1l_2}=c_{l_1}\cdot c_{l_2}$ as required.
\end{proof}

Before looking at examples, let us simplify some notations. Set $x_N'=\frac{x_N}{24}$ and
\begin{equation*}
S_{l,\mathbf{r}}(a)=\sum_{\twoscript{0\leq b<d}{(a,b,d)=1}}\legendre{-Nb}{(a,d)}^{2k}\etp{-\frac{(l-1)x_N'b}{a}}\psi_{l,\mathbf{r}}(a,b).
\end{equation*}
Then according to Theorem \ref{thm:mainApp2} we have
\begin{equation*}
c_l=l^{-\frac{k}{2}}\sum_{a\mid l,\,(a,N)=1}\legendre{a}{\varPi}a^k\cdot c_f\left(\frac{lx_N'}{a^2}\right)S_{l,\mathbf{r}}(a),\quad l\in L_f.
\end{equation*}
We recall that here and below $f$ must be a weakly admissible eta-quotient.
\begin{examp}
If $x_N'=0$, that is, $x_N=\sum_{n\mid N}nr_n=0$, then $c_f(0)=1$ and hence Theorem \ref{thm:multiplicativeEtaQuotients} gives
\begin{equation*}
\sum_{a\mid l_1,\,(a,N)=1}\legendre{a}{\varPi}a^k\cdot S_{l_1,\mathbf{r}}(a)\cdot\sum_{a\mid l_2,\,(a,N)=1}\legendre{a}{\varPi}a^k\cdot S_{l_2,\mathbf{r}}(a)=\sum_{a\mid l_1l_2,\,(a,N)=1}\legendre{a}{\varPi}a^k\cdot S_{l_1l_2,\mathbf{r}}(a)
\end{equation*}
for $l_1,l_2\in L_f$ with $(l_1,l_2)=1$.
\end{examp}

\begin{examp}
If $x_N'\in\numgeq{Z}{0}$ and $(l,x_N')=1$, then $c_f\left(\frac{lx_N'}{a^2}\right)=0$ unless $\frac{lx_N'}{a^2}\in\numZ$ by Theorem \ref{thm:cTlfbycf}(a). Therefore we have
\begin{multline*}
\sum_{a^2\mid l_1,\,(a,N)=1}\legendre{a}{\varPi}a^k\cdot c_f\left(\frac{l_1x_N'}{a^2}\right)S_{l_1,\mathbf{r}}(a)\cdot\sum_{a^2\mid l_2,\,(a,N)=1}\legendre{a}{\varPi}a^k\cdot c_f\left(\frac{l_2x_N'}{a^2}\right)S_{l_2,\mathbf{r}}(a)\\
=\sum_{a^2\mid l_1l_2,\,(a,N)=1}\legendre{a}{\varPi}a^k\cdot c_f\left(\frac{l_1l_2x_N'}{a^2}\right)S_{l_1l_2,\mathbf{r}}(a)
\end{multline*}
if $l_1,l_2\in L_f$ with $(l_1,l_2)=1$, $(l_1,x_N')=1$ and $(l_2,x_N')=1$. If in addition \eqref{eq:pnp2l} holds for $l=l_1,l_2$, then the above identity becomes
\begin{equation}
\label{eq:cfl1l2xNDecomp}
c_f(l_1x_N')S_{l_1,\mathbf{r}}(1)\cdot c_f(l_2x_N')S_{l_2,\mathbf{r}}(1)=c_f(l_1l_2x_N')S_{l_1l_2,\mathbf{r}}(1).
\end{equation}
From this we obtain a property of $c_f(n)$ very closed to the multiplicativity property. (We exclude the half-integral case below for simplicity.)
\end{examp}

\begin{prop}
\label{prop:kZmultiplicativecfl}
Let $f(\tau)=\prod_{n\mid N}\eta(n\tau)^{r_n}=\sum_{n}c_f(n)q^n$ be a weakly admissible eta-quotient (e.g., a function listed in Table \ref{table:admissibleTypeI} or Example \ref{examp:level12}) where $N$ is the level. We keep the notations in \eqref{eq:notationsrn} with $k'=k$. Assume that $x_N'=\frac{x_N}{24}\in\numgeq{Z}{0}$ and $k\in\numZ$. Let $l_1,l_2\in L_f$ satisfy $(l_1,l_2)=1$, $(l_1,x_N')=1$ and $(l_2,x_N')=1$. We require \eqref{eq:pnp2l} holds for $l=l_1$ and $l_2$. Moreover, if $N\equiv0\bmod{2}$ and $N(k+\delta)\equiv2\bmod{4}$, then we require $l_1\equiv l_2\equiv1\bmod{2}$. (See Theorem \ref{thm:cTlfbycf}(b) for the definition of $\delta$.) Then
\begin{equation*}
c_f(l_1x_N')\cdot c_f(l_2x_N')=\varepsilon\cdot c_f(l_1l_2x_N'),
\end{equation*}
where $\varepsilon=u_{l_1,\mathbf{r}}\cdot u_{l_2,\mathbf{r}}\cdot u_{l_1l_2,\mathbf{r}}$ with
\begin{equation*}u_{l,\mathbf{r}}=
\begin{dcases}
1, &\text{ if   }2\mid l,\quad2\nmid N \text{ or }4\mid N(k+\delta)\\
1, &\text{ if   }2\nmid l,\quad4\mid N(k+\delta)(l-1)\\
-1, &\text{ if   }2\nmid l,\quad4\nmid N(k+\delta)(l-1).
\end{dcases}
\end{equation*}
\end{prop}
\begin{proof}
According to \eqref{eq:cfl1l2xNDecomp} we need to prove
\begin{equation}
\label{eq:needtoProveSl1l2}
S_{l_1l_2,\mathbf{r}}(1)=u_{l_1,\mathbf{r}}\cdot u_{l_2,\mathbf{r}}\cdot u_{l_1l_2,\mathbf{r}}\cdot S_{l_1,\mathbf{r}}(1)\cdot S_{l_2,\mathbf{r}}(1).
\end{equation}
By the definition of $\psi_{l,\mathbf{r}}$ in Theorem \ref{thm:cTlfbycf}, we have, since $k\in\numZ$,
\begin{equation*}S_{l,\mathbf{r}}(1)=
\begin{dcases}
l, &\text{ if   }2\mid l,\quad2\nmid N \text{ or }4\mid N(k+\delta)\\
0, &\text{ if   }2\mid l,\quad2\mid N\text{ and }  N(k+\delta)\equiv2\bmod{4}\\
l, &\text{ if   }2\nmid l,\quad4\mid N(k+\delta)(l-1)\\
-l, &\text{ if   }2\nmid l,\quad4\nmid N(k+\delta)(l-1).
\end{dcases}
\end{equation*}
From this \eqref{eq:needtoProveSl1l2} follows. Substituting \eqref{eq:needtoProveSl1l2} in \eqref{eq:cfl1l2xNDecomp} we obtain
$$
c_f(l_1x_N')\cdot c_f(l_2x_N')S_{l_1,\mathbf{r}}(1)S_{l_2,\mathbf{r}}(1)=\varepsilon\cdot c_f(l_1l_2x_N')S_{l_1,\mathbf{r}}(1)S_{l_2,\mathbf{r}}(1).
$$
By the assumption we have $S_{l_1,\mathbf{r}}(1)$ and $S_{l_2,\mathbf{r}}(1)$ are nonzero and hence the required relation follows.
\end{proof}

If $x_N'=0$, then the above proposition holds trivially. We believe there are many weakly admissible eta-quotients with $x_N'\in\numgeq{Z}{1}$ to which we can apply the above proposition. The most well-known example is the discriminant function $\Delta(\tau)=\eta(\tau)^{24}$, which in our terminology is an admissible eta-quotient of type II. The only example in Table \ref{table:admissibleTypeI} is the one presented below.
\begin{examp}
\label{examp:level27Multi}
Set $f=\eta(3\tau)^2\eta(9\tau)^{-1}\eta(27\tau)=\sum_{n\in\numgeq{Z}{1}}c_f(n)q^n$. This is an admissible eta-quotient of type I. We have
\begin{equation*}
N=27,\quad k=1,\quad x_N=24,\quad\varPi=3^5,\quad\delta=1.
\end{equation*}
Since (see \eqref{eq:rntoxc})
\begin{equation*}
\begin{pmatrix}
x_1 \\ x_3 \\ x_9 \\ x_{27}
\end{pmatrix}=
\begin{pmatrix}
27 & 9 & 3 & 1 \\
3  & 9 & 3 & 1 \\
1 & 3 & 9 & 3 \\
1 & 3 & 9 & 27
\end{pmatrix}
\begin{pmatrix}
r_1 \\ r_3 \\ r_9 \\ r_{27}
\end{pmatrix}=
\begin{pmatrix}
16 \\ 16 \\ 0 \\ 24
\end{pmatrix},
\end{equation*}
we have $m_f=3$ (see Corollary \ref{coro:Lfmf}) and hence $L_f=\{3n+1\colon n\in\numgeq{Z}{0}\}$. It follows from Proposition \ref{prop:kZmultiplicativecfl} that $c_f(l_1)c_f(l_2)=c_f(l_1l_2)$ where $l_1,l_2\equiv1\bmod{3}$ are square-free positive integers with $(l_1,l_2)=1$.
\end{examp}

\begin{examp}
Let us give an example of admissible eta-quotient of type II. Set
$$
f=\eta(\tau)\eta(2\tau)\eta(3\tau)\eta(6\tau)^3=\sum_{n\in\numgeq{Z}{1}}c_f(n)q^n.
$$
We have
\begin{equation*}
N=6,\quad k=3,\quad x_N=24,\quad\varPi=2^3\cdot3^2,\quad\delta=0.
\end{equation*}
Since
\begin{equation*}
\begin{pmatrix}
x_1 \\ x_2 \\ x_3 \\ x_{6}
\end{pmatrix}=
\begin{pmatrix}
6 & 3 & 2 & 1 \\
3  & 6 & 1 & 2 \\
2 & 1 & 6 & 3 \\
1 & 2 & 3 & 6
\end{pmatrix}
\begin{pmatrix}
r_1 \\ r_2 \\ r_3 \\ r_{6}
\end{pmatrix}=
\begin{pmatrix}
14 \\ 16 \\ 18 \\ 24
\end{pmatrix},
\end{equation*}
we have $m_f=12$ and hence $L_f=\{12n+1\colon n\in\numgeq{Z}{0}\}$. By Theorem \ref{thm:main}, $\dim_\numC M_3(\Gamma_0(6),\chi)=2$ where $\chi\colon\widetilde{\Gamma_0(6)}\rightarrow\numC^\times$ is the character of $f$. (Of course this descends to a character of the matrix group $\Gamma_0(6)$ since $k\in\numZ$.) Recall that the number of Eisenstein series in $M_3(\Gamma_0(6),\chi)$ is denoted by $n_0$ in \eqref{eq:n0EisSeries}. Since $\chi(\widetilde{\gamma_{a/c}}\widetilde{T}^{w_{a/c}}\widetilde{\gamma_{a/c}}^{-1})=1$ if and only if $x_c\equiv0\bmod{24}$ where $c\mid 6$ and $(a,c)=1$, among the four cusps in $\Gamma_0(6)\backslash\projQ$ there is only one Eisenstein series which is defined at the cusp $\rmi\infty$. Therefore $n_0=1$. Since $f$ is a cusp form, we find that it is admissible of type II. It follows from Proposition \ref{prop:kZmultiplicativecfl} (applied to this $f$) that $c_f(l_1)c_f(l_2)=c_f(l_1l_2)$ whenever $l_1,l_2\equiv1\bmod{12}$ are square-free positive integers with $(l_1,l_2)=1$. It is worthwhile to notice that for other pair $(l_1,l_2)$, the multiplicativity does not necessarily hold. For instance, $c_f(2)=-1$, $c_f(3)=-2$ but $c_f(6)=4$.
\end{examp}

\subsubsection{An example of half-integral weight} The examples presented above are all of integral weights. To conclude this section we supplement an example of half-integral weight where we show how the Gauss sum $\sum_b$ in \eqref{eq:cTlfbycfxN0} can be simplified. Let
\begin{equation}
\label{eq:N4halfintegralExamp}
f(\tau)=\eta(\tau)^{-7}\eta(2\tau)^{17}\eta(4\tau)^{-3}=\sum_{n\in\frac{5}{8}+\numgeq{Z}{0}}c_f(n)q^n.
\end{equation}
The statistics of $f$ are
\begin{equation*}
N=4,\quad k=\frac{7}{2},\quad x_1=3,\quad,x_2=24,\quad x_4=15,\quad\varPi=2^{31},\quad\delta=0,\quad m_f=8.
\end{equation*}
Thus $L_f=\{l^2\colon l\equiv1\bmod{2}\}$. By Theorem \ref{thm:main} we have $\dim_\numC M_{7/2}(\Gamma_0(4),\chi)=2$ where $\chi\colon\widetilde{\Gamma_0(4)}\rightarrow\numC^\times$ is the character of $f$ and by Definition \ref{deff:Eis} the Eisenstein series can only be given at the cusp $1/2$. Thus, $f$ is admissible of type II and Theorems \ref{thm:mainApp2}, \ref{thm:multiplicativeEtaQuotients} can be applied.

Now we calculate $c_{l^2}$ for $l^2\in L_f$ (see Theorem \ref{thm:mainApp2}). We have, by definition,
\begin{equation}
\label{eq:cl2example}
c_{l^2}=l^{-\frac{7}{2}}\sum_{a\mid l^2}\legendre{a}{2}a^{\frac{7}{2}}c_f\left(\frac{5l^2}{8a^2}\right)\etp{-\frac{7l^2}{8a}+\frac{7}{8}}\sum_{\twoscript{0\leq b<l^2/a}{(a,b,l^2/a)=1}}\legendre{-b}{(a,l^2/a)}\etp{-\frac{5(l^2-1)b}{8a}}.
\end{equation}
Note that $c_f\left(\frac{5l^2}{8a^2}\right)=0$ unless $\frac{5l^2}{8a^2}\in\frac{5}{8}+\numZ$, that is, unless $a\mid l$ since $a,\,l$ are odd. It follows that for such $a$ we have $(a,l^2/a)=a$,
$$
t:=\frac{5(l^2-1)\cdot(a,l^2/a)}{8a}=\frac{5(l^2-1)}{8}
$$
is an integer (depending on $a$ and $l$) and
\begin{align}
\sum_{\twoscript{0\leq b<l^2/a}{(a,b,l^2/a)=1}}\legendre{-b}{(a,l^2/a)}\etp{-\frac{5(l^2-1)b}{8a}}&=\sum_{0\leq b<l^2/a}\legendre{-b}{a}\etp{\frac{-tb}{a}}\notag\\
&=\frac{l^2}{a^2}\sum_{0\leq b<a}\legendre{b}{a}\etp{\frac{tb}{a}}\label{eq:sumbSimplify}
\end{align}
is, up to a simple factor, a Gauss sum associated with the character $\legendre{\cdot}{(a,l^2/a)}$. To express the values of these Gauss sums, we let, for any odd positive integer $m$,
\begin{align*}
\radE(m) &= \prod_{\twoscript{p\mid m}{v_p(m)\equiv0\bmod{2}}}p  \qquad &\radO(m) &= \prod_{\twoscript{p\mid m}{v_p(m)\equiv1\bmod{2}}}p \\
\rad(m) &= \radE(m)\radO(m) \qquad &\radp(m) &= \radE(m)^2\radO(m) \\
\irad(m) &= m/\rad(m) \qquad &\iradp(m)&=m/\radp(m)
\end{align*}
where $p$ denotes a prime.
By \cite[Lemma 5.6]{ZZ23}, if $\irad(a)\nmid t$ then $\sum_{0\leq b<a}\legendre{b}{a}\etp{\frac{tb}{a}}=0$; if $\irad(a)\mid t$ then
\begin{equation}
\label{eq:aGaussSum}
\sum_{0\leq b<a}\legendre{b}{a}\etp{\frac{tb}{a}}
=\varepsilon_a\frac{a}{\sqrt{\radp(a)}}\legendre{t/\iradp(a)}{\radO(a)}\prod_{p \mid \radE(a)}\left(p-1-p\legendre{t/\irad(a)}{p}^2\right),
\end{equation}
where $\legendre{t/\iradp(a)}{\radO(a)}$, $\legendre{t/\irad(a)}{p}$ are Jacobi symbols and $\varepsilon_a=1$ ($\varepsilon_a=\rmi$ respectively) if $a\equiv1\bmod{4}$ ($a\equiv3\bmod{4}$ respectively). Inserting \eqref{eq:aGaussSum} and \eqref{eq:sumbSimplify} into \eqref{eq:cl2example} and noting that $\etp{-\frac{7l^2}{8a}+\frac{7}{8}}=\etp{\frac{a-1}{8}}$, $\legendre{a}{2}\varepsilon_a\etp{\frac{a-1}{8}}=1$ we obtain the following formula.
\begin{prop}
For $f(\tau)=\eta(\tau)^{-7}\eta(2\tau)^{17}\eta(4\tau)^{-3}=\sum_{n\in\frac{5}{8}+\numgeq{Z}{0}}c_f(n)q^n$ and $1\leq l\equiv1\bmod{2}$ we have
\begin{equation*}
c_{l^2}=l^{-\frac{3}{2}}\sum_{\twoscript{a\mid l}{\irad(a)\mid t}}\frac{a^{\frac{5}{2}}}{\sqrt{\radp(a)}}\prod_{p \mid \radE(a)}\left(p-1-p\legendre{t/\irad(a)}{p}^2\right)\legendre{t/\iradp(a)}{\radO(a)}c_f\left(\frac{5l^2}{8a^2}\right)
\end{equation*}
where $t=\frac{5(l^2-1)}{8}$. In particular, if $l$ is square-free, then
\begin{equation*}
c_{l^2}=l^{-\frac{3}{2}}\sum_{a\mid l}a^2\cdot\legendre{t}{a}c_f\left(\frac{5l^2}{8a^2}\right).
\end{equation*}
\end{prop}
\begin{rema}
Since $\frac{a^{5/2}}{\sqrt{\radp(a)}}\in\numZ$ and $c_f(n)\in\numZ$, we have $l^{3/2}\cdot c_{l^2}\in\numZ$.
\end{rema}

According to Theorem \ref{thm:multiplicativeEtaQuotients} (applied to \eqref{eq:N4halfintegralExamp}) we have $c_{l_1^2}c_{l_2^2}=c_{l_1^2l_2^2}$ whenever $l_1,l_2$ are coprime odd integers. In the special case $l_1$ and $l_2$ are square-free, this identity reads
\begin{multline}
\label{eq:N4ExampMultisf}
\sum_{a\mid l_1}a^2\cdot\legendre{5(l_1^2-1)/8}{a}c_f\left(\frac{5l_1^2}{8a^2}\right)\cdot\sum_{a\mid l_2}a^2\cdot\legendre{5(l_2^2-1)/8}{a}c_f\left(\frac{5l_2^2}{8a^2}\right)\\
=\sum_{a\mid l_1l_2}a^2\cdot\legendre{5(l_1^2l_2^2-1)/8}{a}c_f\left(\frac{5l_1^2l_2^2}{8a^2}\right).
\end{multline}
From this we find the genuine multiplicativity property of \eqref{eq:N4halfintegralExamp} as stated below. The reader may compare this to the integral-weight counterpart---Proposition \ref{prop:kZmultiplicativecfl}.
\begin{prop}
\label{prop:semiMultiExampLevel4}
Let $f(\tau)=\eta(\tau)^{-7}\eta(2\tau)^{17}\eta(4\tau)^{-3}=\sum_{n\in\frac{5}{8}+\numgeq{Z}{0}}c_f(n)q^n$. Let $l_1,l_2$ be coprime odd square-free integers. Then
\begin{equation*}
c_f\left(\frac{5l_1^2}{8}\right)c_f\left(\frac{5l_2^2}{8}\right)=c_f\left(\frac{5l_1^2l_2^2}{8}\right).
\end{equation*}
\end{prop}
\begin{proof}
First suppose $l_1$ and $l_2$ are primes. By \eqref{eq:N4ExampMultisf} we have
\begin{multline}
\label{eq:N4ExampMultisfPrimes}
\left(c_f\left(\frac{5l_1^2}{8}\right)+l_1^2\cdot\legendre{5(l_1^2-1)/8}{l_1}\right)\cdot\left(c_f\left(\frac{5l_2^2}{8}\right)+l_2^2\cdot\legendre{5(l_2^2-1)/8}{l_2}\right)\\
=c_f\left(\frac{5l_1^2l_2^2}{8}\right)+l_2^2\cdot\legendre{5(l_1^2l_2^2-1)/8}{l_2}c_f\left(\frac{5l_1^2}{8}\right)+l_1^2\cdot\legendre{5(l_1^2l_2^2-1)/8}{l_1}c_f\left(\frac{5l_2^2}{8}\right)+l_1^2l_2^2\cdot\legendre{5(l_1^2l_2^2-1)/8}{l_1l_2}.
\end{multline}
It is not hard to prove, by the periodicity and multiplicativity of Jacobi symbols, that
\begin{align*}
\legendre{5(l_1^2l_2^2-1)/8}{l_1}&=\legendre{5(l_1^2-1)/8}{l_1},\\
\legendre{5(l_1^2l_2^2-1)/8}{l_2}&=\legendre{5(l_2^2-1)/8}{l_2},\\
\legendre{5(l_1^2l_2^2-1)/8}{l_1l_2}&=\legendre{5(l_1^2-1)/8}{l_1}\cdot\legendre{5(l_2^2-1)/8}{l_2}.
\end{align*}
Inserting these into \eqref{eq:N4ExampMultisfPrimes} gives the desired identity in the case $l_1$ and $l_2$ are primes.

Second, suppose $l_1$ is a prime and $l_2$ is arbitrary. Again by \eqref{eq:N4ExampMultisf} we have
\begin{multline}
\label{eq:N4ExampMultisfOnePrime}
\left(c_f\left(\frac{5l_1^2}{8}\right)+l_1^2\cdot\legendre{5(l_1^2-1)/8}{l_1}\right)\cdot\sum_{a\mid l_2}a^2\cdot\legendre{5(l_2^2-1)/8}{a}c_f\left(\frac{5l_2^2}{8a^2}\right)\\
=\sum_{a\mid l_2}a^2\cdot\legendre{5(l_1^2l_2^2-1)/8}{a}c_f\left(\frac{5l_1^2l_2^2}{8a^2}\right)+\sum_{a\mid l_2}l_1^2a^2\cdot\legendre{5(l_1^2l_2^2-1)/8}{l_1a}c_f\left(\frac{5l_2^2}{8a^2}\right).
\end{multline}
Since for $a\mid l_2$,
$$
\legendre{5(l_1^2-1)/8}{l_1}\cdot\legendre{5(l_2^2-1)/8}{a}=\legendre{5(l_1^2l_2^2-1)/8}{l_1a},
$$
we obtain from \eqref{eq:N4ExampMultisfOnePrime} that
\begin{align}
c_f\left(\frac{5l_1^2}{8}\right)\cdot\sum_{a\mid l_2}a^2\cdot\legendre{5(l_2^2-1)/8}{a}c_f\left(\frac{5l_2^2}{8a^2}\right)&=\sum_{a\mid l_2}a^2\cdot\legendre{5(l_1^2l_2^2-1)/8}{a}c_f\left(\frac{5l_1^2l_2^2}{8a^2}\right)\notag\\
&=\sum_{a\mid l_2}a^2\cdot\legendre{5(l_2^2-1)/8}{a}c_f\left(\frac{5l_1^2l_2^2}{8a^2}\right).
\label{eq:N4ExampMultisfOnePrime2}
\end{align}
The desired identities for $l_1$ prime and $l_2$ arbitrary now follows from \eqref{eq:N4ExampMultisfOnePrime2}, the already proven case and the induction on the number of prime divisors of $l_2$.

Finally, if $l_1$ and $l_2$ are arbitrary, then, by the already proven cases, both sides of the desired identity are equal to $\prod_{p\mid l_1l_2}c_f\left(\frac{5p^2}{8}\right)$ and hence the identity holds which concludes the proof.
\end{proof}

We have discussed consequences of Theorem \ref{thm:multiplicativeEtaQuotients} applied to \eqref{eq:N4halfintegralExamp}. The above exploration is also valid for any admissible eta-quotient of half-integral weight. (Of course the integral weight case is simpler to deal with.) In addition, for \eqref{eq:N4halfintegralExamp} or any admissible eta-quotient of half-integral weight, we can as well simplify the left-hand side of \eqref{eq:mainApp2} as we have just done for $c_{l^2}$. This will give more interesting identities about $c_f(n)$. We will not include more details about this due to the length of the paper.

\section{Miscellaneous observations and open questions}
\label{sec:Miscellaneous observations and open questions}
\textbf{More general multiplier systems.} The dimension formulas stated in Theorem \ref{thm:main} concern multiplier systems induced by eta-quotients of fractional exponents. As a comparison, most of the formulas that have appeared in the literature concern multiplier systems induced by Dirichlet characters (in the case of integral weights) or Dirichlet characters times the multiplier system of a power of $\eta(\tau)^{-2}\eta(2\tau)^5\eta(4\tau)^{-2}$ (in the case of half-integral weights). It is meaningful to deduce an explicit formula for $\dim_\numC M_k(\Gamma_0(N),\chi_1\chi_2)$ where $k\in\numQ$, $\chi_1$ is induced by an eta-quotient of fractional exponents and $\chi_2$ by a Dirichlet character.

\textbf{On generalized modular forms.} The theory of generalized modular forms, initiated in \cite{KM03}, grows naturally out of questions inherent in rational conformal field theory. There are many excellent works on this topic; cf. e.g. \cite{KM09,Koh10,KR10,Raj13,BCD18}. In these works, the space $M_k(G,\chi)$ was studied (among others), where $G$ is a cofinite Fuchsian group of the first kind, $\chi$ is a possibly non-unitary multiplier system, and $k$ is a complex number. As we mentioned in Section \ref{sec:Introduction}, formulas for $\dim_\numC M_k(G,\chi)$ were given in \cite[p. 194]{Pet38} when $k$ is complex with a restriction on the degree of the corresponding divisor, and in \cite[Theorem 2.5.4]{Fis87} when $k$ is real, greater than $2$ and $\chi$ is unitary. It is therefore a meaningful problem to derive computationally explicit formulas for concrete pairs $(k,G,\chi)$.

\textbf{Do generalized double coset operators exist for rational weights?} Theorem \ref{thm:doubleCosetOperator} let us know for what numbers $l$ does the expression $T_lf$ make sense. The multiplier systems involved are required to be induced by eta-quotients of \emph{integral} exponents. The question is: if the multiplier systems are induced by eta-quotients of fractional exponents, is there any nontrivial $T_l$ (i.e. $l>1$)? All identities presented in Section \ref{subsec:identitiesFourier} are based on Theorem \ref{thm:doubleCosetOperator}. Thus if we can find any $T_l$ with $l>1$ in the case of fractional exponents, then, taking into account of Theorem \ref{thm:main}, we would obtain identities involving $c_f(n)$ for infinitely many eta-quotients $f$ of \emph{fractional} exponents. However, no such operator has been found up to now. In fact, we have spent a considerable amount of time searching for such operators to the best of our ability, but all attempts ended in failure. We have tried double coset operators $\Dcover{\Gamma_0(N)}{D}\widetilde{\tbtmat{a}{b}{c}{d}}\Dcover{\Gamma_0(N)}{D}$ and $\Gamma_0(N)^\bullet\widetilde{\tbtmat{a}{b}{c}{d}}\Gamma_0(N)^\bullet$ (see Theorem \ref{thm:slZbullet} for the notation $\Gamma_0(N)^\bullet$) for $\tbtmat{a}{b}{c}{d}\in\glpQ$. However, for $k\not\in\frac{1}{2}\numZ$, the action of these operators on $M_k(\Dcover{G}{D},\chi)$ is always trivial (the zero operator). Nontrivial operators may indeed exist but were overlooked during our exploration. Alternatively, one may search for them among $\tbtmat{a}{b}{c}{d}\in\mathrm{GL}_2^{+}(K)$ where $K$ is a  totally real number field. (This is difficult and novel since it forces us to consider operators given by infinite sums.) All in all, this problem remains open.

\textbf{$L$-functions and Euler products.} According to Theorem \ref{thm:multiplicativeEtaQuotients}, it is natural to associate an $L$-function $L(f,s)=\sum_{l\in L_f}\frac{c_l}{l^s}$ with a weakly admissible eta-quotient $f$. This association is different from the usual one $\sum_{n}\frac{c_f(n)}{n^s}$. The identities presented in Section \ref{subsec:identitiesFourier} can be rephrased as properties of this $L(f,s)$. For instance, it has an Euler product. It is interesting to investigate $L(f,s)$, e.g., the functional equations, the corresponding Weil's theorem, etc.

\textbf{Listing all weakly admissible eta-quotients.} Open question: to find out all weakly admissible eta-quotients (if there exist only finitely many). Another open question: for each weakly admissible eta-quotient, to deduce identities like the one in Proposition \ref{prop:semiMultiExampLevel4} from \eqref{eq:mainApp2}.

\begin{appendix}
\section{The $D$-covers}
\label{apx:covers}
Let $D$ be a positive integer. We have defined $\Dcover{\glpR}{D}$ and $\Dcover{\slR}{D}$ in Section \ref{sec:Modular forms of rational weight}. They are called $D$-covers because the fibers of the natural projections have cardinality $D$. Historically, Kubota \cite{Kub66,Kub67} was the first to systematically investigate topological coverings of $\mathrm{SL}_2(F)$, where $F$ is a completion, by a finite or infinite place, of a number field (finite extension of $\numQ$). Indeed, suppose that the polynomial $x^D-1$ splits in $F$, then one can define the local Hilbert symbol $(\cdot,\cdot)_{D,F}\colon F^\times\times F^\times\rightarrow\{z\in\numC\colon z^D=1\}$; cf. e.g. \cite[p. 260--262]{Wei74}. Then Kubota \cite[Theorem]{Kub67} constructed a $D$-fold, nontrivial, topological covering of $\mathrm{SL}_2(F)$ using a so-called factor set expressed by the local Hilbert symbol. For $F=\numR$, $x^D-1$ splits in $\numR$ if and only if $D=1$ or $2$. Thus, Kubota's construction yields only a double (metaplectic) cover of  $\slR$. For $D\geq3$, $\Dcover{\slR}{D}$ is not covered by Kubota's construction.

There is another approach to extending $\slR$, namely using a central extension of $\slR$ by $S^1$, the unit circle. This was first employed by Aoki \cite{Aok17} to study modular forms and Jacobi forms of real weight. The group $\Dcover{\slR}{D}$ is then in fact a subgroup of an appropriate central extension. In our opinion, $\Dcover{\slR}{D}$ seems more convenient for rational weights.

In this appendix, we first provide an explicit formula for the multiplication in $\Dcover{\glpR}{D}$, and then describe its Lie group structure. The latter is not directly related to the main theme of this paper, but it may be of interest to researchers in Lie group theory, and it provides us with a deeper insight. Last but not least, we describe some direct applications of $D$-covers.

Recall that we adopt the convention \eqref{eq:branchRoot}.
\begin{prop}
Let $\left(\tbtmat{a_1}{b_1}{c_1}{d_1}, \varepsilon_1\right)$ and $\left(\tbtmat{a_2}{b_2}{c_2}{d_2}, \varepsilon_2\right)$ be elements of $\Dcover{\glpR}{D}$. Set $\gamma_j=\tbtmat{a_j}{b_j}{c_j}{d_j}$ and $\tbtmat{a_3}{b_3}{c_3}{d_3}=\gamma_1\gamma_2$. Then
\begin{equation*}
\left(\tbtmat{a_1}{b_1}{c_1}{d_1}, \varepsilon_1\right)\cdot\left(\tbtmat{a_2}{b_2}{c_2}{d_2}, \varepsilon_2\right)=\left(\tbtmat{a_3}{b_3}{c_3}{d_3}, \varepsilon_1\varepsilon_2\sigma(\gamma_1,\gamma_2)\right),
\end{equation*}
where $\sigma(\gamma_1,\gamma_2)\in\{\etp{1/D},\etp{-1/D},1\}$ is determined as follows:
\begin{enumerate}
  \item[(a)] $\sigma(\gamma_1,\gamma_2)=\etp{1/D}$ if one of the following two conditions holds:
  \begin{itemize}
    \item $c_1=c_2=0$ and $d_1<0,\,d_2<0$,
    \item $c_1 \geq 0,\, c_2 \geq 0$ but $c_3 < 0$.
  \end{itemize}
  \item[(b)] $\sigma(\gamma_1,\gamma_2)=\etp{-1/D}$ if $c_1 < 0$, $c_2 < 0$ but $c_3 \geq 0$.
  \item[(c)] $\sigma(\gamma_1,\gamma_2)=1$ in other cases.
\end{enumerate}
\end{prop}
\begin{proof}
Without loss of generality, we can assume $\det(\gamma_1)=\det(\gamma_2)=1$. Set $z_1=c_1(a_2\tau+b_2)/(c_2\tau+d_2)+d_1$ and $z_2=c_2\tau+d_2$, then $z_1z_2=c_3\tau+d_3$, and
$$\sigma\left(\tbtmat{a_1}{b_1}{c_1}{d_1},\tbtmat{a_2}{b_2}{c_2}{d_2}\right)=\frac{\sqrt[D]{z_1}\sqrt[D]{z_2}}{\sqrt[D]{z_1z_2}}=\exp\left(\frac{\log z_1+\log z_2 - \log(z_1z_2)}{D}\right).$$
If $c_1 < 0$, $c_2 < 0$ but $c_3 \geq 0$, then $\Im z_1<0$, $\Im z_2<0$, and $\Im(z_1z_2)\geq0$. It follows that $\log z_1+\log z_2 - \log(z_1z_2)=\rmi(\arg(z_1)+\arg(z_2)-\arg(z_1z_2))=-2\uppi\rmi$. Hence $\sigma\left(\tbtmat{a_1}{b_1}{c_1}{d_1},\tbtmat{a_2}{b_2}{c_2}{d_2}\right)=\etp{-1/D}$ as desired. The proofs for other cases are similar.
\end{proof}

Now we introduce a topology on $\Dcover{\glpR}{D}$. We fix a $\tau_0\in\uhp$ and a disk $B$ centered at $\tau_0$ with $\overline{B}\subset\uhp$, where $\overline{B}$ is the closure in $\numC$. Let $A=\left(\tbtmat{a}{b}{c}{d}, \varepsilon\right)=\eleglptRaDp{a}{b}{c}{d}{\varepsilon}{D}$ be an element of $\Dcover{\glpR}{D}$, where $\tbtmat{a'}{b'}{c'}{d'}:=(ad-bc)^{-1/2}\cdot\tbtmat{a}{b}{c}{d}\in\slR$. Set $\gamma=\tbtmat{a}{b}{c}{d}$ and $\phi_\gamma(\tau)=(c^\prime\tau+d^\prime)^{\frac{1}{D}}$. Then we can write $A=(\gamma,\varepsilon\cdot\phi_\gamma)$. We define
\begin{equation}
\label{eq:defUA}
U(A;\delta_1,\delta_2):=\left\{(\gamma_1,\varepsilon_1\cdot\phi_{\gamma_1})\in\Dcover{\glpR}{D}\colon\lVert\gamma_1-\gamma\rVert<\delta_1,\,\lVert\varepsilon_1\cdot\phi_{\gamma_1}-\varepsilon\cdot\phi_{\gamma}\rVert<\delta_2\right\},
\end{equation}
where $\delta_1,\delta_2>0$, $\lVert\gamma_1-\gamma\rVert$ is an arbitrary matrix norm (e.g., the sup norm), and where\footnote{We can use the more natural compact-open topology, i.e., the topology of uniform convergence on compact sets. But the one we use here suffices.}
$$\lVert\varepsilon_1\cdot\phi_{\gamma_1}-\varepsilon\cdot\phi_{\gamma}\rVert:=\sup\nolimits_{\tau\in\overline{B}}\lvert\varepsilon_1\cdot\phi_{\gamma_1}(\tau)-\varepsilon\cdot\phi_{\gamma}(\tau)\rvert.$$
We say a set $\mathcal{S}\subseteq\Dcover{\glpR}{D}$ open if for every $A\in \mathcal{S}$, there exist $\delta_1$ and $\delta_2$ such that $U(A;\delta_1,\delta_2)\subseteq S$. This defines a topology on $\Dcover{\glpR}{D}$. Equivalently, we can realize $\Dcover{\glpR}{D}$ as a subgroup of the semi-direct product of $\glpR$ and a certain group of invertible holomorphic functions, and then define our topology as the subspace topology induced from the product topology. The former method is more explicit and the latter more natural.
\begin{lemm}
\label{lemm:neighborhood}
If $c\neq0$, or $c=0$ but $d>0$, then there exist $r_1,r_2>0$ such that for all $0<\delta_1<r_1$, we have
\begin{equation}
\label{eq:smallneighborhood1}
U(A;\delta_1,r_2)=\left\{(\gamma_1,\varepsilon_1\cdot\phi_{\gamma_1})\in\Dcover{\glpR}{D}\colon\lVert\gamma_1-\gamma\rVert<\delta_1,\,\varepsilon_1=\varepsilon\right\}.
\end{equation}
On the other hand, if $c=0$ and $d<0$, then there exist $r_1,r_2>0$ such that for all $0<\delta_1<r_1$, we have
\begin{multline}
\label{eq:smallneighborhood2}
U(A;\delta_1,r_2)=\left\{(\gamma_1,\varepsilon_1\cdot\phi_{\gamma_1})\in\Dcover{\glpR}{D}\colon\gamma_1=\tbtmat{a_1}{b_1}{c_1}{d_1},\,c_1\geq0,\,\lVert\gamma_1-\gamma\rVert<\delta_1,\,\varepsilon_1=\varepsilon\right\}\\
\cup\left\{(\gamma_1,\varepsilon_1\cdot\phi_{\gamma_1})\in\Dcover{\glpR}{D}\colon\gamma_1=\tbtmat{a_1}{b_1}{c_1}{d_1},\,c_1<0,\,\lVert\gamma_1-\gamma\rVert<\delta_1,\,\varepsilon_1=\varepsilon\cdot\etp{1/D}\right\},
\end{multline}
where the union is disjoint.
\end{lemm}
\begin{proof}
We only write down the proof for $c=0,d<0$, since the other case is simpler and similar. The first and second sets in the right-hand side of the desired identity are denoted by $V(A;\delta_1)_{+}$ and $V(A;\delta_1)_{-}$, respectively. We need to find $r_1$ and $r_2$ and to prove $U(A;\delta_1,r_2)=V(A;\delta_1)_{+}\cup V(A;\delta_1)_{-}$. Without loss of generality, we assume that $d'=d$, namely, $ad-bc=1$.

As a prerequisite, we define a function on a small disk centered at $d$:
\begin{equation*}
f(z)=\begin{dcases}
z^{\frac{1}{D}}=\exp\frac{1}{D}\log z &\text{ if }\Im z\geq0,\\
\etp{1/D}z^{\frac{1}{D}}=\exp\frac{1}{D}\left(\log z+2\uppi\rmi\right) &\text{ if }\Im z<0.
\end{dcases}
\end{equation*}
Then $f(z)$ is continuous on small disks not containing $0$. (Indeed, it is holomorphic.)

Set $l=\min_{0\leq i\neq j<D}\abs{\etp{i/D}-\etp{j/D}}$, and set $r_2=\frac{l}{2+l}\abs{d^{\frac{1}{D}}}$. By continuity of $f$ and boundedness of $\overline{B}$, we can find a $r_1>0$ with the property that for all $2\times2$ real matrices $\gamma_1=\tbtmat{a_1}{b_1}{c_1}{d_1}$ (not necessarily in $\glpR$)
\begin{equation}
\label{eq:r1withproperty}
\lVert\gamma_1-\gamma\rVert<r_1,\,c_1\geq0\Longrightarrow d_1\neq0,\,\abs{d^{\frac{1}{D}}-(c_1^\prime\tau+d_1^\prime)^{\frac{1}{D}}}<r_2\text{ for all }\tau\in\overline{B},
\end{equation}
and
\begin{equation}
\label{eq:r1withpropertyminus}
\lVert\gamma_1-\gamma\rVert<r_1,\,c_1<0\Longrightarrow d_1\neq0,\,\abs{d^{\frac{1}{D}}-\etp{1/D}(c_1^\prime\tau+d_1^\prime)^{\frac{1}{D}}}<r_2\text{ for all }\tau\in\overline{B}.
\end{equation}
We shall prove the above $r_1$ and $r_2$ suffice. Let $0<\delta_1<r_1$.

Let $(\gamma_1,\varepsilon_1\cdot\phi_{\gamma_1})\in V(A;\delta_1)_{+}$ be arbitrary. Then $c_1\geq0$, $\lVert\gamma_1-\gamma\rVert<\delta_1$, and $\varepsilon_1=\varepsilon$. This, together with \eqref{eq:r1withproperty}, implies $\lVert\varepsilon_1\cdot\phi_{\gamma_1}-\varepsilon\cdot\phi_{\gamma}\rVert<r_2$. Therefore, $V(A;\delta_1)_{+}\subseteq U(A;\delta_1,r_2)$. Now let $(\gamma_1,\varepsilon_1\cdot\phi_{\gamma_1})\in V(A;\delta_1)_{-}$ be arbitrary. Then $c_1<0$, $\lVert\gamma_1-\gamma\rVert<\delta_1$, and $\varepsilon_1=\varepsilon\cdot\etp{1/D}$. This, together with \eqref{eq:r1withpropertyminus}, implies $\lVert\varepsilon_1\cdot\phi_{\gamma_1}-\varepsilon\cdot\phi_{\gamma}\rVert<r_2$, so $V(A;\delta_1)_{-}\subseteq U(A;\delta_1,r_2)$.

Conversely, let $(\gamma_1,\varepsilon_1\cdot\phi_{\gamma_1})\in U(A;\delta_1,r_2)$ be arbitrary. Then $\lVert\gamma_1-\gamma\rVert<\delta_1$, and $\lVert\varepsilon_1\cdot\phi_{\gamma_1}-\varepsilon\cdot\phi_{\gamma}\rVert<r_2$. Set $\varepsilon_2:=\varepsilon$ if $c_1\geq0$ and $\varepsilon_2:=\etp{1/D}\varepsilon$ if $c_1<0$. The desired conclusion $(\gamma_1,\varepsilon_1\cdot\phi_{\gamma_1})\in V(A;\delta_1)_{+}\cup V(A;\delta_1)_{-}$ will follow if we can prove $\varepsilon_1=\varepsilon_2$. To this end, we note that
\begin{equation}
\label{eq:normInequality}
\abs{\varepsilon_1-\varepsilon_2}\cdot\lVert\phi_{\gamma_1}\rVert=\lVert\varepsilon_1\phi_{\gamma_1}-\varepsilon_2\phi_{\gamma_1}\rVert\leq\lVert\varepsilon_1\phi_{\gamma_1}-\varepsilon\phi_{\gamma}\rVert+\lVert\varepsilon\phi_{\gamma}-\varepsilon_2\phi_{\gamma_1}\rVert.
\end{equation}
We have $\lVert\varepsilon_1\phi_{\gamma_1}-\varepsilon\phi_{\gamma}\rVert<r_2$ by the assumption. By $\lVert\gamma_1-\gamma\rVert<\delta_1$, and by \eqref{eq:r1withproperty} or \eqref{eq:r1withpropertyminus} (according to $c_1\geq0$ or $c_1<0$, respectively), we have $\lVert\varepsilon\phi_{\gamma}-\varepsilon_2\phi_{\gamma_1}\rVert<r_2$. Inserting these into \eqref{eq:normInequality}, we have
\begin{equation}
\label{eq:normInequality2}
\abs{\varepsilon_1-\varepsilon_2}\cdot\lVert\phi_{\gamma_1}\rVert<2r_2.
\end{equation}
Again by \eqref{eq:r1withproperty} or \eqref{eq:r1withpropertyminus}, we have
$$\lVert\phi_{\gamma_1}\rVert=\lVert\varepsilon_2\phi_{\gamma_1}\rVert=\lVert\varepsilon_2\phi_{\gamma_1}-\varepsilon\phi_\gamma+\varepsilon\phi_\gamma\rVert\geq\lVert\phi_\gamma\rVert-\lVert\varepsilon_2\phi_{\gamma_1}-\varepsilon\phi_\gamma\rVert>\abs{d^{\frac{1}{D}}}-r_2.$$
Combining this and \eqref{eq:normInequality2} we find that
\begin{equation}
\label{eq:normInequality3}
\abs{\varepsilon_1-\varepsilon_2}<\frac{2r_2}{\abs{d^{\frac{1}{D}}}-r_2}=l,
\end{equation}
from which and the definition of $l$ we conclude that $\varepsilon_1=\varepsilon_2$. This finishes the whole proof.
\end{proof}

\begin{rema}
\label{rema:openNeighborhood}
(1) Roughly speaking, this lemma tells us that if $c\neq0$, or $c=0$ but $d>0$, then there are open neighborhoods $U$ of $A=(\gamma,\varepsilon)$ whose elements are all of the form $(\gamma_1,\varepsilon)$ with $\gamma_1$ sufficiently close to $\gamma$. On the other hand, if $c=0$ and $d<0$, then every open neighborhood $U$ of $A$ contains elements of both the form $(\gamma_1,\varepsilon)$ (with $c_1\geq0$) and $(\gamma_1,\etp{1/D}\varepsilon)$ (with $c_1<0$). This complexity originates from the discontinuity of the log function across the negative real axis.

(2) We can further require $r_2$ to be arbitrarily small (then $r_1$ depends on $r_2$). For instance, if $c=0$ and $d<0$, we can replace the construction $r_2=\frac{l}{2+l}\abs{d^{\frac{1}{D}}}$ in the above proof with $r_2<\frac{l}{2+l}\abs{d^{\frac{1}{D}}}$. Then the rest of the proof remains unchanged except that \eqref{eq:normInequality3} should be replaced with
$$
\abs{\varepsilon_1-\varepsilon_2}<\frac{2r_2}{\abs{d^{\frac{1}{D}}}-r_2}<l.
$$
\end{rema}

\begin{prop}
\label{prop:LieDcover}
Let $\Dcover{\glpR}{D}$ be equipped with the above topology. Then we have the following conclusions:
\begin{enumerate}
  \item This topology is independent of the choices of $\tau_0$ and $B$.
  \item $\Dcover{\glpR}{D}$ is Hausdorff, second-countable, locally compact, and locally homeomorphic to $\numR^4$.
  \item The natural projection $p_D\colon\Dcover{\glpR}{D}\rightarrow\glpR$, $\elesltRaDs{a}{b}{c}{d}{\varepsilon}\mapsto\tbtmat{a}{b}{c}{d}$ is a topological covering map.
  \item There is a unique $C^\infty$ manifold structure on $\Dcover{\glpR}{D}$ such that $p_D$ is a $C^\infty$ covering map. Moreover, this structure makes $\Dcover{\glpR}{D}$ a $C^\infty$ Lie group, and $p_D$ a Lie group homomorphism.
\end{enumerate}
\end{prop}
\begin{proof}
For (a), let $\mathscr{T}_{\tau_0,B}$ denote the topology we just defined, and let $U_{\tau_0,B}(A;\delta_1,\delta_2)=U(A;\delta_1,\delta_2)$. Let $(\tau_0',B')$ be another pair. By Lemma \ref{lemm:neighborhood} and Remark \ref{rema:openNeighborhood}(2), for all $A$ and all $\delta_1,\delta_2>0$, there exist $\delta_1',\delta_2'>0$ such that $U_{\tau_0',B'}(A;\delta_1',\delta_2')\subseteq U_{\tau_0,B}(A;\delta_1,\delta_2)$. It follows that $\mathscr{T}_{\tau_0,B}\subseteq\mathscr{T}_{\tau_0',B'}$. The reverse inclusion is proved similarly.

The part (b) is an immediate consequence of Lemma \ref{lemm:neighborhood} and Remark \ref{rema:openNeighborhood}(2).

For (c), let $\gamma=\tbtmat{a}{b}{c}{d}\in\glpR$ and let $\varepsilon\in \Omega_D$, where $\Omega_D\subseteq\numC$ is the set of $D$th roots of unity. We define $V((\gamma,\varepsilon);\delta)$ to be the right-hand side of \eqref{eq:smallneighborhood2} if $c=0$ and $d<0$, and that of \eqref{eq:smallneighborhood1} otherwise. By Lemma \ref{lemm:neighborhood}, we can find $r_1,r_2>0$ such that $U((\gamma,\varepsilon);\delta_1,r_2)=V((\gamma,\varepsilon);\delta_1)$ for all $0<\delta_1<r_1$ and all $\varepsilon$. Set $U(\gamma;\delta_1)=\{\gamma_1\in\glR\colon\lVert\gamma_1-\gamma\rVert<\delta_1\}$, which is an open neighborhood of $\gamma$ in $\glR$. Then
\begin{equation}
\label{eq:covering}
p_D^{-1}(U(\gamma;\delta_1))=\bigcup_{\varepsilon\in\Omega_D}V((\gamma,\varepsilon);\delta_1)=\bigcup_{\varepsilon\in\Omega_D}U((\gamma,\varepsilon);\delta_1,r_2),
\end{equation}
which is a disjoint union of open sets. It follows that $p_D$ is continuous. By Lemma \ref{lemm:neighborhood} and Remark \ref{rema:openNeighborhood}(2), $p_D\vert_{U((\gamma,\varepsilon);\delta_1,r_2)}$ is a homeomorphism from $U((\gamma,\varepsilon);\delta_1,r_2)$ onto $U(\gamma;\delta_1)$. Since $p_D$ is obviously surjective, we conclude that $p_D$ is a topological covering map.

For (d), note that $\glpR$ is a connected $C^\infty$ Lie group of dimension $4$ (the identity component of $\glR$). The existence and uniqueness of the $C^\infty$ manifold structure follow from this and \cite[Proposition 4.40]{Lee13}. To show that $\Dcover{\glpR}{D}$ is a Lie group, it remains to show the multiplication map $m_D\colon\Dcover{\glpR}{D}\times\Dcover{\glpR}{D}\rightarrow\Dcover{\glpR}{D}$ and the inversion map $i_D\colon\Dcover{\glpR}{D}\rightarrow\Dcover{\glpR}{D}$ are $C^\infty$. We show the former, since the latter is similar. Let $A_1=(\gamma_1,\varepsilon_1)$ and $A_2=(\gamma_2,\varepsilon_2)$ be elements of $\Dcover{\glpR}{D}$ and set $A_3=A_1A_2=(\gamma_1\gamma_2,\varepsilon_3)$. We choose a $\delta_3>0$ such that $V(A_3;\delta_3)$  is an open neighborhood of $A_3$ (due to Lemma \ref{lemm:neighborhood}) and that \eqref{eq:covering} holds for $\gamma=\gamma_1\gamma_2$ and $\delta_1=\delta_3$. The multiplication map $m\colon\glpR\times\glpR\rightarrow\glpR$ is $C^\infty$. Therefore, there exist open sets $U(\gamma_1;\delta_1)$ and $U(\gamma_2;\delta_2)$ such that $m(U(\gamma_1;\delta_1)\times U(\gamma_2;\delta_2))\subseteq U(\gamma_1\gamma_2;\delta_3)$, that $U(\gamma_j;\delta_j)$ is evenly covered by $p_D$ (see \cite[p. 278]{Lee11}), and that \eqref{eq:covering} holds for $U(\gamma_j;\delta_j)$, $j=1,2$. It follows that $V(A_j;\delta_j)$ is an open neighborhood of $A_j$ for $j=1,2$. If we can prove for sufficiently small $\delta_1$ and $\delta_2$,
\begin{equation}
\label{eq:mDinclusion}
m_D(V(A_1;\delta_1)\times V(A_2;\delta_2))\subseteq V(A_3;\delta_3)
\end{equation}
(which is nontrivial if one looks merely at the definition of $V(A_j;\delta_j)$), then by the construction of the $C^\infty$ manifold structure on $\Dcover{\glpR}{D}$, we can conclude that $m_D$ is $C^\infty$ on $V(A_1;\delta_1)\times V(A_2;\delta_2)$, and hence on the whole $\Dcover{\glpR}{D}\times\Dcover{\glpR}{D}$ by the arbitrariness of $A_1$ and $A_2$.

To find $\delta_1$ and $\delta_2$ such that \eqref{eq:mDinclusion} holds, we return to the definition \eqref{eq:defUA}. By construction, $V(A_j;\delta_j)=U(A_j;\delta_j,r)$ for $j=1,2$ (we have chosen a common $r>0$) and $V(A_3;\delta_3)=U(A_3;\delta_3,r_3)$. Let $(A_1',A_2')=((\gamma_1',\varepsilon_1'),(\gamma_2',\varepsilon_2'))\in V(A_1;\delta_1)\times V(A_2;\delta_2)=U(A_1;\delta_1,r)\times U(A_2;\delta_2,r)$. By \eqref{eq:multplicationCover}, we have
\begin{align*}
A_1'A_2'&= (\gamma_1'\gamma_2',\varepsilon_1'\varepsilon_2'\phi_{\gamma_1'}(\gamma_2'\tau)\phi_{\gamma_2'}(\tau)),\\
A_3=A_1A_2  &= (\gamma_1\gamma_2,\varepsilon_1\varepsilon_2\phi_{\gamma_1}(\gamma_2\tau)\phi_{\gamma_2}(\tau)),
\end{align*}
where $\phi_\gamma(\tau)=(c^\prime\tau+d^\prime)^{\frac{1}{D}}$ and $\gamma\tau=\frac{a\tau+b}{c\tau+d}$ if $\gamma=\tbtmat{a}{b}{c}{d}$. Note that
\begin{multline}
\label{eq:normMultSplit}
\lVert\varepsilon_1'\varepsilon_2'\phi_{\gamma_1'}(\gamma_2'\tau)\phi_{\gamma_2'}(\tau)-\varepsilon_1\varepsilon_2\phi_{\gamma_1}(\gamma_2\tau)\phi_{\gamma_2}(\tau)\rVert
\leq\lVert\varepsilon_1'\phi_{\gamma_1'}(\gamma_2'\tau)\rVert\cdot\lVert\varepsilon_2'\phi_{\gamma_2'}(\tau)-\varepsilon_2\phi_{\gamma_2}(\tau)\rVert\\
+\lVert\varepsilon_2\phi_{\gamma_2}(\tau)\rVert\cdot\left(\lVert\varepsilon_1'\phi_{\gamma_1'}(\gamma_2'\tau)-\varepsilon_1'\phi_{\gamma_1'}(\gamma_2\tau)\rVert+\lVert\varepsilon_1'\phi_{\gamma_1'}(\gamma_2\tau)-\varepsilon_1\phi_{\gamma_1}(\gamma_2\tau)\rVert\right).
\end{multline}
Since $A_2'\in V(A_2;\delta_2)=U(A_2;\delta_2,r)$ we have $\lVert\varepsilon_2'\phi_{\gamma_2'}(\tau)-\varepsilon_2\phi_{\gamma_2}(\tau)\rVert<r$, and $\lVert\gamma_2'-\gamma_2\rVert<\delta_2$. Since $\gamma_2\overline{B}$ is a compact set in $\uhp$, shrinking $\delta_1$ if necessary we obtain $\lVert\varepsilon_1'\phi_{\gamma_1'}(\gamma_2\tau)-\varepsilon_1\phi_{\gamma_1}(\gamma_2\tau)\rVert<r$. By Lagrange's Mean Value Theorem, we have $\lVert\varepsilon_1'\phi_{\gamma_1'}(\gamma_2'\tau)-\varepsilon_1'\phi_{\gamma_1'}(\gamma_2\tau)\rVert<r$ if $\delta_1$ and $\delta_2$ are sufficiently small. Inserting these estimates into \eqref{eq:normMultSplit} we find that
\begin{equation*}
\lVert\varepsilon_1'\varepsilon_2'\phi_{\gamma_1'}(\gamma_2'\tau)\phi_{\gamma_2'}(\tau)-\varepsilon_1\varepsilon_2\phi_{\gamma_1}(\gamma_2\tau)\phi_{\gamma_2}(\tau)\rVert< r\cdot\lVert\varepsilon_1'\phi_{\gamma_1'}(\gamma_2'\tau)\rVert+2r\cdot\lVert\varepsilon_2\phi_{\gamma_2}(\tau)\rVert
\end{equation*}
for sufficiently small $\delta_1$ and $\delta_2$ and for $(A_1',A_2')\in V(A_1;\delta_1)\times V(A_2;\delta_2)=U(A_1;\delta_1,r)\times U(A_2;\delta_2,r)$. The function $(\gamma_1',\gamma_2',\tau)\mapsto\phi_{\gamma_1'}(\gamma_2'\tau)$ is bounded since $(\gamma_1',\gamma_2',\tau)$ lies in a compact set in $\numR^4\times\numR^4\times\overline{B}$. Therefore,
\begin{equation*}
\lVert\varepsilon_1'\varepsilon_2'\phi_{\gamma_1'}(\gamma_2'\tau)\phi_{\gamma_2'}(\tau)-\varepsilon_1\varepsilon_2\phi_{\gamma_1}(\gamma_2\tau)\phi_{\gamma_2}(\tau)\rVert<rC,
\end{equation*}
where $C$ is a positive constant independent of $(A_1',A_2')$. By Remark \ref{rema:openNeighborhood}(2), we can shrink $r$ so that $rC<r_3$. It follows from this, $m(U(\gamma_1;\delta_1)\times U(\gamma_2;\delta_2))\subseteq U(\gamma_1\gamma_2;\delta_3)$, and $V(A_3;\delta_3)=U(A_3;\delta_3,r_3)$ that $m_D(A_1',A_2')=A_1'A_2'\in V(A_3;\delta_3)$. Now \eqref{eq:mDinclusion} follows since $A_1'$ and $A_2'$ are arbitrary.

Finally, the assertion that $p_D$ is a Lie group homomorphism is merely a combination of what has been proved above.
\end{proof}

To conclude this appendix, we mention two more observations about $D$-covers. The first is an application in which Zhu \cite{Zhu25} constructed unitary characters and $\numZ$-valued homomorphisms of $\Gamma_0(N)$. The $D$-cover formalism plays a key role in this construction; cf. \cite[Propositions 2.1 and 3.1]{Zhu25}.

The second observation addresses an essential difference between the double cover of $\slZ$ and the $D$-covers when $\gcd(D,6)=1$. This is a new result. To begin with, we introduce some terminology.
\begin{deff}
Let $D$ be a positive integer and let $G$ be a subgroup of $\Dcover{\glpR}{D}$. Let $n\in\numgeq{Z}{1}$. We say $G$ is \emph{$n$-fold} if for every $(\gamma_0,\varepsilon_0)\in G$, we have
\begin{equation*}
\#\{(\gamma,\varepsilon)\in G\colon\gamma=\gamma_0\}=n.
\end{equation*}
Let $H$ be a subgroup of $G$. If $p_D(H)=p_D(G)$ (see Proposition \ref{prop:LieDcover}(c) for $p_D$), then we say $H$ and $G$ are \emph{base-equivalent}.
\end{deff}
\begin{prop}
Let $D$ be a positive integer and let $G$ be a subgroup of $\Dcover{\glpR}{D}$. Then there exists a unique $n>0$ such that $G$ is $n$-fold. Moreover, we have $n\mid D$.
\end{prop}
\begin{proof}
The uniqueness of $n$ is immediate. To show the existence, let $R=\{(I,\varepsilon)\in G\}$, and $n=\#R$. We now prove $G$ is $n$-fold. Equivalently, we need to prove that for every $\gamma_0\in p_D(G)$, we have $\#\{(\gamma_0,\varepsilon)\in G\}=n$. Let $(\gamma_0,\varepsilon)\in G$; then $(\gamma_0,\varepsilon)\cdot g=(\gamma_0,\varepsilon')\in G$ for all $g\in R$, and hence $\#\{(\gamma_0,\varepsilon)\in G\}\geq n$. On the other hand, if $(\gamma_0,\varepsilon_1), (\gamma_0,\varepsilon_2)\in G$, then $(\gamma_0,\varepsilon_1)^{-1}(\gamma_0,\varepsilon_2)\in R$, and hence $\#\{(\gamma_0,\varepsilon)\in G\}\leq\#((\gamma_0,\varepsilon_1)\cdot R)=n$. Therefore, $\#\{(\gamma_0,\varepsilon)\in G\}=n$. Finally, since $R$ is a subgroup of $\{(I,\varepsilon)\colon\varepsilon^D=1\}$, we have $n\mid D$.
\end{proof}

If $G$ is a subgroup of $\glpR$, then a $1$-fold base-equivalent subgroup of $\Dcover{G}{D}$ is canonically isomorphic to $G$. Namely, $G$ can be realized as a subgroup of $\Dcover{G}{D}$. Geometrically speaking, a $1$-fold base-equivalent subgroup of $\Dcover{G}{D}$ is a section of the covering map $p_D\colon\Dcover{G}{D}\rightarrow G$ that is also a subgroup. The following proposition determines for which $D$ the group $\Dcover{\slZ}{D}$ has a $1$-fold base-equivalent subgroup.
\begin{prop}
\label{prop:basisequisubgroup}
Let $D=2^\alpha3^\beta D_0$ where $\alpha,\beta\in\numgeq{Z}{0}$ and $D_0\geq1$ satisfies $\gcd(D_0,6)=1$. Then $\Dcover{\slZ}{D}$ has a $2^\alpha3^\beta$-fold base-equivalent subgroup. In particular, if $\gcd(D,6)=1$, then $\Dcover{\slZ}{D}$ has a subgroup canonically isomorphic to $\slZ$.
\end{prop}

The proof rests on the following general lemma.
\begin{lemm}
\label{lemm:condition1foldbe}
Let $D\in\numgeq{Z}{1}$, $k=\frac{t}{D}$ where $t\in\numZ$ is coprime to $D$. Let $G$ be a subgroup of $\slR$. If there exists a linear character $\chi\colon\Dcover{G}{D}\rightarrow\numC^\times$ satisfying
\begin{itemize}
  \item for all $\gamma\in\Dcover{G}{D}$ we have $\chi(\gamma)^D=1$,
  \item $\chi\left(I,\etp{1/D}\right)=\etp{-k}$,
\end{itemize}
then $\Dcover{G}{D}$ has a $1$-fold base-equivalent subgroup.
\end{lemm}
\begin{proof}
Let $R$ be a set of generators of $G$ (e.g., $R=G$). For any $\gamma\in R$, let $m_\gamma\in\numZ$ satisfy $\chi(\widetilde{\gamma})\etp{-km_\gamma}=1$. Such an $m_\gamma$ exists due to the first assumption on $\chi$ and $\gcd(D,t)=1$. Set
\begin{equation*}
R'=\{\widetilde{\gamma}\cdot(I,\etp{m_\gamma/D})\colon\gamma\in R\}.
\end{equation*}
Then for all $g\in R'$, we have
\begin{equation}
\label{eq:chigis1}
\chi(g)=\chi(\widetilde{\gamma})\chi(I,\etp{m_\gamma/D})=\chi(\widetilde{\gamma})\etp{-km_\gamma}=1,
\end{equation}
where we have used the second assumption on $\chi$ in the second equality. Let $H$ be the subgroup of $\Dcover{G}{D}$ generated by $R'$. We claim that $H$ is a $1$-fold base-equivalent subgroup. The groups $H$ and $\Dcover{G}{D}$ are obviously base-equivalent. To see $H$ is $1$-fold, let $(\gamma,\varepsilon_1)$ and $(\gamma,\varepsilon_2)$ be arbitrary elements of $H$ whose first components are the same. Then they are both products of some $g$ or $g^{-1}$ where $g\in R'$, so by \eqref{eq:chigis1} $\chi(\gamma,\varepsilon_1)=\chi(\gamma,\varepsilon_2)=1$. It follows that
\begin{equation*}
1=\chi(\gamma,\varepsilon_2)=\chi(\gamma,\varepsilon_1)\chi(I,\varepsilon_2\varepsilon_1^{-1})=\chi(I,\varepsilon_2\varepsilon_1^{-1}).
\end{equation*}
Since $\chi(I,\etp{j/D})=1$ if and only if $D\mid j$ (by the second assumption on $\chi$ and $\gcd(D,t)=1$), we have $\varepsilon_2\varepsilon_1^{-1}=1$, namely, $\varepsilon_1=\varepsilon_2$.
\end{proof}
\begin{rema}
\label{rema:genset1foldbeSubgroup}
The proof actually constructs a generating set of a $1$-fold base-equivalent subgroup, namely $R'=\{(\gamma,\etp{m_\gamma/D})\colon\gamma\in R\}$, where $R$ is an arbitrary generating set of $G$. To find $m_\gamma$ occurring therein, we should first find an integer $t$ and a character $\chi$ subject to the conditions stated in this lemma, and then find an $m_\gamma\in\numZ$ such that $\chi(\widetilde{\gamma})=\etp{tm_\gamma/D}$. For $G=\Gamma_0(N)$, this can usually (not always) be done by using the character of some eta-quotient of cover index $D$.
\end{rema}

To prove Proposition \ref{prop:basisequisubgroup}, we need to know the group of linear characters of $\Dcover{\slZ}{D}$. Recall that $T=\tbtmat{1}{1}{0}{1}$, $S=\tbtmat{0}{-1}{1}{0}$, and $I=\tbtmat{1}{0}{0}{1}$.
\begin{lemm}{\cite[Lemma 5.2]{Zhu23}}
\label{lemm:presentationslZ}
The group $\Dcover{\slZ}{D}$ has a presentation with generators $\widetilde{T}$, $\widetilde{S}$ and relations $\widetilde{S}^{4D}=\widetilde{I}$, $(\widetilde{S}\widetilde{T})^3=\widetilde{S}^2$, $\widetilde{S}^4\widetilde{T}=\widetilde{T}\widetilde{S}^4$. As a consequence, there are exactly $12D$ linear characters on $\Dcover{\slZ}{D}$ which form a cyclic group.
\end{lemm}
\begin{proof}
See \cite[Lemma 5.2]{Zhu23}.
\end{proof}
The linear characters of $\Dcover{\slZ}{D}$ are exactly the followings:
\begin{equation*}
\widetilde{T}\mapsto \etp{\frac{i}{12D}},\qquad \widetilde{S}\mapsto \etp{-\frac{i}{4D}},
\end{equation*}
where $i$ ranges over $1,2,\dots,12D$.

\begin{proof}[Proof of Proposition \ref{prop:basisequisubgroup}]
First assume that $\gcd(D,6)=1$. We need to prove $\Dcover{\slZ}{D}$ has a $1$-fold base-equivalent subgroup. Let $t=12$, $k=\frac{12}{D}$, and let $\chi\colon\Dcover{\slZ}{D}\rightarrow\numC^\times$ be the character given by $\chi(\widetilde{T})=\etp{1/D}$ and $\chi(\widetilde{S})=\etp{-3/D}$. Then all the conditions of Lemma \ref{lemm:condition1foldbe} (with $G=\slZ$) are fulfilled. (To see why $\chi(I,\etp{1/D})=\etp{-k}$, we use the fact that $\widetilde{S}^4=(I,\etp{1/D})$.) Therefore, Lemma \ref{lemm:condition1foldbe} implies that $\Dcover{\slZ}{D}$ has a $1$-fold base-equivalent subgroup.

Then assume that $D$ is general. Let $p_{D,D_0}\colon\Dcover{\slZ}{D}\rightarrow\Dcover{\slZ}{D_0}$ be the projection introduced in Remark \ref{rema:mulSysProp}. Then $p_{D,D_0}$ is a surjective group homomorphism, and $\ker p_{D,D_0}\cong\numZ/2^\alpha3^\beta\numZ$. Since $\gcd(D_0,6)=1$, there is a $1$-fold base-equivalent subgroup $H$ of $\Dcover{\slZ}{D_0}$. Then $p_{D,D_0}^{-1}(H)$ is a $2^\alpha3^\beta$-fold base-equivalent subgroup of $\Dcover{\slZ}{D}$.
\end{proof}

Actually, a $2^\alpha3^\beta$-fold base-equivalent subgroup of $\Dcover{\slZ}{D}$ has been explicitly constructed in the above proof. For example, suppose that $\gcd(D,6)=1$. Let $\chi$ be the character used in the above proof. Let $m_T,m_S$ be integers satisfying $12m_T\equiv1\bmod{D}$ and $4m_S\equiv-1\bmod{D}$, then we have
\begin{equation*}
\chi(\widetilde{T})=\etp{1/D}=\etp{12m_T/D},\qquad\chi(\widetilde{S})=\etp{-3/D}=\etp{12m_S/D}.
\end{equation*}
It then follows from Remark \ref{rema:genset1foldbeSubgroup} that the group $\slZ^\bullet$ generated by $(T,\etp{m_T/D})$ and $(S,\etp{m_S/D})$ is a $1$-fold base-equivalent subgroup of $\Dcover{\slZ}{D}$. In other words, $p_D$ restricts to a group isomorphism from $\slZ^\bullet$ onto $\slZ$.

Indeed, there is a more explicit construction:
\begin{thm}
\label{thm:slZbullet}
Let $D$ be a positive integer with $\gcd(D,6)=1$. Let $m\in\numZ$ satisfy $12m\equiv1\bmod{D}$. Then
\begin{equation*}
\slZ^\bullet:=\left\{\left(\gamma,\etp{m\Psi(\gamma)/D}\right)\colon\gamma\in\slZ\right\}
\end{equation*}
is a $1$-fold base-equivalent subgroup of $\Dcover{\slZ}{D}$, where $\Psi$ is defined in \eqref{eq:Psi}.
\end{thm}
\begin{proof}
Setting in \eqref{eq:charEtaQuotient} $N=1$ and $r_1=24/D$ we obtain a character $\chi_1\colon\Dcover{\slZ}{D}\rightarrow\numC^\times$, $(\gamma,\varepsilon)\mapsto\varepsilon^{-12}\etp{\frac{1}{D}\Psi(\gamma)}$. Then $\chi=\chi_1$ and $t=12$ satisfy the conditions of Lemma \ref{lemm:condition1foldbe} (with $G=\slZ$). Therefore, by Remark \ref{rema:genset1foldbeSubgroup} (with $R=\slZ$), $R'=\{(\gamma,\etp{m_\gamma/D})\colon\gamma\in \slZ\}$ generates a $1$-fold base-equivalent subgroup of $\Dcover{\slZ}{D}$, where $m_\gamma$ should satisfy $\chi_1(\widetilde{\gamma})=\etp{12m_\gamma/D}$. Indeed, $m_\gamma=m\Psi(\gamma)$ suffices as $12m\equiv1\bmod{D}$. Moreover, it is immediate that the group generated by $R'$ is $R'$ itself since it is $1$-fold.
\end{proof}
\begin{rema}
In fact, the character $\chi_1$ above is the same as the character $\chi$ constructed in the proof of Proposition \ref{prop:basisequisubgroup}.
\end{rema}

Finally, we show by an example that when $\gcd(D,6)>1$, $\Dcover{\slZ}{D}$ may have no $1$-fold base-equivalent subgroup:
\begin{prop}
Let $\alpha$ be a positive integer. Then the only base-equivalent subgroup of $\Dcover{\slZ}{2^\alpha}$ is this group itself.
\end{prop}
\begin{proof}
Let $G$ be an arbitrary base-equivalent subgroup of $\Dcover{\slZ}{2^\alpha}$. Then there is a $t\in\numZ$ such that $(S,\etp{t/2^\alpha})\in G$. Then $(I,\etp{(4t+1)/2^\alpha})=(S,\etp{t/2^\alpha})^4\in G$. Let $a,b\in\numZ$ satisfy $a(4t+1)+b2^\alpha=1$. Then $(I,\etp{1/2^\alpha})=(I,\etp{(a(4t+1)+b2^\alpha)/2^\alpha})=(I,\etp{(4t+1)/2^\alpha})^a\in G$. It follows that $\widetilde{S}=(S,\etp{t/2^\alpha})\cdot (I,\etp{1/2^\alpha})^{-t}\in G$. Similarly, $\widetilde{T}\in G$. By Lemma \ref{lemm:presentationslZ}, $\widetilde{T}$ and $\widetilde{S}$ generate $\Dcover{\slZ}{2^\alpha}$, and hence $G=\Dcover{\slZ}{2^\alpha}$.
\end{proof}

\section{Usage of SageMath code}
\label{apx:Usage of SageMath code}
The SageMath code for computing eta-quotients, $D$-covers, dimensions, and other related quantities, as well as for verifying the theorems and generating the tables can be found at \url{http://xjzmath.com/code/SageCodeDimensionRationalWeight.html}.
\end{appendix}

\section*{Acknowledgement}
We thank the anonymous referee for their constructive comments, which have greatly improved this paper.

\bibliographystyle{amsalpha}
\bibliography{ref}
\end{document}